\documentclass[12pt]{amsart}
\usepackage{amsmath,amssymb,latexsym,amscd}
\usepackage{fullpage,hyperref}
\theoremstyle{definition}

\newtheorem{lemma}[equation]{Lemma}
\newtheorem{definition}[equation]{Definition}
\newtheorem{conjecture}[equation]{Conjecture}

\newtheorem{cor}[equation]{Corollary}

\newtheorem{prop}[equation]{Proposition}
\newtheorem{lem}[equation]{Lemma}

\renewcommand{\O}{\mathcal{O}}
\newcommand{\Fb}{\overline{F}}

\newcommand{\ff}[4]{\overset{\alpha_1}{#1}----\overset{\alpha_2}{#2} ==>==\overset{\alpha_3}{#3}----\overset{\alpha_4}{#4}}

\newcommand{\x}[2]{x_{#2}({#1}_{#2})}
\newcommand{\il}{\int\limits}
\newcommand{\ul}{\underline}
\newcommand{\gm}{\gamma}
\newcommand{\wt}{\widetilde}
\newcommand{\on}{\operatorname}
\newcommand{\FC}[1]{^{(#1,\psi_{#1})}}

\newcommand{\bpm}{\begin{pmatrix}}
\newcommand{\epm}{\end{pmatrix}}
\newcommand{\bsm}{\begin{smallmatrix}}
\newcommand{\esm}{\end{smallmatrix}}
\newcommand{\bspm}{\left(\begin{smallmatrix}}
\newcommand{\espm}{\end{smallmatrix}\right)}
\newcommand{\bm}{\begin{matrix}}
\renewcommand{\em}{\end{matrix}}

\newcommand{\bd}{\left|\begin{matrix}}
\newcommand{\ed}{\end{matrix}\right|}

\newcommand{\bbm}{\begin{bmatrix}}
\newcommand{\ebm}{\end{bmatrix}}

\newcommand{\fh}{\mathfrak{h}}

\newcommand{\fsl}{\mathfrak{sl}}

\newcommand{\cH}{\mathcal{H}}
\newcommand{\iFA}{\int\limits_{F\bs \A}}
\newcommand{\iiFA}[1]{\int\limits_{(F\bs\A)^{#1}}}

\newcommand{\bs}{\backslash}
\newcommand{\quo}[1]{#1(F)\bs #1(\A)}
\newcommand{\iq}[1]{\int\limits_{\quo{#1}}}
\newcommand{\C}{\mathbb{C}}
\newcommand{\A}{\mathbb{A}}
\newcommand{\G}{\mathbb{G}}
\newcommand{\Q}{\mathbb{Q}}

\newcommand{\R}{\mathbb{R}}

\newcommand{\Tr}{\operatorname{Tr}}
\newcommand{\rank}{\operatorname{rank}}
\newcommand{\Stab}{\operatorname{Stab}}

\newcommand{\Gal}{\operatorname{Gal}}
\newcommand{\diag}{\operatorname{diag}}
\newcommand{\Spin}{\operatorname{Spin}}
\newcommand{\Lie}{\operatorname{Lie}}
\newcommand{\GSpin}{\operatorname{GSpin}}

\newcommand{\Sp}{\widetilde{Sp}}
\newcommand{\pr}{\operatorname{pr}}
\newcommand{\Ind}{\operatorname{Ind}}

\newcommand{\Mat}{\operatorname{Mat}}

\newcommand{\sym}{\operatorname{sym}}

\newcommand{\Res}{\operatorname{Res}}

\newcommand{\st}{\text{ s.t. }}

\newcommand{\bC}{\wt{C}(\A)}
\newcommand{\bG}{\wt F_4(\A)}
\newcommand{\vph}{\varphi}
\newcommand{\la}{\langle}
\newcommand{\ra}{\rangle}

\newcommand{\ad}{{\on{ad}}}

\begin{document}
\date{}
\title[ Global Constructions]
{\bf Constructions of Global Integrals in the Exceptional Groups }
\author{ David Ginzburg and Joseph Hundley}
\thanks{The first named author is partly supported by a grant from the Israel
Science Foundation no. 2010/10
The second named author was supported  by NSA
Grant MSPF-08Y-172
and NSF Grant DMS-1001792 during the period this work was completed.
}

\address{ School of Mathematical Sciences\\
Sackler Faculty of Exact Sciences\\ Tel-Aviv University, Israel
69978 } \maketitle \baselineskip=18pt
\begin{abstract}
Motivated by known examples of global integrals 
which represent automorphic $L$-functions, this paper
initiates the study of a certain two-dimensional array
of global integrals attached to any reductive 
algebraic group, indexed by maximal parabolic subgroups 
in one direction and by unipotent conjugacy classes 
in the other.  Fourier coefficients attached to unipotent classes, Gelfand-Kirillov dimension of automorphic representations, and an identity which, empirically, appears
to constrain the unfolding process are presented in detail
with examples selected from the exceptional groups.  Two 
new Eulerian integrals are included among these examples.\\Keywords:  Integral representations,  Rankin-Selberg, exceptional groups.\\
{2010 Mathematics Subject Classification.} Primary: 32N10, Secondary: 11F70, 11F30, 
\end{abstract}
\tableofcontents
\section{\bf  Introduction }

One of the important tools in the study of analytic properties of
automorphic $L$ functions is the  method of integral
representations. In this method one seeks to write down a global
integral which depends on a complex variable $s$, show that this
integral converges for $\text{Re}(s)$ large, and has a continuation
to the whole complex plane. Then the goal is to show that the
integral is Eulerian, i.e. that it has an Euler product. In the
literature, when the global integral contains an integration over a
reductive group, this method is known as the Rankin-Selberg method.
In the case when the integral is the Whittaker coefficient of an
Eisenstein series, this method is known as the Langlands-Shahidi
method. From the perspective which is adopted in this work, this division 
is not particularly natural, and hence we prefer to refer simply to 
``the  method of integral representations.''

It is an easy task to write down an integral which has an analytic
continuation. It is also easy to write global integrals which are
Eulerian.  However, to find integrals which satisfy both conditions
is quite hard. 
We are still unaware of a general method of obtaining
all integrals which satisfy both conditions.

In this paper we use ideas which were introduced in \cite{G1} and
\cite{G3} in order to start a certain systematic approach for
constructing examples of global integrals. We do it in the
exceptional groups. There are two main reasons for that. First, the
exceptional groups are not a part of an infinite family of
groups. This enables us to check a finite list of cases. Second, the
exceptional groups are less studied from these aspects than the
classical groups. This gives the hope of possibly finding some
interesting constructions. However, the methods introduced here are
also applicable for the classical groups, and we do hope to study
this issue in the future.

There are two main ingredients in our approach. The first is to use
the classification of unipotent orbits
(i.e., conjugacy classes consisting of unipotent elements)
in order to write down global
integrals. The main idea is as follows. Given a unipotent orbit
$\mathcal{O}$ in an exceptional group $H$, we attach to
it a  parabolic subgroup $P_\O = M_\O U_{\mathcal{O}}$ and a character
$\psi_{U_{\mathcal{O}}}$ defined on $U_{\mathcal{O}}(F)\backslash
U_{\mathcal{O}}({\A })$ (or, in some cases, a subgroup).
The stabilizer of $\psi_{U_\O}$ in $M_\O$ makes sense as an
algebraic group defined over $F,$ and it follows from the results
in  \cite{C}, that the identity component of this
stabilizer  is a reductive group $C$. Let
$E_\tau(h,s)$ denote an Eisenstein series defined on $H({\A}),$
associated with the induced representation $\Ind_{P({\A })}^{H({\A})}\tau\delta_P^s$. See section 4 for details. Let $\pi$ denote an
irreducible cuspidal representation of $C({\A })$. Then, under
certain assumptions on the unipotent orbit $\mathcal{O}$, we can
form the global integral
\begin{equation}\label{intro1}
\int\limits_{C(F)\backslash C(\A )} \int\limits_{U_{\mathcal
{O}}(F)\backslash U_{\mathcal {O}}(\A )}
\varphi_\pi(g)E_\tau(ug,s)\psi_{U_{\mathcal {O}}}(u)\,du\,dg.
\end{equation}
See sections 3 and 4 for details. Because of the cuspidality of
$\pi$, this integral converges for $\text{Re}(s)$ large and has at
least a meromorphic continuation to the whole complex plane. Thus it
clearly satisfies the first requirement we imposed above on our
integrals. Notice that  for  certain unipotent
orbits, the group $C$ is trivial. In this case the above
integral is just the Fourier coefficient of the Eisenstein series.
When $U_{\mathcal {O}}$ is the maximal unipotent subgroup of $H$,
then $\psi_{U_{\mathcal {O}}}$ is a generic character, and our
construction specializes to the Langlands-Shahidi construction.

The next step is to address the problem of when an integral of the
form given in equation \eqref{intro1} is Eulerian. This clearly
depends on the representations $\pi$ and $\tau$. Assume that $\pi$
is a generic representation.
A sufficient condition for the integral \eqref{intro1} to be Eulerian
is that it is equal to an integral which involves the
Whittaker function attached to $\varphi_\pi,$ and some function
in a model of $\tau$ with similar factorization properties, integrated
over a factorizable domain.  We refer to such an
integral as a ``Whittaker integral.''
Experience shows that if
an integral of the type \eqref{intro1} unfolds to Whittaker integral, and is Eulerian, then a certain dimension equality holds. See
\cite{G3} for some more details. Therefore, we restrict to those
integrals which satisfy this equation.  This is all discussed in
detail in section 4. We sketch the main ideas. To each of the
representations $\pi$ and $E_\tau(\cdot,s)$ we attach a number which
we refer to as the Gelfand-Kirillov dimension of the representation.
The identity we require is
$$\text{dim}\ C+\text{dim}\ U_{\mathcal{O}}=\text{dim}\ \pi+\text{dim}\
E_\tau(\cdot,s)$$ This gives us an equation for $\text{dim}\ \tau,$
which then gives us certain information of which representation
$\tau$ to choose. To be clear, as far as we know, every global
Eulerian integral of the type \eqref{intro1},  which unfolds to a
Whittaker coefficient of $\pi$ satisfies the above dimension
identity. However, there are global integrals, which satisfy this
identity but do not unfold to a Whittaker coefficient of $\pi$.
There are many such examples; we obtain some in section \ref{An Example}.
Further, the
question of whether an integral of the type \eqref{intro1} is Eulerian,
and the question of whether it unfolds to a Whittaker coefficient 
of $\pi$ are two related, but separate questions.  An answer to one does
not necessarily determine an answer to the other either way.  

The content of the paper is as follows. After fixing some notations
we show how to attach a set of Fourier coefficients
to a unipotent orbit of an exceptional group. This is done in section 3. In section
4 we construct the global integrals we intend to study and give a
precise description of the basic dimension identity we use. There
are several cases, and we discuss each of them. At the end of section
4 we write down a table of the condition which must be
placed on the  representation $\tau$ in each case, in order for the dimension formula to be  satisfied. We do this for the
exceptional group $F_4$. This condition gives information about the Fourier
coefficients  supported by $\tau.$ It is worthwhile to mention that in four
of the rows in the table, the group $C$ is trivial.  In these cases, the integral
\eqref{intro1} is just a Fourier coefficient of an Eisenstein series.
The integrals
obtained by considering the last row, are the Langlands-Shahidi
integrals for the exceptional group $F_4$.

Section 5 is devoted to the general process of unfolding a global
integral. We carry out this process in detail and as generally as we
can. In this way we also fix some notation which will be of help to us
later on. According to the general outcome, we partition the set of
integrals into various types. Maybe the most interesting
integrals are the ones we refer to as open orbit type
integrals.  See definition \ref{def4}. These are those integrals, such
that after some unfolding, we obtain as inner integration, an
integral of a similar type to the one we started with, and which
satisfies a similar dimension identity. This shows that in some
sense the process is inductive. The study of the global integral is
reduced to the study of a global integral defined on smaller groups.

Starting from section 6 we concentrate on two examples in the group
$F_4$. These examples are typical and represent the general process.
Therefore we study them in detail. In the notations of the table at
the end of section 4, we consider in section 8 the global integral
attached to the unipotent orbit $\mathcal{O}=\widetilde{A}_2$ and in
section 10 the integral attached to the unipotent integral
$\mathcal{O}=A_1$.

As mentioned above, this paper is a starting point towards a certain
systematic study of constructing global integrals, first in the
exceptional groups, and later in classical groups. Clearly there
are many details which need to be addressed. First, in this project we
assume that $\pi$ is generic. However, one can study also the cases
where the representation $\pi$ is cuspidal and not generic.
Secondly, the integrals which we construct here are natural in a
certain sense. But one can also study global integrals, which
satisfy the dimension identity given in section 4, but not of the
type which we construct in this paper. Finally, there is the
interesting question of determining which $L$ functions we obtain
for those integrals which are Eulerian. This is, of course, a local
question, and needs to be addressed by means of local computations.
In the two examples we consider in this paper, we obtain two
examples of such integrals which are new. We hope to study these
topics in the  future.

\section{\bf  Basic notations}
\label{section:   Basic notations}

Let $H$ denote one of the  exceptional groups $F_4$ or $E_i$ for
$i=6,7,8$.
By this we mean the unique split connected simply connected simple algebraic group
of the given type defined over our number field $F.$
We shall label the simple roots of $H$ as follows.
First, for the group $H=F_4$, the Dynkin is given by
$$\overset{\alpha_1}{0}----\overset{\alpha_2}{0} ==>==
\overset{\alpha_3}{0}----\overset{\alpha_4}{0}$$ As for the group
$H=E_8$, we label the simple roots as
\begin{equation}
\setcounter{MaxMatrixCols}{13}
\begin{matrix} \overset{\alpha_1}0&--&\overset{\alpha_3}0&--
&\overset{\alpha_4}0&--&\overset{\alpha_5}0&--&\overset{\alpha_6}0&
--&\overset{\alpha_7}0&--&\overset{\alpha_8}0\\
&&&&|&&&&&&&\\ &&&&\underset{\alpha_2}0&&&&&&&\end{matrix}\notag
\end{equation}
The Dynkin diagram for the groups $H=E_6$ and $H=E_7$, are the
ones derived from the diagram for $E_8$ omitting the relevant
roots.

For construction of several of our integrals, we will need to use
the similitude groups for $E_6$ and $E_7$. We shall denote these
groups by $GE_6$ and $GE_7$, and use the notations defined in
\cite{G2}.

We assume $H$ to be equipped with a
choice of maximal torus $T$ and  Borel subgroup, $B=TU_{\max}^H.$
Here $U_{\max}^H$ is the unipotent radical of $B,$ which is a maximal
unipotent subgroup of $H.$
If $G$ is any $T$-stable subgroup of $H$ we denote the set of roots
of $T$ in $G$ by $\Phi(G,T).$  We also assume $H$ to be equipped with
 a
realization, in the sense of \cite{Springer}, i.e., a choice of isomorphism $x_{\alpha}: \G_a \to U_{\alpha}$
for each root $\alpha\in \Phi(H,T),$ subject to certain compatibility conditions.
Here $U_\alpha$ is the one-dimensional subgroup of
$H$ attached to $\alpha.$
We further assume that the structure constants for this realization
are given as in \cite{Gilkey-Seitz}.

For each root $\alpha,$ the groups $U_\alpha$ and $U_{-\alpha}$
generate a subgroup of $H$ which is isomorphic to $SL_2.$  We
denote this subgroup
by $SL_2^\alpha.$  Likewise, given two roots $\alpha$ and $\beta,$
the
subgroup of $H$ generated by $U_{\pm \alpha}$ and $U_{\pm \beta},$
is denoted by  $SL_3^{\{\alpha, \beta\}}$
if it is isomorphic to $SL_3,$ by $Sp_4^{\{\alpha, \beta\}}$
if it is isomorphic to $Sp_4,$ etc.

We shall denote the maximal parabolic subgroups of $H$ as follows.
Let $P$ be a standard maximal parabolic subgroup of $H$. Let $\alpha_i$
denote the unique simple root of $H$ such that the one parameter
unipotent subgroup $x_{\alpha_i}(r)$ is in the unipotent radical
of $P$. We shall then denote $P$ by $P_i$.  We also denote
the unipotent radical by $U_i$ and the standard Levi subgroup
by $M_i.$

The choice of realization allows us to work with elements of $H$ via
explicit generators and relations.  In this approach we shall refer to
specific roots of $H$ very often.  If $\alpha_1, \dots, \alpha_r$
are the simple roots of $H,$ we identify the root $\sum_{i=1}^r n_i \alpha_i$
with the tuple $(n_1, \dots, n_r).$  Since none of the coefficients $n_i$
ever exceeds $9,$ we shall normally simply run the digits together.
For example, the root $2\alpha_1+3\alpha_2 + 4\alpha_3 + 2\alpha_4$
of $F_4$ will be denoted $2342$ or $(2342).$

\subsubsection{Some conventions for defining characters}  \label{sss:conventions for defining characters}  Suppose that $U$ is a unipotent group and
we need to define a character of $U.$  This will occur very frequently in this work.  We introduce two convenient notational conventions for this, at least in the
case when $U$ is $T$-stable, in which case it is
 a product of groups $U_\alpha.$
Assume that this is so.
 Fix a total ordering
on $\Phi(U,T).$  Then we may define an isomorphism of varieties $\G_a^{\dim U} \to U$ by
$$
(u_\alpha)_{\alpha \in \Phi(U,T)}
\mapsto \prod_{\alpha \in \Phi(U,T)} x_\alpha(u_\alpha),
$$
with the product defined using our fixed total ordering.
This in effect defines ``coordinates'' $(u_\alpha)_{\alpha \in \Phi(U,T)}$ on $U.$
Further, the $\alpha$-coordinate $u_\alpha$ of $u \in U$
depends on the choice of total ordering only if $U_\alpha$ lies in the commutator subgroup of $U.$  In particular, characters may be written using these coordinates without ambiguity, even when the total ordering has not been specified.  For example, one can define a character of the maximal unipotent subgroup of $F_4$ by
$\psi_{U_{\max}}(u) = \psi( u_{0100}+u_{0010}+u_{0001}).$
Note that $u_{1000}, u_{0100},u_{0010}$ and $u_{0001}$
are the only coordinates of $u \in U_{\max}$ which are well defined, independently of the total ordering chosen.  On the other hand, they are also the only coordinates on which a character of $U_{\max}$ may depend.
Finally, note that while these coordinates are independent of the choice of total order, they are very much dependent on the choice of realization $\{ x_\alpha: \alpha \in \Phi(H,T)\}.$  But the realization has been fixed once and for all, so this is no cause for concern.

Sometimes a different convention is more convenient. If we write
$$
\psi_U( x_{\beta_1}(r_1) x_{\beta_2}(r_2) u') = \psi(r_1+r_2),
$$
or the like, the convention is that the restriction of $U$ to $U_\alpha$ is trivial for all $\alpha \in \Phi(U,T)$ except for those listed (here $\beta_1,$ and $\beta_2$).  Also $u'$ is an element of the product of the groups $U_\alpha, \alpha \in \Phi(U,T)$ with the listed roots excluded.

\section{\bf  Construction of Fourier Coefficients}
\label{section: construction of fourier coefficients}

In this section we shall explain how to construct a Fourier coefficient
from a weighted diagram corresponding to a unipotent orbit.
Our construction is essentially a slight refinement of the one given in
\cite{G-R-S1}, and applies to any split connected reductive
algebraic group.
Our main
interest are constructions in the exceptional groups, however this
setup holds for classical groups as well.
Let $H$ denote a split connected reductive
algebraic group.

We shall work in characteristic zero.  Consequently, we may pass freely back and forth between unipotent
orbits in a reductive group and nilpotent orbits in its Lie algebra.  Indeed, the exponential map gives an equivariant isomorphism between the nilpotent and unipotent subvarieties.  We may also pass back and forth between adjoint orbits in the Lie algebra and coadjoint orbits in its dual, as described in \cite{C-M}, section 1.3.

We use the classification results
described in \cite{C}, particularly the tables of orbits which begin on page 401, as well as the Bala-Carter labels used therein.
For the list of dimensions of these
orbits, we use \cite{C-M} page 128.   A well known result of Dynkin which is described in section 5.6 of \cite{C} associates  with any unipotent orbit a Dynkin diagram
whose roots are labeled with zeros, ones and twos,
which determines the orbit completely.
For exceptional groups, the Bala-Carter label or the weighted Dynkin diagram is the most common method of specifying an orbit.  For classical
groups it is more common to use partitions as in \cite{G-R-S1}.  However,
the parametrization via weighted Dynkin diagrams is valid in general and,
indeed, the first part of the construction given in \cite{G-R-S1} amounts
to recovering the weighted Dynkin diagram attached to an orbit
from the partition.

Let  ${\mathcal O}$ denote a unipotent orbit for the group $H$.  (We shall identify each algebraic group which we consider with its $\Fb$ points where $\Fb$ is a fixed algebraic closure of $F.$)
First, we shall associate to ${\mathcal O}$ a parabolic subgroup of
$H$ which we shall denote by $P_{{\mathcal O}}$. We shall write its
Levi decomposition $P_{{\mathcal O}}=M_{{\mathcal O}}U_{{\mathcal O}}$.
Here $M_{{\mathcal O}}$ is the standard Levi subgroup
of $P_{{\mathcal O}}$ and
$U_{{\mathcal O}}$ is its unipotent radical.

The  parabolic subgroup $P_{{\mathcal O}}$ is defined to be the standard
parabolic subgroup of $H$ whose Levi part is generated by the
maximal torus of $H$, and all copies of $SL_2$ corresponding to
those simple roots which are labeled by zero in the diagram. For
example, in the group $H=E_6$ the diagram corresponding to the
unipotent class ${\mathcal O}=2A_2$ is given by
\begin{equation}\label{2a2}
\begin{matrix} \overset{2}0&--&\overset{}0&--
&\overset{}0&--&\overset{}0&--&\overset{2}0\\
&&&&|&&&&\\ &&&&\underset{}0&&&& \end{matrix}
\end{equation}
Here and throughout, roots without a number are considered as labeled with
the number zero. In other words, the roots
$\alpha_2,\alpha_3,\alpha_4$ and $\alpha_5$ are labeled with the
number zero. Thus, in this case $M_{{\mathcal O}}=\Spin_8\cdot
GL_1^2$.

As a second example, consider in $H=F_4$ the unipotent class
${\mathcal O}=B_2$. Its diagram is given by
$$\overset{2}{0}----\overset{}{0} ==>==
\overset{}{0}----\overset{1}{0}$$ Therefore, the Levi part of
$P_{{\mathcal O}}$ is given by $M_{{\mathcal O}}=Sp_4\cdot GL_1^2$.

Since $P_\O$ is standard,  $U_{{\mathcal O}}$  is generated by one dimensional
unipotent subgroups $x_\alpha(r)$,  where $\alpha$ is positive.
For example, in the case of diagram \eqref{2a2}, the unipotent
group $U_{{\mathcal O}}$ is generated by all $x_\alpha(r)$ such that
if we write $\alpha=\sum_{i=1}^6n_i\alpha_i$, then either $n_1>0$ or
$n_6>0$. Thus $\dim U_{{\mathcal O}}=48$.

Let $\psi$ denote a nontrivial character of $F\backslash \A$.
To the unipotent orbit ${\mathcal O}$ we shall associate  a
character $\psi_{U_{{\mathcal O}}}$ defined on a subgroup of
$U_{{\mathcal O}}$.   In the case of the classical group this
procedure is explained in detail in \cite{G1} section two. For
the exceptional groups this is done in a similar way. To describe the
construction, we fix some notations.

We will say that the root $\alpha$ is in $U_{{\mathcal O}}$ if the
one parameter unipotent subgroup $x_\alpha(r)$ is in $U_\O$. Assume that the simple roots
$\alpha_{i_1},\ldots,\alpha_{i_r}$ are the simple roots which are
labeled zero in the diagram corresponding to the unipotent orbit
${\mathcal O}$. By definition, a positive root
$\alpha=\sum_in_i\alpha_i$ is in  $U_{{\mathcal O}}$ if and only if
$n_{i_1}+\ldots +n_{i_r}>0$. For each $k>0$ we set
$U^{(k)}_{\mathcal O}$ equal to the product of the groups $U_\alpha$ for $\alpha$ in the set $\{ \alpha=\sum_in_i\alpha_i\ \ :  \ n_{i_1}+\ldots
+n_{i_r}\ge k\}.$ It is easy to see that $U^{(k)}_{\mathcal O}$ is a
group and that $L_\O^{(k)}=U^{(k)}_{\mathcal O}/U^{(k+1)}_{\mathcal O}$
is an abelian group.
We also define $V_\O^{(k)}$ as the product of the groups
$U_\alpha$ corresponding to those roots
$\alpha= \sum_i n_i \alpha_i$ with $\sum_i c_i n_i \ge k.$
Here,  $c_i \in \{0,1,2\}$ is the
weight attached to the simple root $\alpha_i$ in the weighted Dynkin
diagram of $\O.$
We consider the group $V_\O^{(2)}/ V_\O^{(3)}.$  It may equal $L_\O^{(1)},$
or $L_\O^{(2)},$ or some combination of the two.
The group $M_\O$ acts on this group with an open orbit.
We say that an element of $L_\O^{(2)}$ is ``in general position'' if it is in this orbit.
 See proposition 5.7.3 of
\cite{C} for details. It follows from proposition  5.5.9, p. 156 of \cite{C}
that the identity component of
the stabilizer inside $M_{{\mathcal O}}$ of a representative of the
open orbit is a reductive group.

To define the character $\psi_{U_{{\mathcal O}}}$, assume first that
the diagram attached to ${\mathcal O}$ has only zeros and twos. In
this case, $V_\O^{(2)}/V_\O^{(3)}= L_\O^{(1)},$ and
$\psi_{U_{{\mathcal O}}}$ will be a character of $U_\O$ itself.
We consider the action of $M_{{\mathcal O}}$ on the group
$L_\O^{(1)}$.
This action is essentially a rational representation of $M_\O$ and the $F$-points of the dual rational representation are identified with  the set of all characters
of $U_\O$ by our choice of $\psi.$
As explained above, over an algebraically closed
field, the action of $M_\O$ on $L^{(1)}_\O$ has an open orbit, and the identity component of
the stabilizer of any element of this orbit is reductive.  The same is true of the action of $M_\O$ on the dual rational
representation, which we denote $L^{(1),*}_\O.$  However, the action of $M_\O(F)$ on $L^{(1),*}_\O(F)$ need not have an open orbit.  Indeed, the intersection of $L^{(1),*}_\O(F)$ with the open orbit in $L^{(1)}_\O$ may be a union of several $M_\O(F)$ orbits, with stabilizers which are isomorphic to one another only over the closure and not over $F.$
This issue is discussed in \cite{G1} in
example 2 page 328.
For $\psi_{U_\O}$ we choose a character which is
in the open orbit.
Following the tables given in \cite{C} page 401,
we shall denote the connected component of the stabilizer  by  $C$.

Let $\varphi$ denote an automorphic form defined on the group $H(\A),$ It
follows from the above discussion that the following Fourier coefficient
\begin{equation}\label{int1}
\varphi^{(U_\O, \psi_{U_\O})}(g):=
\int\limits_{U_{{\mathcal O}}(F)\backslash U_{{\mathcal O}}(\A)}
\varphi(ug)\psi_{U_{{\mathcal O}}}(u)du
\qquad( g \in C(\A))
\end{equation}
defines a  left-$C(F)$-invariant function $C(\A) \to \C.$

\begin{lem}
The function $\varphi\FC{U_\O}$ is smooth, and $\varphi\FC{U_\O}$
and all of its derivatives have moderate growth.
\end{lem}
\begin{proof}
We define ``smooth''
and ``moderate growth'' as in \cite{MW}, section I.2.5.
We are given that $\varphi$ is smooth.  Since $\varphi\FC{U_\O}$
is obtained by integrating $\varphi$ over a compact set
against a smooth, bounded function, it is clear that $\varphi\FC{U_\O}$
inherits this property from $\varphi.$

We are also given that $\varphi$ is of moderate growth.  This is defined
using a choice of embedding $H \hookrightarrow GL_N$
for some $N$ (though the condition obtained is independent
of the choice of embedding).  If we then define
``moderate growth for $C$ using the embedding $C \hookrightarrow H \hookrightarrow GL_N,$ then it is clear that $\varphi\FC{U_\O}$ inherits this
property as well.
\end{proof}
We remark that $\vph\FC{U_\O}$ also inherits the $K$-finiteness property
of the automorphic form $\vph,$ provided that the maximal compact
subgroups of $H(F_\infty)$ and $C(F_\infty)$ are chosen
compatibly, but that it will \ul{not}, in general be $\frak z$-finite, and therefore
that it will not, in general, be an automorphic form.
(Here, $\frak z$ indicates the center of the universal enveloping algebra
of $\Lie( C( F_\infty)).$)  We elaborate on this remark slightly in an appendix.

We consider a few examples. First, let $H=E_6$ and ${\mathcal
O}=2A_2$. The roots in $L^{(1)}_\O$ are all those roots
$\alpha=\sum_{i=1}^6n_i\alpha_i$ such that $n_1+n_6=1$. There are 16
such roots. The group $M_{{\mathcal O}}=\Spin_8\cdot GL_1^2$ acts on
this group. The representation is reducible, and up to the action of
$GL_1^2$ it is equal to two of the three eight dimensional
representations of $\Spin_8$. It follows from \cite{C} that the
stabilizer of the open orbit is the exceptional group $G_2$. For
$u\in U_{{\mathcal O}}$ define
$$\psi_{U_{{\mathcal O}}}(u)=\psi_{U_{{\mathcal O}}}(x_{111100}(r_1)x_{101110}(r_2)
x_{010111}(r_3)x_{001111}(r_4)u')=\psi(r_1+r_2+r_3+r_4)$$ where
$u'\in U_{{\mathcal O}}$ is any  product of unipotent elements
$x_\alpha(r)$
 in $U_{{\mathcal O}}$ such that $\alpha$ is not one of the above four roots. Then the stabilizer inside
$M_{{\mathcal O}}$ is the exceptional group $G_2$.

As an another example, consider in $E_8$, the unipotent class ${\mathcal O}=D_4$. Its
diagram is given by
\begin{equation}
\setcounter{MaxMatrixCols}{13}
\begin{matrix} \overset{}0&--&\overset{}0&--
&\overset{}0&--&\overset{}0&--&\overset{}0&--&\overset{2}0&--&\overset{2}0\\
&&&&|&&&&&&&\\ &&&&\underset{}0&&&&&&&\end{matrix}\notag
\end{equation}
In this case $L^{(1)}_\O$ consists of all roots
$\alpha=\sum_{i=1}^8n_i\alpha_i$ such that $n_7+n_8=1$. There are 28
such roots. The action of  $M_{{\mathcal O}}=E_6\cdot GL_1^2$ on the
roots in $L^{(1)}_\O$ gives a sum of two irreducible
representations. First, on $x_{\alpha_8}(r)$ we obtain a one
dimensional representation of $M_{{\mathcal O}}$. The second
representation corresponds to the 27 roots
$\alpha=\sum_{i=1}^8n_i\alpha_i$ such that $n_7=1$ and $n_8=0$, is
the 27 dimensional representation of $E_6$.  If we define the
character
$$\psi_{U_{{\mathcal O}}}(u)=\psi_{U_{{\mathcal O}}}(x_{00000001}(r_1)x_{11221110}(r_2)
x_{11122110}(r_3)x_{01122210}(r_4)u')=\psi(r_1+r_2+r_3+r_4)$$
then the stabilizer in $M_{{\mathcal O}}$ is the exceptional group $F_4$.

As a final example of this type, suppose that the diagram attached
to ${\mathcal O}$ contains only twos and no zeros. Then
$M_{{\mathcal O}}$ is the torus of the group $H$ and $U_\O$ is the maximal unipotent subgroup of $H$. Therefore, in this
case we may take $\psi_{U_{{\mathcal O}}}$ the standard generic
 character corresponding to our chosen realization. In other words, if $\alpha_i$ are the simple roots of
$H$, then
$$\psi_{U_{{\mathcal O}}}(u)=\psi_{U_{{\mathcal O}}}
(x_{\alpha_1}(r_1)\ldots x_{\alpha_m}(r_m)u')=\psi(r_1+\cdots
+r_m).$$

Next, we consider the case when the diagram attached to the
unipotent orbit contains only zeros and ones. In  this case we
define the character on the unipotent subgroup $U^{(2)}_\O$.
More precisely, we consider the action of $M_{{\mathcal O}}$ on the
group $L^{(2)}_\O$, and we define the character
$\psi_{U^{(2)}_\O}$  in a similar way as we defined the
character $\psi_{U_{{\mathcal O}}}$ in the previous case.

As a first example we consider in the group $H=E_6$ the unipotent
orbit ${\mathcal O}=A_1$. The diagram which corresponds to this
unipotent class is
\begin{equation}\label{a1}
\begin{matrix} \overset{}0&--&\overset{}0&--
&\overset{}0&--&\overset{}0&--&\overset{}0.\\
&&&&|&&&&\\ &&&&\underset{1}0&&&& \end{matrix}
\end{equation}
In this case, the group $U_{{\mathcal O}}$ is a Heisenberg group, and
the group $L^{(2)}_\O$ has only one root, which is
$(122321)$. Therefore, the character we attach to this unipotent
class is given by $\psi_{U^{(2)}_\O}(x_{122321}(r))=\psi(r)$. In this case, the connected component
of the stabilizer $C$ is a group of type $A_5$. As in integral
\eqref{int1}, given an automorphic  representation $\sigma$ of
$E_6(\A)$, we can define the Fourier coefficient which
corresponds to this orbit by
\begin{equation}\label{int2}
\int\limits_{U^{(2)}_\O(F)\backslash U^{(2)}_\O(\A)}
\varphi_\sigma(ug)\psi_{U^{(2)}_\O}(u)du.
\end{equation}
This Fourier coefficient defines an automorphic function defined
on the group $C(\A)$.

As an another example, consider the unipotent class  ${\mathcal O}=\widetilde{A}_1$
in $H=F_4$. Its diagram is
\begin{equation}\label{tila1}
\overset{}{0}----\overset{}{0} ==>== \overset{}{0}----\overset{1}{0}.
\end{equation}
In this case the group  $U^{(2)}_\O$ consists of the seven
roots $\alpha=\sum_{i=1}^4n_i\alpha_i$ in $H$ such $n_4=2$. If we
define the character
$$\psi_{U^{(2)}_\O}(u)=\psi_{U^{(2)}_\O}(x_{1232}(r)u')=\psi(r),$$
then the connected component of the stabilizer inside $M_\O$ is the group $\Spin_6$.

Returning to the general case of diagrams which contains zeros and
ones only, it will be convenient to extend the unipotent group we
integrate over. This extension, which involves the theta
representation defined on the double cover of a suitable symplectic
group, is discussed in \cite{G-R-S1} section 1 in detail for
unipotent orbits of the symplectic group. In the exceptional groups
it is exactly the same. The group $U_\O/U^{(3)}_\O$ has the structure of a generalized Heisenberg
group.
(By this we mean a group which has a projection onto a Heisenberg group
such that the kernel is contained in the center.  See section \ref{section: dealing with theta functions} for a more detailed definition.)
This is a general phenomenon, which we explain briefly in section \ref{section: dealing with theta functions}. For example, in the case of the unipotent
class $A_1$ in $E_6$, the group $U^{(3)}_\O$ is trivial, and
$U_{{\mathcal O}}$ is the Heisenberg group with 21 variables. The action of  the group $C$ on $U_\O$ by conjugation
therefore induces a homomorphism
into
certain symplectic group. In the above example, it is the group
$Sp_{20}$.

In general there is a projection map $l : U_\O/U^{(3)}_\O\to {\mathcal H}_m$ where ${\mathcal H}_m$ is a
Heisenberg group with $m$ variables. Here
$m = \dim L_\O^{(1)} +1.$
As an example to this phenomena, consider in $H=F_4$, the unipotent
class ${\mathcal O}=A_1+\widetilde{A}_1$. Its diagram is given by
\begin{equation}
\overset{}{0}----\overset{1}{0}
==>==\overset{}{0}----\overset{}{0}\notag
\end{equation}
In this case, the roots in $L^{(1)}_\O$ are all twelve roots
$\alpha=\sum_{i=1}^4n_i\alpha_i$ such that $n_2=1$, and the roots in
$L^{(2)}_\O$ are the six roots with $n_2=2$. The group
$U_{{\mathcal O}}/U^{(3)}_\O$ is a  generalized
Heisenberg group, and in this case we define $l : U_\O/U^{(3)}_\O\to {\mathcal H}_{13}$ as follows.
Recall that $\cH_{13}$ is $\G_a^6 \times \G_a^6 \times \G_a$
equipped with the operation
$$(x_1|y_1|z_1)\cdot (x_2|y_2|z_2)
 = \left(x_1+x_2\;\left|\;y_1+y_2\;\left|\;z_1+z_2 + \frac
 12 (x_1 \cdot\,_t y_2 - y_1 \cdot \, _tx_2) \right.\right.\right).$$
 Here $x_1, y_, x_2, y_2 \in \G_a^6,$ realized as row vectors, and $_t$ denotes the lower transpose:
 $$
 _t(x_1 \dots x_n) = \bpm x_n \\\vdots \\x_1 \epm.
 $$
 Also $z_1, z_2 \in \G_a.$  (We separate the components of $\cH_{13}$
 with vertical bars to aid legibility in computations where individual
 entries in the row vectors are written out.)

 The mapping of
$L^{(2)}_\O$ onto the center of ${\mathcal H}_{13}$
should be an element of $L_\O^{(2),*}$ (the $M_\O$-module
dual to $L_\O^{(2)}$) in general position.
One option is as follows. For any $u'\in U^{(3)}_\O$ we have
$$l(x_{1220}(r_1)x_{1221}(r_2)x_{1222}(r_3)x_{1231}(r_4)x_{1232}(r_5)
x_{1242}(r_6)u')=(0|0|r_3+r_4)$$ It is not hard to check that the
stabilizer inside $M_{{\mathcal O}}$ is a group of type $A_1+A_1$ as
indicated in the table in \cite{C}.

In conjunction with the commutator map $L_\O^{(1)} \times L_\O^{(1)}
\to L_\O^{(2)},$ the linear form on $L_\O^{(2)}$ which has been
chosen determines a skew-symmetric form on $L_\O^{(1)} \times L_\O^{(1)}.$  To extend it to a projection $U_\O \to \cH_{13},$ one needs an isomorphism $L_\O^{(1)} \to \G_a^6 \times \G_a^6$ such that the
preimage of $\{(x|0|0): x \in \G_a^6\}$ in $L_\O^{(1)}$ is isotropic with
respect to this skew-symmetric form, and so is the preimage of $\{(0|y|0): y \in \G_a^6\}.$

 In the case at hand,  one might take
$$l(x_{0100}(m_1)x_{1100}(m_2)x_{0110}(m_3)x_{1110}(m_4)x_{0111}(m_5)x_{0120}(m_6))=(X|0|0)$$
where $X=(m_1,m_2,\ldots,m_6)\in \G_a^6.$
Then $l$ will map
$$
\{x_{1120}(m_1)x_{1111}(m_2)x_{0121}(m_3)x_{1121}(m_4)x_{0122}(m_5)x_{1122}(m_6): m_1, \dots, m_6\}
$$
isomorphically onto $\{(0|Y|0): Y \in \G_a^6\}\subset \cH_{13}.$
The precise isomorphism will depend on the
structure constants of $F_4.$  This issue is discussed in more detail
later on in section \ref{section: dealing with theta functions}.

Returning to the general case, having fixed a projection map $l : U_\O/U^{(3)}_\O\to {\mathcal H}_{2k+1}$ we get  a homomorphism of the
group $C$ into $Sp_{2k}$.
This is because $Sp_{2k}$ may be identified with the group of automorphisms of $\cH_{2k+1}$ which restrict to the identity on the center,
and because the action of $C$ on $U_\O$ will be given by
such automorphisms.

Let $\Theta^\psi_{k}$ denote the theta
representation of $\cH_{2k+1}\rtimes\widetilde{Sp}_{2k}(\A)$. Then the above
discussion indicates that the following integral
\begin{equation}\label{int3}
\int\limits_{U_{{\mathcal O}}(F)\backslash U_{{\mathcal O}}({\A})}\theta_\phi^\psi(l(u)g) \varphi_\sigma(ug)du
\end{equation}
is well defined and is left invariant under $g\in C(F)$. Here
$\theta_\phi^\psi$ is a vector in the space of $\Theta^\psi_{k}$.
Depending on whether the group $C(\A)$ splits inside
$\widetilde{Sp}_{2k}(\A)$, integral \eqref{int3} is a well
defined function on either $C(\A)$ or the metaplectic double cover $\widetilde{C}(\A)$. Notice
the difference between integrals \eqref{int2} and \eqref{int3}. Clearly
 the domain of integration is different.  But more importantly,
integral \eqref{int3} will sometimes define an automorphic function on the group  $C({\A}),$
and other times will define a genuine automorphic function on the
metaplectic
  double cover of $C(\A),$ whereas integral
\eqref{int2} always defines an automorphic function of $C(\A).$ As we will work only with integrals of the
type \eqref{int3}, this will be important for our construction.

The final case to consider is the case when the diagram which
corresponds to the given unipotent orbit ${\mathcal O}$ consists of
zeros, ones and twos. In this case we combine the above two cases.
In order to explain this, recall the group $V_\O^{(2)}/V_\O^{(3)}.$
 In the case when the weighted diagram has only zeros and twos, it is $L^{(1)}_\O.$  In the case when
there are only zeros and ones, it is $L^{(2)}_\O.$  If both ones and twos are
present, it contains a part of each of these groups.  The group $U_\O/V^{(3)}_\O$ is the direct product of the abelian group $(V^{(2)}_\O/V^{(3)}_\O) \cap L^{(1)}_\O,$ and a generalized Heisenberg group with center $(V^{(2)}_\O/V^{(3)}_\O)\cap L^{(2)}_\O.$
Define projections $l: U_\O \to \cH_{2n+1}$
and $l': U_\O \to (V^{(2)}_\O/V^{(3)}_\O) \cap L^{(1)}_\O.$
Here $2n$ is the dimension of $V^{(1)}_\O/V^{(2)}_\O.$

 As in \cite{G-R-S1} equation (1.3), the
corresponding Fourier coefficient is then given by
\begin{equation}\label{int31}
\int\limits_{U_{{\mathcal O}}(F)\backslash U_{{\mathcal O}}({\A})}\theta_\phi^\psi(l(u)g) \varphi_\sigma(ug) \psi_{U_\O}(l'(u))du
\end{equation}
Here  $\theta_\phi^\psi$ is a vector in the theta representation
defined on the suitable symplectic group.
As before we have a character of $\quo{V_\O^{(2)}},$ which, in this
case is a combination of $\psi_{U_\O}$ and
the central character of $\Theta^\psi_n,$ defined on the group
 $(V^{(2)}_\O/V^{(3)}_\O)\cap L^{(2)}_\O$ using the projection $l.$
 This character corresponds to an
 element of the $M_\O$-module dual to
 $V_\O^{(2)}/ V_\O^{(3)},$ and must be in the open orbit for the $M_\O$-action.

As an example, consider in $H=F_4$ the unipotent class ${\mathcal
O}=B_2$. Its diagram is given by
$$\overset{2}{0}----\overset{}{0} ==>==
\overset{}{0}----\overset{1}{0}$$ As was mentioned above, the Levi
part of $P_{{\mathcal O}}$ is given by $M_{{\mathcal O}}=Sp_4\cdot
GL_1^2$. In this case, the action of $M_\O$ on  $L^{(1)}_\O$ decomposes as a direct sum of two irreducible subrepresentations.
One subrepresentation corresponds to the subgroup of $M_\O$ on which $T$ acts with the roots
 $(1000), (1100), (1110), (1120), (1220),$ and
 the other corresponds to the subgroup of $M_\O$ on which $T$ acts with the roots $(0001),
(0011), (0111), (0121)$.

Similarly the action of $M_\O$ on
  $L^{(2)}_\O$
  decomposes as a direct sum of two irreducible
  subrepresentations.
  One corresponds to a four-dimensional subgroup with roots,
$(1111), (1121), (1221),$ $(1231),$ and the other
corresponds to the subgroup $U_{0122}$.
The function
$
\sum_i n_i \alpha_i \mapsto \sum_i n_i c_i
$
in this case is $\sum_{i=1}^4 n_i \alpha_i \mapsto 2n_1+n_4.$
The group $V^{(2)}_\O/V^{(3)}_\O$ consists of the first component
of $L^{(1)}_\O$ and the second component of $L^{(2)}_\O.$  In other words,
it corresponds to the six roots $(1000), (1100), (1110), (1120), (1220)$
and $(0122).$
The group $U_\O/ V^{(3)}_\O$
is the product of a five dimensional abelian group corresponding
to the roots $(1000), (1100), (1110), (1120), (1220)$
and a five dimensional Heisenberg group corresponding to
the roots $(0001),
(0011), (0111), (0121),$ and $(0122).$

  In order to define a Fourier coefficient in this case, we need to take a character in general position on the five dimensional abelian group, as well as a nontrivial character of the center of the Heisenberg group.  We may then extend the latter to a theta representation.

Our five dimensional representation of $M_\O$ can be thought of as the standard representation of $SO_5$
(with $Sp_4$ then appearing in its guise as $\Spin_5$).
A point in this representation is in general position if  the $M_\O$-invariant quadratic form does not vanish at that point.   The character $\psi_{U_\O}(u) = \psi(u_{1110})$
is easily seen to be in general position.
 The identity component
of its stabilizer in $\Spin_5$ is
 $\Spin_4\cong SL_2 \times SL_2.$
 The stabilizer in $M_\O$ contains an additional one-dimensional torus,
 however, if we restrict to elements of $M_\O$ which
 also act trivially on $U_{0122},$
 the resulting group is $SL_2^{\alpha_2} \cdot SL_2^{0120},$
and indeed,
  the tables in  \cite{C} reflect that for
this unipotent class, the group $C$ is of type $A_1\times A_1$.
 As mentioned earlier, the root subgroups  $U_\alpha$ for the roots
 $(0001), (0011), (0111), (0121),$ and
$(0122)$ form a subgroup  isomorphic to
${\mathcal H}_5$. Thus, the corresponding Fourier coefficient, for
this unipotent class is given by integral \eqref{int31}. Here
$\theta_\phi^\psi$ is a vector in the space of $\Theta^\psi_2$, the
theta representation defined on $\widetilde{Sp}_4(\A)$, and $l$
is the projection from $U_{{\mathcal O}}$ to ${\mathcal H}_5$. As in
previous examples, the character $\psi_{U^{(2)}_\O}$ which is
nontrivial on the group $U_{0122}$ is built in the theta function. We
should mention, that since the group $SL_2\times SL_2 $ does not split under the
cover of $\widetilde{Sp}_4(\A)$, it follows that the above integral
defines a genuine automorphic function on the metaplectic cover $\widetilde {SL}_2({\A})\times \wt{SL}_2 ({\A})$.

There is a third type of Fourier coefficient which can be attached 
to an orbit such that the weighted Dynkin diagram has ones in it
in certain cases.  To be specific: recall that
$U_\O/V_\O^{(2)}$ is essentially a symplectic vector space, equipped 
with a nondegenerate skew-symmetric form which is fixed by the 
group $C.$  
Let $W$ be a maximal isotropic subspace, 
and let $U_\O^{(3/2)}$ denote the preimage of $W$ in $U_\O.$
Let $\psi_{U_\O^{(3/2)}}$ denote the
trivial extension of $\psi_{U_\O^{(2)}}$ to $U_\O^{(3/2)}.$
Then the third type of Fourier coefficient is 
\begin{equation}\label{FC:u3/2}
 \iq{U_\O^{(3/2)}} \varphi_\sigma(ug) \psi_{U_\O^{(3/2)}} \, du.
\end{equation}
Note that this integral 
defines a function on $\quo{C},$ {\it only} if
$C$ normalizes $U_\O^{(3/2)},$ i.e., if the maximal 
isotropic subspace $W   \subset U_\O/V_\O^{(2)}$
is $C$-invariant.

A coefficient of this type was used in \cite{BG-GL4}.  
Also, the  integral given in \cite{Jacquet-Shalika-EulerProductsAndClassificationII} for the convolution $L$-function 
of $GL_n \times GL_m$ involves a Fourier coefficient of this type 
whenever $n-m$ is positive and even.

For completeness, we record the analogue of  lemma 1.1 of \cite{G-R-S1}.
\begin{lem}
Let $f: \quo{\cH_{2n+1}}\to \C$ be any smooth function.  The following 
are equivalent:\begin{itemize}
 \item $$\iFA f((0|0|z)u) \psi(z) \, dz = 0, \qquad (\forall u \in \cH_{2n+1}(\A)), $$
\item $$\iFA \iiFA n f((0|y|z) u) \psi(z) \, dy\, dz = 0, \qquad(\forall u \in \cH_{2n+1}(\A))$$
\item 
$$
\iq{\cH_{2n+1}} f(u'u) \theta_\phi^\psi(u') \, du' = 0, \qquad (\forall u \in \cH_{2n+1}(\A), \; \phi \in S( \A^n)).
$$
\end{itemize}
\end{lem}
\begin{proof}
The same arguments given to prove lemma 1.1 of \cite{G-R-S1} actually prove this stronger statement.
\end{proof}
\begin{cor}
For an automorphic representation $(\sigma, V_\sigma)$ of $H(\A)$ (or a covering
group), 
the following are equivalent:
\begin{description}
\item[1)] integral \eqref{int2} is zero for all $\varphi_\sigma\in V_\sigma,$ 
\item [2)]
integral \eqref{int31} is zero for all $\varphi_\sigma\in V_\sigma$  and all $\phi \in S(\A^n),$
\item[3)] integral \eqref{FC:u3/2} is zero for all $\varphi_\sigma\in V_\sigma.$ 
\end{description}
\end{cor}

\section{\bf  Constructing Global Integrals}
\label{section: constructing global integrals}

From the discussion in the previous section, it follows that given a
unipotent orbit ${\mathcal O}$ of an exceptional group $H$, we can
associate with it a set of Fourier coefficients, which are defined
on any automorphic representation of the group $H(\A)$. In this
section we shall use this to construct certain global
integrals.

We start with the case where the corresponding diagram has only
zeros and twos. Let $\pi$ denote a cuspidal representation defined
on the group $C(\A)$. Given a maximal parabolic subgroup $P$ of
$H$, with Levi decomposition $P=MU$, we let $E_\tau(h,s)$ denote an
Eisenstein series of the group $H(\A)$
formed using an element of the induced representation $\Ind_{P(\A)}^{H({\A})}\tau\delta_P^s$. Here $\tau$ is any automorphic representation
of the group $M(\A)$.

The global integral we consider in this case is given by
\begin{equation}\label{int4}
\int\limits_{C(F)\backslash C(\A)}
\int\limits_{U_{{\mathcal O}}(F)\backslash U_{{\mathcal O}}(\A)}
\varphi_\pi(g)E_\tau(ug,s)\psi_{U_{{\mathcal O}}}(u)dudg
\end{equation}

In the case when the diagram corresponding to the unipotent orbit
contains ones, we consider the integral
\begin{equation}\label{int6}
\int\limits_{C(F)\backslash C(\A)} \int\limits_{U_\O(F)\backslash U_{{\mathcal O}}(\A)}
\varphi_\pi(g)\theta_\phi^\psi(l(u)g)E_\tau(ug,s)\psi_{U_\O}(l'(u))dudg
\end{equation}
where $l$ and $l'$ project $U_\O/V^{(3)}_\O$ onto its Heisenberg
and abelian factors respectively.
(See the discussion before equation \eqref{int31}.)
As  explained in the previous section, in this
type of integral, the Fourier coefficient given by the integration
along $U_{{\mathcal O}}(F)\backslash U_{{\mathcal O}}(\A)$ may
define a genuine automorphic function on $\widetilde{C}(\A)$,
the double cover of $C(\A)$. If this is the case, then for
integral \eqref{int6} to be well defined, one of the representations $\pi$ or $E_\tau(g,s)$
should be defined on the double cover of the relevant group.

Both integrals \eqref{int4} and \eqref{int6} converge
absolutely whenever $s$ is not a pole of the Eisenstein
series. This follows from the moderate growth property
of automorphic forms and their Fourier coefficients,
and the rapid decrease property of cusp forms.
 Also, each defines a meromorphic  function on the
 whole complex plane, simply because the same is
 true of the Eisenstein series.

There are two other  families of global integrals which are of interest. We
refer to as reductive type integrals, and formally attach them to the
zero nilpotent orbit.  (For which $C$ equals $H$ itself.)
The first reductive type are integrals of the form
\begin{equation}\label{int7}
\int\limits_{H(F)\backslash H(\A)}\varphi_\pi(h)\theta(h)
E_\tau(h,s)dh
\end{equation}
where $\theta(h)$ is an
element of some irreducible  automorphic representation  $\Theta$ of  $H(\A)$. Here $\pi$ denotes a cuspidal representation of the group
$H(\A)$.
(The similarity in notation between $\Theta$ and $\Theta_n^\psi$ is deliberate:  in practice $\Theta$ is usually attached to a small orbit,
and in this way is like $\Theta_n^\psi,$ which is attached to the
minimal orbit of $Sp_2n.$)
The second family will be integrals of the type
\begin{equation}\label{int8}
\int\limits_{Z(\A)G(F)\backslash G(\A)}\varphi_\pi(g)
E_\tau(g,s)\,dg
\end{equation}
where $G$ is a reductive subgroup in $H$, $Z$ is its center,
and $\pi$ is a
cuspidal representation defined on $G(\A)$.

As mentioned before, when we deal with the groups $E_6$ and $E_7$,
we may need to use their similitude groups, in order to guarantee
that the integral will unfold to a suitable model. In these cases,
we may define the group $C$ as the stabilizer of a suitable
character $\psi_{U_\O}$ in the similitude group.  It will
certainly contain the center, $Z$ of the similitude group,
and we replace the integral over $\quo{C}$ by an
integral over $Z(\A)C(F)\bs C(\A).$
Each of the above integrals converges absolutely, and satisfies a
functional equation which is derived from the Eisenstein series.

Our main goal is to determine when  these integrals are Eulerian. To
try to answer this question, at least partly, we shall fix some
notation. We need a good way to parameterize automorphic
representations. One way to do it, which is suitable for our
problem, is by means of Fourier coefficients.

In \cite{G1}, it is explained how to associate with each unipotent
orbit a set of Fourier coefficients.
This was essentially repeated in the previous section.
 The following definition is a way to
attach to any automorphic representation a set of unipotent orbits.

\begin{definition}\label{def1}

( \cite{G1} Definition 2.1) Let $\sigma$ denote an automorphic
representation of a given split reductive group $G(\A),$ or of a
metaplectic cover of such a group. We shall denote
by $\mathcal{O}rb_G(\sigma)$ the set of all unipotent orbits ${\mathcal O}$ of $G$
such that  $\sigma$ has
a nonzero Fourier coefficient which corresponds to the unipotent
orbit ${\mathcal O}$, and by ${\mathcal O}_G(\sigma)$ the set of maximal elements
of $\mathcal{O}rb_G(\sigma).$

\end{definition}

The structure of the set ${\mathcal O}_G(\sigma)$ is not yet well
understood. In \cite{G1} there are several conjectures and
assertions which are related to this set.
To each representation which appears in integrals \eqref{int4},
\eqref{int6}, \eqref{int7} and \eqref{int8} we shall attach a number which we refer to as the
Gelfand-Kirillov dimension of the representation.
(By way of analogy with the ``dimension''
introduced in \cite{Kawanaka}, theorem 2.4.1, (iib).  Cf. also
the references on p. 158 of \cite{Kawanaka} for related
results on real Lie groups and the connection with the notion originally
introduced by Gelfand and Kirillov.)

\begin{definition}\label{gk}
\begin{enumerate}
\item Assume that the set $\mathcal{O}_M(\tau)$ contains a single
unipotent orbit $\O.$
Let $E_\tau(\cdot,s)$ denote
the Eisenstein series defined right before integral \eqref{int4}.
Then define the Gelfand-Kirillov dimension
$\text{dim}\ \tau$ of $\tau$ to be $\frac{1}{2}\text{dim}\ \mathcal{O}$.
For $\text{Re}\ s$ large define the Gelfand-Kirillov dimension of
$E_\tau(\cdot,s)$ to be $\text{dim}\ U+\text{dim}\ \tau$.
\item Since we will assume that $\pi$ is generic, then we define
$\text{dim}\ \pi=\text{dim}\ U_{\max}^C$ where  $U_{\max}^C$ is the
maximal unipotent subgroup of $C$.
\item For the theta
representation $\Theta_n^\psi$ defined on $\cH_{2n+1}(\A) \rtimes \widetilde{Sp}_{2n}({\A})$ we define $\text{dim}\ \Theta_n^\psi=n$.
\item Finally, for integrals of the type \eqref{int7}, we assume that
$\mathcal{O}_H(\Theta)$ is a singleton set $\{\O\},$ and define $\dim \Theta = \frac 12 \dim \O.$
\end{enumerate}
\end{definition}

{\bf Remark:}
Observe that $\O_C(\pi)$ is the singleton set containing the
maximal orbit of $C,$ which has dimension $2 \dim U_{\max}^C.$
Also $\O_{Sp_{2n}}( \Theta_n^\psi)$ is the singleton set containing
the minimal orbit of $Sp_{2n},$ which has dimension $2n.$
Finally, suppose that $\mathcal{E}_{\tau,s}$ is the space
of Eisenstein series attached to the induced representation
$\Ind_{P(\A)}^{H(\A)}\tau \delta_P^s.$  Then,
following \cite{G1}, proposition 5.16, we expect that
$$
\O_M(\tau) = \{\O\} \implies \O_H(\mathcal{E}_\tau) = \{ \O'\},
\text{ where }
\dim \O' = \dim \O+ 2 \dim U(P).
$$
 Actually, we can prove it for low rank
groups, including the exceptional group $F_4$. However, since we
will not need it, we omit it.

Note that the representation $\tau,$ or the representation
$\Theta$ appearing in \eqref{int7} could itself be an Eisenstein
series.  In this case the dimension is defined using the unique
element of $\O_M(\tau)$ or $\O_H(\Theta),$ and not using
definition \ref{def1}.  The motivation for this is that in the
practice of unfolding one needs to know information regarding
what Fourier coefficients these representations support,
and prefers, where possible, to make arguments that
are valid for any representation supporting a given set
of Fourier coefficients.

Next we consider the cuspidal representation $\pi$ which appears in
integrals \eqref{int4} and \eqref{int6}. Even though our main
interest is the case when $\pi$ is generic, and hence $\dim \pi$ is well defined, for possible future applications we prove

\begin{lemma}\label{lem2}

Assume that $H=F_4$ or $H=E_6$. For any nontrivial
unipotent orbit of $H$, let
$\pi$ be an irreducible cuspidal representation of $C(\A)$.
Then $\dim \pi$ is well defined.

\end{lemma}

\begin{proof}

The lemma is proved by a case by case consideration.
The type of $C$ for each orbit is given in the tables of \cite{C}.
It will be a group of type $A_1, A_2, C_2, G_2, A_3, B_3, C_3, A_5,$
or some combination of these.

If $G$ is a simple group of any of the types listed above, except for
$C_3$ and $A_5,$   then
the orbits  of $G$ are totally ordered, and so every subset has
a unique maximal element.
For $C_3,$ the set of all  special orbits is totally ordered, and it is proved in \cite{G-R-S1} theorem 2.1
that any orbit appearing in $\O_C(\pi)$ must be special.
(For a definition of ``special,'' see \cite{C}, p. 389.
For an alternate description for classical groups, see \cite{C-M}, p. 100.)

This leaves a single orbit in $E_6$ when the stabilizer is of type $A_5.$
In fact (since we work with the split form of $E_6$) the stabilizer for this case
 is isomorphic
to $SL_6$ (If we work in $GE_6,$ the stabilizer is isomorphic to $\{ g \in GL_6: \det g \text{ is a square} \}.$  This case was considered in \cite{nagoya}.)
Every cuspidal representation of $SL_6$ is generic with respect to
some generic character, so in this case $\dim \pi$ is always defined
and equal to the dimension of the maximal unipotent subgroup in $SL_6.$
\end{proof}

Finally, we consider the representations occur in integrals
\eqref{int7} and \eqref{int8}. In these cases, we will assume that
$\pi$ is generic, and for the function $\theta$ which occurs in
integral \eqref{int7}, we will assume that its dimension is well
defined. This will not cause a problem, since the for cases we will
consider this assumption will be satisfied automatically. We record all
this in the next lemma.

\begin{lemma}\label{lem3}

With the above assumptions, the notion of dimension is well defined
for any representation occurring in integrals \eqref{int4},
\eqref{int6}, \eqref{int7} and \eqref{int8}.

\end{lemma}

Recall that our goal is to study the question when are the integrals
\eqref{int4}-\eqref{int8} are Eulerian. Following the discussion in
the introduction of  \cite{G3} we introduce the following definition.

\begin{definition}\label{equation}
({\bf The basic equation:})  An integral of the type \eqref{int4},
\eqref{int6}, \eqref{int7}, or \eqref{int8}, is said to  satisfy the
basic equation if the sum of the Gelfand-Kirillov dimensions of the
representations involved in the integral is equal to the dimension
of the domain over which we integrate.
\end{definition}

We now write the basic equation in detail. Consider first integral
\eqref{int4}. The basic equation in this case is given by
\begin{equation}\label{dim0}
\dim C+\dim U_{{\mathcal O}}=\dim \pi+\dim E_\tau(\cdot,s)
\end{equation}
For integral \eqref{int6} the equation is given by
$$\dim C+\dim U_\O=\dim \pi+\dim E_\tau(\cdot,s)+\dim \Theta^\psi_k.$$
Here $k=\frac12(\dim U_\O-\dim V_\O^{(2)})$ is the unique integer such that $U_\O$ has a projection onto
$\cH_{2k+1},$ with the kernel of this projection contained in $V^{(2)}_\O.$
In the first identity above,
we
have $\dim U_{{\mathcal O}}=\frac{1}{2}\dim {\mathcal
O},$ while in the second case, we have
$\dim U_{{\mathcal O}}-\dim \Theta^\psi_k=\frac{1}{2}\dim {\mathcal O}$.
This follows from lemma 4.1.1, p. 56 of \cite{C-M}.

Hence, it follows that if we start from a nonzero unipotent orbit ${\mathcal
O}$ then the relevant integral \eqref{int4} or
\eqref{int6} satisfies the basic equation if and
only if it satisfies the equation
\begin{equation}\label{dim1}
\dim C+\frac{1}{2}\dim {\mathcal O}=\dim \pi+\dim E_\tau(\cdot,s)\notag
\end{equation}
Since we assumed that $\pi$ is generic, it follows that $\dim \pi=\dim U_C$ where $U_C$ is the maximal unipotent subgroup
of $C$. From definition \ref{gk} we also have that $\dim E_\tau(\cdot,s)=\dim U(P)+\dim \tau$. Plugging  this
inside the above  equation we obtain the basic equation
\begin{equation}\label{dim2}
\dim \tau=\dim B_C+\frac{1}{2}\dim {\mathcal O}
-\dim U(P)
\end{equation}
where $B_C$ is the Borel subgroup of $C$. Notice that this last
formula holds also for integral \eqref{int8}, if we take
$\dim {\mathcal O}=0$ and $C=G$. Arguing in a similar way for
integral \eqref{int7}, we deduce that the formula that should hold
in this case is
\begin{equation}\label{dim3}
\dim \theta+\dim \tau=\dim B_H -\dim U(P)
\end{equation}
where $B_H$ is the Borel subgroup of $H$.

We consider a few examples. We start with an example related to
integral \eqref{int7} for the group $H=F_4$. In this case, equation
\eqref{dim3} reduces to $\dim \theta+\dim \tau=28
-\dim U(P)$. Since $F_4$ has two maximal parabolic subgroups $P$
such that $\dim U(P)= 20,$ and two such that $\dim U(P)= 15,$ it follows that there are two possibilities.
First suppose that $\dim U(P)=20.$  This is the case
if $P=P_2$ or $P_3.$  Then \eqref{dim3} reduces to
$\dim \theta+\dim \tau=8$.  Now, $\dim \theta$ is equal to one half
of the dimension of a nontrivial nilpotent orbit of $F_4.$
The dimensions of the orbits are listed on page 128 of \cite{C-M}.
The minimal orbit has dimension $16,$ so it follows that in this case
$\theta$ must be attached to the minimal orbit, and
$\tau$ must be a character.  An example of  a representation
which is attached to the minimal orbit of $F_4$
is constructed and studied in \cite{G4}. In fact, it is a
representation which is defined on the double cover of $F_4(\A)$.

Now suppose that $\dim U(P) = 15.$
This is the case when
$P=P_1$ or $P=P_4,$ so that the Levi part is isomorphic to $GSp_6$
or $\GSpin_7,$
respectively.
In this case, \eqref{dim3} reduces to
$\dim \theta+\dim \tau=13.$
It then follows from \cite{C-M}
page 128, that either $\theta$
is attached to the minimal orbit, and $\dim \theta =8,$
or else $\theta$ is attached to the unipotent orbit $\widetilde{A}_1,$
and $\dim \theta =11.$
If $\dim \theta$ is $11,$ then we need $\dim \tau =2.$
  However, neither $GSp_6$
nor $\GSpin_7,$
 has an
orbit whose dimension is 4.  So, once again $\theta$ must be
attached to the minimal orbit.
The basic equation will be satisfied if we can find a representation
$\tau$ defined on $GSp_6$
or $\GSpin_7,$
such that the dimension is five.
Each of these two groups has a unique unipotent
orbit $\O$ of dimension $10.$  The unique $10$-dimensional unipotent
orbit of $GSp_6$ is $(2^21^2),$ and that of $\GSpin_7$ is $(31^5).$
(For the parametrization of orbits by partitions, see \cite{C-M}, p. 70.
To compute the dimension, use \cite{C-M}, \S 5.2, and  lemma 4.1.1.)

From this, it follows that there are essentially six different
possibilities for an integral of
 the form
\eqref{int7}.  First, suppose that the  Eisenstein series
 $E_\tau(h,s)$ comes from a representation induced from either $P_1$
or $P_4$, and $\tau$ is any representation of $GSp_6$
or $\GSpin_7$ whose
dimension is five.  Since the representation $\theta$ is on the
double cover, it follows that either $\pi$ or $E_\tau(g,s)$ has to
be defined on the double cover. If we form the Eisenstein series on
the double cover of $H$, we obtain an integral of the form
$$\int\limits_{H(F)\backslash H(\A)} \varphi_\pi(h)
\widetilde{\theta}(h)\widetilde{E}_\tau(h,s)dh.$$ Here $\tau$ is an
automorphic representation defined on the double cover of $GSp_6(\A)$ or
$\GSpin_7(\A)$ whose dimension is five. A second option is to let $\pi$
be a cuspidal representation defined on the metaplectic
double cover $\widetilde{H}(\A)$ of $H(\A).$ Thus, in this case, we have
$$\int\limits_{H(F)\backslash H(\A)} \widetilde{\varphi}_\pi(h)
\widetilde{\theta}(h)E_\tau(h,s)\, dh.$$ In this option $\tau$ is an
automorphic representation of $GSp_6(\A)$ or $\GSpin_7(\A)$ whose dimension
is five.

Now suppose that the Eisenstein series
$E_\tau(h,s)$  is attached to the parabolic
subgroup $P_2$ or $P_3$. In this case $\tau$ must be a character, and
it follows that $E_\tau(h,s)$ can not be a genuine function on the double
cover of $H(\A).$   Hence we must assume that $\pi$ is a cuspidal
representation defined on the double cover of $H(\A)$. The integral we
construct is similar to the above, where the only change is the
choice of the Eisenstein series.

As a second example,  consider the unipotent orbit ${\mathcal
O}=A_1$ for the group $H=E_6$. Its diagram is given in \eqref{a1}.
The relevant equation to consider is \eqref{dim2}. It follows from
\cite{C-M} that $\frac{1}{2}\dim  {\mathcal O}=11$, and from
\cite{C} page 402, it follows that $C$ is of type $A_5$. Thus
$\dim B_C=20$. Hence equation \eqref{dim2} is $\dim \tau=31-\dim U(P)$. The group $E_6$ has 4 non associated
maximal standard  parabolic subgroups. Their unipotent radicals have
dimensions $\dim U(P_1) =16,
\dim U(P_2)=21,\dim U(P_3)=25$ and $\dim U(P_4)=29,$
respectively.
The integral to consider in this case is the one written in
\eqref{int6}.  We consider each parabolic in turn.   First, if $P=P_1$, then
$\dim \tau=15$. The Levi part of $P_1$ is a group of type
$D_5$. It is not hard to check that this group has a unique
unipotent orbit such that half of its dimension is 15. This
unipotent orbit is $(3^31)$. Next, when $P=P_2,$ we obtain
$\dim \tau=10$. Since the Levi part of $P_2$ is a group of type
$A_5$, and since this group has no unipotent orbit of dimension twenty,
it follows that  the set of solutions to equation \eqref{dim2} is empty in this case. For $P=P_3$ we
obtain $\dim \tau=6$. The Levi part of this parabolic subgroup
is a group of type $A_1\times A_4$. There is one possibility. On the
group of type $A_1$ we take a representation attached to the orbit
$(1^2),$ i.e.,
a character, and, on the
group of type $A_4$, we take a representation attached to the
unipotent orbit $(2^21)$. We denote this by $(1^2 | 2^21)$. Finally,
for $P=P_4$, we have $\dim \tau=2$ and since the Levi part of
this parabolic subgroup is of type $A_2\times A_1\times A_2$, then the
only choice is  take a representation attached to the
unipotent orbit $(21)$ on one of the two $A_2$-components, and take characters on the other two components.
By symmetry, it does not matter which of the two $A_2$-components we
take the representation attached to $(21)$ on, so in this case there
is only one possibility, which we represent by $(21|1^2|1^3).$

In the following table we summarize the following data for the
group $H=F_4.$  First, in the leftmost column we label
the unipotent orbit in question. The non special orbits are labeled
with bold letters.  In the next column we write the type of the
group $C$. Since $H=F_4$, we have four maximal parabolic subgroups.
In the next two columns we write down the dimensions of $\tau$
required so that equation \eqref{dim2} holds.  Since  the unipotent radicals of $P_1$ and $P_4$
have the same dimension, we can write them in the same column.
Similarly for $P_2$ and $P_3$. In the next four columns we write
down possible unipotent orbits, defined on the Levi part of $P$,
whose dimension is twice the dimension of $\tau$. For example, consider
the unipotent orbit $B_2$. Since it is not special we label it with
bold letters. The stabilizer is of type $A_1+A_1$. When $P$ is such
that its Levi part is of type $C_3$ or $B_3$, then the dimension of
$\tau$ required to satisfy equation \eqref{dim2} is seven. The group
of type $C_3$ has 2 unipotent orbits of dimension 14. They are
$(3^2)$ and $(41^2)$. They are listed at the fifth column. Notice,
that since $(41^2)$ is not special we labeled it with bold letters.

Finally, the first row in the table is devoted to the integral of
type \eqref{int8}. The numbers indicated are the numbers
$\dim \tau+\dim \theta$ such that equation
\eqref{dim3} is satisfied. For example, if $P=P_1$ then $\dim \tau + \dim \theta$ must equal $13,$ and, as discussed above, this means that $\theta$
must be attached to the minimal orbit of $F_4.$
The Bala-Carter label of this orbit is ${\bf A_1}.$
Also,   $\tau$ must be attached to the orbit $(2^21^2).$

One may also consider  what we refer as the ``dual integrals'' to
integrals (9) and (10),
in which the roles of the cusp form and the Eisenstein series
are reversed.
For example, the integral dual to integral (9) is given
by
\begin{equation}\label{dual}
\int\limits_{C(F)\backslash C(\A)} \int\limits_{U_{{\mathcal
O}}(F)\backslash U_{{\mathcal O}}(\A)}
\varphi_\pi(ug)E_\tau(g,s)\psi_{U_{{\mathcal O}}}(u)\,du\,dg,\notag
\end{equation}
where $\varphi_\pi$ is a cusp form in the space
of a cuspidal automorphic representation of $\quo H$ and $E_\tau$ is an Eisenstein series on $\quo C.$
Convergence of such integrals follows from the fact that the Fourier coefficients of cusp forms inherit their rapid decay property, as has
recently been established in general by Miller and Schmid \cite{MS}.
However, in this paper we are mainly interested in integrals such
that the representation $\pi$ is generic, and which satisfy the equation
$$\text{dim}\ C+\frac{1}{2}\text{dim}\ {\mathcal{O}}=\text{dim}\
\pi+\text{dim}\ E_\tau( \cdot, s) $$

Since we require that $\pi$ is
generic, then $\text{dim}\ \pi=24$.  Further, $\dim E_\tau( \cdot, s)$
is strictly positive.
To be precise, it is at least equal to the smallest integer which
is the dimension of the unipotent radical of a parabolic of $C.$
There are only three nonzero orbits in $F_4$
such that $\frac 12\dim \O+ \dim C - 24$ is positive.  These
are ${\bf A_1}, \wt A_1,$ and $\wt A_2.$
 In the case of $\wt A_1,$ we have $\frac 12\dim \O+ \dim C - 24=2,$
 and $C \cong SL_4.$  This implies $\dim E_\tau( \cdot, s)
 \ge 3,$ ruling out this case.

 For $\wt A_2$ we have $\frac 12\dim \O+ \dim C - 24=5,$
 and $C \cong G_2.$  If $P$ is either of the maximal
 parabolic subgroups of $G_2,$ then the dimension
 of the unipotent radical of $P$ is $5,$ and so
 $\dim E_\tau( \cdot, s)=5$ if $\tau$ is a character of
 the Levi part of $P.$
  One of these cases can be reduced to a Whittaker
integral, and appears in \cite{G-R}.

 For ${\bf A_1},$ we have $\frac 12\dim \O+ \dim C - 24=5,$
 and $C \cong Sp_6.$
 It is possible to arrange that $\dim E_\tau( \cdot, s)=5$
 only by taking  $E_\tau( \cdot, s)$ to be an Eisenstein
 series  attached to
 the induced representations
 $\Ind_{P(\A)}^{Sp_6(\A)} \tau\delta_P^s,$ where $P$
the Levi subgroup of $P$ is isomorphic to $GL_1 \times Sp_4,$
and $\tau$ is a character.

 Thus, there are in effect only three cases for this type of integral,
 and one of them has already appeared in the literature.  We
 remark that a preliminary analysis suggests that neither of
 the others can unfold to a Whittaker integral.

\newpage

\centerline{\bf The group $F_4$}
\bigskip\bigskip

\begin{center}
\begin{tabular}{ | l | r | r | r | r | r | r | r | r |}
\hline label & stabilizer  & $P_1,P_4$ & $P_2,P_3$& $C_3$&$B_3$& $A_2+A_1$& $A_1+A_2$ \\
\hline 1&$F_4$&13&8&$[{\bf A_1},\ \ $&$[{\bf A_1},\ $&
 $[{\bf A_1},(1^3|1^2)]$&$[{\bf A_1},(1^2|1^3)]$\\
      &&&&$(2^21^2)]$&$(31^4)]$&&\\
\hline
  &&5&0 & $(2^21^2)$&$(31^4)$  & $(1^3|1^2)$& $(1^2|1^3)$\\
  ${\bf A_1}$& $C_3$&&&&&&\\
  \hline
  && 5&0 & $(2^21^2)$&$(31^4)$& $(1^3|1^2)$& $(1^2|1^3)$\\
  $\widetilde{A}_1$& $A_3$&&&&&&\\
  \hline
  &&1&---&---&---&---&---\\
  $A_1+\widetilde{A}_1$& $A_1+A_1$&&&&&&\\
  \hline
  && 5&0 & $(2^21^2)$&$(31^4)$  & $(1^3|1^2)$& $(1^2|1^3)$\\
  $A_2$& $A_2$&&&&&&\\
  \hline
  && 8&3&$(42)$&$(51^2)$ &$(3|1^2),(21|2)$&$(1^2|3),(2|21)$ \\
  $\widetilde{A}_2$& $G_2$&&&&&&\\
  \hline
  && 4&--- &---&${\bf (2^21^3)}$&---&---\\
  ${\bf A_2+\widetilde{A}_1}$ & $A_1$&&&&&&\\
  \hline
  && 7&2 &$(3^2)$&$(3^21)$ & $(21|1^2)$& $(1^2|21)$\\
  ${\bf B_2}$ & $A_1+A_1$&&&${\bf (41^2)}$&&&\\
  \hline
  &&5&0&$(2^21^2)$&$(31^4)$& $(1^3|1^2)$& $(1^2|1^3)$ \\
  ${\bf \widetilde{A}_2+A_1}$& $A_1$&&&&&&\\
  \hline
  & & 6& 1& $(2^3)$&$(32^2)$  & $(1^3|2)$& $(2|1^3)$ \\
  ${\bf C_3(a_1)}$ & $A_1$&&&&&&\\
  \hline
  && 5 &0 &$(2^21^2)$&$(31^4)$& $(1^3|1^2)$& $(1^2|1^3)$\\
  $F_4(a_3)$ & 0&&&&&&\\
  \hline
  && 8&3&$(42)$&$(51^2)$ &$(3|1^2),(21|2)$&$(1^2|3),(2|21)$\\
  $B_3$ & $A_1$&&&&&&\\
  \hline
  & &8&3&$(42)$&$(51^2)$ &$(3|1^2),(21|2)$&$(1^2|3),(2|21)$\\
  $C_3$ & $A_1$&&&&&&\\
  \hline
  && 7&2&$(3^2)$&$(3^21)$ & $(21|1^2)$&$(1^2|21)$\\
  $F_4(a_2)$ & 0&&&${\bf (41^2)}$&&&\\
  \hline
  && 8 & 3&$(42)$&$(51^2)$ &$(3|1^2),(21|2)$&$(1^2|3),(2|21)$ \\
  $F_4(a_1)$  & 0&&&&&&\\
  \hline
  & & 9& 4& $(6)$&$(7)$& $(3|2)$&$(2|3)$\\
  $F_4$ & 0&&&LS&LS&LS\ \ &LS\ \ \\
  \hline
\end{tabular}
\end{center}

\newpage

\section{\bf  Unfolding Global integrals}
\label{section: unfolding global integrals}

In this section we indicate several steps in the unfolding of
integrals \eqref{int4} and \eqref{int6}. The integrals \eqref{int7}
and \eqref{int8} are treated similarly. Integral \eqref{int6} is
more complicated since it also involves the theta representation.
However, at least the first steps are similar, hence we will
first concentrate on integral \eqref{int4}, and then discuss
integral \eqref{int6} more briefly.  We will also return to integral
\eqref{int6} in section \ref{section: dealing with theta functions}.

\subsection{ The General Process}
\label{Subsection:  The General Process}

The purpose of the unfolding process is to determine if the integral
is Eulerian, by reducing things to certain functionals defined on
the representations in question. Then one studies if these
functionals are unique. For example, integral \eqref{int4} is given
by
\begin{equation}\label{int41}
I=\int\limits_{C(F)\backslash C(\A)} \int\limits_{U_{\O}(F)\backslash U_{{\mathcal O}}(\A)}
\varphi_\pi(g)E_\tau(ug,s)\psi_{U_{\O}}(u)dudg
\end{equation}
The two representations which are of import to us are $\pi$ and
$\tau$. After we finish the  unfolding  process we will obtain
certain inner integration, which may be an integral over a reductive
group, or a Fourier coefficient or some combination of the two.
Since we
assume that the cuspidal representation $\pi$ is generic, it is
natural to hope to obtain a Whittaker coefficient after the
unfolding process is finished. We  assume that $\tau$ has
Gelfand-Kirillov dimension as specified by the dimension formula. In
most cases, this means that $\tau$ is not generic.

We now write several steps in the unfolding process. Unfolding the
Eisenstein series we obtain
$$E_\tau(ug,s)=\sum_{\gamma\in P(F)\backslash H(F)}f_\tau(\gamma
ug,s)=$$ $$\sum_{w\in P(F)\backslash H(F)/P_{{\mathcal O}}(F)}\ \
\sum_{\gamma\in (w^{-1}P(F)w\bigcap P_{{\mathcal O}}(F))\backslash
P_{{\mathcal O}}(F)}f_\tau(w\gamma ug,s)$$ The first sum is finite,
and representatives can be chosen to be Weyl elements of $H$. Thus,
if we define
\begin{equation}\label{int42}
I_w=\int\limits_{C(F)\backslash C(\A)} \int\limits_{U_\O(F)\backslash U_{{\mathcal O}}(\A)}
\varphi_\pi(g)\sum_{\gamma\in (w^{-1}P(F)w\bigcap P_\O(F))\backslash P_{{\mathcal O}}(F)}f_\tau(w\gamma
ug,s)\psi_{U_{{\mathcal O}}}(u)dudg\notag
\end{equation}
then $I=\sum_w I_w$ where the sum is over the space of double cosets
$P(F)\backslash H(F)/P_{{\mathcal O}}(F)$. Moreover, we
can write the sum
$$\sum_{\gamma\in (w^{-1}P(F)w\bigcap P_\O(F))\backslash P_{{\mathcal O}}(F)}=\sum_{\gamma\in
Q_w(F)\backslash M_{{\mathcal O}}(F)}\ \ \sum_{\mu\in\  U_{\O}^w(F)\bs U_{\O}(F)}
=\sum_{\gamma\in
Q_w(F)\backslash M_{{\mathcal O}}(F)}\ \ \sum_{\mu\in\  U_{\O,w}(F)}$$ Here $Q_w=M_\O \cap w^{-1} P w,$ which is a parabolic subgroup of $M_{{\mathcal O}}$,  $U_\O^w = U_\O \cap w^{-1} P w,$
and $U_{{\mathcal O},w}= U_{{\mathcal O}}\cap w^{-1} U^- w,$ where $U^-$ is the unipotent radical of the parabolic subgroup which is opposite to $P.$

The next step is to consider the space $Q_w(F)\backslash M_\O(F)/C(F)$. We have
$$\sum_{\gamma\in Q_w(F)\backslash M_{{\mathcal O}}(F)}=
\sum_{\nu\in Q_w(F)\backslash M_{{\mathcal O}}(F)/C(F)}\ \
\sum_{\gamma\in (\nu^{-1}Q_w(F)\nu\bigcap C(F))\backslash C(F)}$$
Plugging all this inside $I_w$, we then collapse summation with
integration. Thus, if we denote
\begin{equation}
\label{Eq: Contribution from the open orbit}
I_{w,\nu}=\int\limits_{(\nu^{-1}Q_w(F)\nu\bigcap C(F))\backslash
C(\A)}\ \  \int\limits_{U_{{\mathcal O}}^{w\nu}(F)\backslash
U_{{\mathcal O}}(\A)} \varphi_\pi(g)f_\tau(w\nu
ug,s)\psi_{U_{{\mathcal O}}}(u)dudg
\end{equation}
then $I_w=\sum_{\nu} I_{w,\nu}$, where the sum is over
representatives $\nu$ for the set  $Q_w(F)\backslash M_{{\mathcal O}}(F)/C(F),$ and  $U_{{\mathcal O}}^{w\nu}= U_\O \cap (w\nu)^{-1} P w\nu.$

Denote $L_\nu= \nu^{-1}Q_w\nu\bigcap C$ and $V=(w\nu) U_\O^{w\nu}
(w\nu)^{-1}$.
For $v\in V$ we denote $\psi_V(v)=\psi_{U_\O}((w\nu)^{-1}vw\nu)$. Factoring the measure in the above
integral, $I_{w,\nu}$ is equal to
\begin{equation}\label{int44}
\begin{aligned}
\int\limits_{L_\nu (\A)\backslash C(\A)}\ \
\int\limits_{U_{{\mathcal O}}^w(\A)\backslash U_\O(\A)} \int\limits_{L_\nu(F)\backslash L_\nu (\A)}
\int\limits_{V(F)\backslash V(\A)}&
F_\tau(m(v)m(w\nu l(w\nu ) ^{-1});\, w\nu ug,s)\\
&\psi_{V}(v)\varphi_\pi(lg)\psi_{U_\O}(u)\left|\delta_P(w\nu l(w\nu
)^{-1} )\right|^s \,dv\,dl\,du\,dg\end{aligned}
\end{equation}
To explain the notation, we recall that the natural definition for
an element of $\Ind_{P(\A)}^{H(\A)} \tau \delta_P^s$ is as a
function from $H(\A)$ taking values in the space of $\tau\otimes \delta_P^s.$  Since
the space of $\tau\otimes \delta_P^s$ consists of automorphic forms$: M(\A) \to \C,$
an element of $\Ind_{P(\A)}^{H(\A)} \tau \delta_P^s$ thus becomes a
function $M(\A) \times H(\A) \to \C,$ which we write
$F_\tau( m; h, s).$ This function satisfies
$$
F_\tau(m_1; m_2h, s) = \tau\otimes \delta_P^s(m_2) F_\tau(m_1; h,s) =  F_\tau ( m_1m_2; h,s).
$$
 for all $h\in H(\A), m_1, m_2\in M(\A)$ and $s\in \C.$
In order to define Eisenstein series, one passes to an alternate
realization of $\Ind_{P(\A)}^{H(\A)} \tau \delta_P^s$ consisting of
complex-valued functions, by defining $f_\tau(h,s) = F_\tau(e;
h,s).$
See page 60 of \cite{LNM1254} for a further discussion.
Also, in \eqref{int44} we have denoted the natural projection from
the parabolic subgroup $P$ to its Levi $M$ by $m.$

To proceed with the
unfolding process we consider the following integral
\begin{equation}\label{int45}
\int\limits_{L_\nu(F)\backslash L_\nu (\A)}
\int\limits_{V(F)\backslash V(\A)}
\varphi_\pi(lg)\varphi_\tau(m(v)m(w\nu l\nu^{-1} w^{-1}))\psi_{V}(v)\,dv\,dl
\end{equation}
It follows from the discussion after integral \eqref{int44} that, as
far as the unfolding process goes, unfolding integral \eqref{int45}
is equivalent to unfolding  integral \eqref{int44}. We
emphasize that integral \eqref{int45} is not an inner integration to
integral \eqref{int44}.

Thus, we have reduced the computation to a lower rank situation.

We expect that for most elements $\nu$, integral \eqref{int44} will
be zero. We are interested in those cases where the following holds.

\begin{definition}\label{def4}

We will say that that the global integral \eqref{int41} is an {\bf open
orbit type} integral if there is a unique element $w_0\in
P(F)\backslash H(F)/P_{{\mathcal O}}(F)$ and a unique element in the
space $\nu_0\in Q_w(F)\backslash M_{{\mathcal O}}(F)/C(F)$ such that
the integral \eqref{int45} satisfies the dimension identity
\begin{equation}\label{id1}
\dim \pi +\dim \tau= \dim L_{\nu_0}
+\dim V
\end{equation}
If
 \eqref{id1} is satisfied by more than one element $\nu,$
 but
 we can show that $I_{w,\nu}$ vanishes for all but one,
 then we shall say that the integral \eqref{int41} is
{\bf weakly open orbit type}.
\end{definition}

To explain this terminology, consider the action of $P \times U_\O C$
on $H$ by
$(p, ug) \cdot h = p h (ug)^{-1}.$
Then the stabilizer of $w \nu$ is isomorphic to $L_{\nu} U_\O^{w\nu},$
and in particular is of dimension $\dim L_\nu + \dim V.$
The identity \eqref{dim0} implies that
$$\dim \pi + \dim \tau = \dim U_\O + \dim C + \dim P - \dim H.$$
This is equal to $\dim L_\nu+ \dim V$ if and only if
the orbit of $w\nu$ has the same dimension as $H,$ in which
case it is open.
Observe that an open orbit for this action (i.e., an
open $P, U_\O C$-double coset), will certainly be contained in the open
$P, P_\O$-double coset.  Let $w_0$ denote a representative
for this orbit, chosen to be the shortest element of the Weyl group
which is contained in the orbit.
Then the mapping $Q_{w_0} \nu C \to P w_0 \nu C U_\O$ is a
well defined bijection from $Q_{w_0} \bs M / C$ to the space of $P, U_\O C$-double cosets contained in the open $P, P_\O$-double coset.  Hence,
there is an open $P, U_\O C$-double coset in $H$ if and only if there
is a $Q_{w_0}, C$-double coset in $M_\O.$

Now assume that $w\nu \in H(F).$  (Recall that
we identify $H$ with $H(\Fb),$ where $\Fb$ is a fixed algebraic closure
of $H,$ and likewise for other algebraic groups.)
Then $H(F) \cap P w\nu U_\O C$ may not be a single $P(F) \times U_\O(F) C(F)$-orbit.  In cases where it is not, the integral \eqref{int41} is
not of open orbit type, but it may still be of weakly open orbit type.

Notice that the dimension identity given in definition \ref{def4} is
analogous to the dimension identity \eqref{dim0}, but this time
the representations under consideration are
 $\pi$ and $\tau$ (as opposed to $\pi$ and $\Ind_{P(\A)}^{G(\A)} \tau \otimes \delta_P^s$). Experience indicates the
following:

\begin{conjecture}\label{conj1}

Suppose that integral \eqref{int41} is an open orbit type integral.
Then  integral \eqref{int44} is zero, except when $w=w_0$ and
$\nu=\nu_0.$ Moreover, the element $w_0$ can be chosen to
be Weyl element which corresponds to the longest Weyl element in the
space $P(F)\backslash H(F)/P_{{\mathcal O}}(F)$.

\end{conjecture}

At this point our ability to check whether Conjecture \ref{conj1} holds, is
done only by carrying out the unfolding process.  In addition, unfolding is,
in general, necessary to determine whether or not an integral is weakly
open orbit type, in the case when  $Q_{w_0} \bs M_\O/C$
 possesses an open orbit, which is a union of two or more orbits
 for the action of $Q_{w_0}(F) \times C(F).$

We mention that all of the above is similar for integrals of the
form \eqref{int6}. The theta function which appears in the integral
will appear in the same way in the analogue of integral
\eqref{int45}, and definition \ref{def4} and conjecture \ref{conj1}
are similar with the addition of the theta function.  See section \ref{subsection: Unfolding in the presence of theta functions} for some details.

We make one more definition.

\begin{definition}\label{preWhittaker}

Suppose that conjecture \ref{conj1} holds. We will refer to the
integral \eqref{int41} as a {\bf preWhittaker} integral if, for
any representation $\tau$ such that ${\mathcal O}(\tau)$ is given in
the above tables, there is a process of further unfolding of
integral \eqref{int45} which involves the  Whittaker coefficient of
the representation $\pi$.  We shall also refer to an integral as {\bf weakly preWhittaker} if we are able to unfold the integral, but only by placing additional conditions on $\tau$ (e.g., requiring that $\tau$ be an Eisenstein series) which do not follow from the conditions on  $\O(\tau).$
\end{definition}

To prove that a given global integral
is weakly preWhittaker, we only need to find a representation $\tau$
such that integral \ref{int45} will unfold to a Whittaker coefficient in
$\pi$.
This is frequently quite easy.
Proving that a global integral is preWhittaker is more challenging, and
generally requires some knowledge about the class of all
representations of the Levi $M$ attached to a given orbit.
It is not clear how to prove that an integral is not weakly preWhittaker,
or
preWhittaker. Even though it is usually easy to guess when an
integral does not unfold to a Whittaker coefficient in $\pi$, it is
not clear how to prove it.

To satisfy the dimension equations, we require that $\tau$ be a
representation which supports certain Fourier coefficients corresponding to specific unipotent orbits.
This is a method of sorting automorphic representations into classes which is convenient for our purposes.  Another method is to sort the representations into cuspidal representations, Eisenstein series, and residual representations, and then sort the Eisenstein series according to which parabolic subgroup they are induced from, etc.
Note that knowing  the Gelfand-Kirillov dimension of a representation $\tau,$ or even the more refined invariant $\O_G(\tau),$ does not tell us where $\tau$ falls in this other classification.
For example, a representation which is generic could be cuspidal, or it could be an Eisenstein series induced from any parabolic subgroup.
 However, it may happen that
for some representation $\tau$ integral \eqref{int45} can be further
reduced by means of Fourier expansions to a Whittaker type integral.
See the examples given in sections \ref{An Example} and \ref{sec: Example with theta}.

There is one exception to the general principle that $\O rb_G(\tau)$ does
not determine the type of $\tau.$

\begin{lem}  If $\O_G(\tau)$ is the singleton set containing the
zero orbit, then $\tau$ is a character.
\end{lem}
\begin{proof}
One may express $G$ as a quotient of $G_1\times \dots \times G_k \times T$
for some simple simply connected groups $G_1, \dots , G_k$ and
torus $T,$  so the general case reduces to the case when $G$ is
simple and simply connected.

We show that if $\varphi$ is an automorphic form which does not
support any Fourier coefficient attached to a unipotent
orbit then  $\vph$ is invariant
by $U_\alpha(\A)$ for all $\alpha,$ and hence invariant
by $(G,G)(\A),$ where $(G,G)$ is the commutator subgroup.

To prove $U_\alpha(\A)$-invariance for all $\alpha,$ it suffices to prove it
for one element of each Weyl-orbit in $\Phi(G,T),$ i.e., for one
root of each length.  If $\O_{\min}$ is the minimal
orbit then $U_{\O_{\min}}$ is a Heisenberg group with center
$U_{\alpha}$ for some long root $\alpha_L.$ The fact that
Fourier coefficients attached to this orbit vanish implies that
$\vph$ is invariant by the center.  This by itself treats the
case when $G$ is simply laced.  In the case when there
are long and short roots, this proves that $\vph$ is invariant
by $U_{\alpha}(\A)$ for all long roots $\alpha.$  Further,
in each case one my identify a second orbit $\O'$ such that
$V_{\O'}^{(2)}$ is the product of $U_{\alpha_S}$ for some
unique short root $S,$ and possibly one or more root subgroups
corresponding to long roots.  Further, in each case the character
$\psi_{V_{\O'}^{(2)}}(v) = \psi( v_{\alpha_S})$ is in general
position.  Since $\vph$ does not support any Fourier coefficients
attached to $\O'$ either, it follows that $\vph$ is invariant by $U_{\alpha}(\A)$
for short roots as well.
\end{proof}

\subsection{A Special Case}
\label{Subsection: A Special Case}
In this subsection we consider a special case which is valid for any
 group $H$. We explain the details for the case of
integral \eqref{int4}, but it is clearly the same for integrals of
the type given by \eqref{int6}.  Let ${\mathcal O}$ be such that
$P_{{\mathcal O}}$ is maximal.  Let $w_\ell$ denote the longest
element of the Weyl group.  Then $w_\ell M_\O w_\ell^{-1}$ is a
standard maximal Levi subgroup.  Let $P_\O^a$ be the standard
maximal parabolic subgroup of $H$ such that $M =   w_\ell M_\O w_\ell^{-1}.$  (Among exceptional groups, $P_\O^a$ is equal to $P_\O$
except in some cases in $E_6.$)
 We consider the integral  \eqref{int4} in the special case
 when the parabolic
subgroup $P=MU$ from which we form the Eisenstein series
$E_\tau(h,s)$
is equal to $P_\O^a.$
 Thus $\tau$ is an automorphic representation of
$M(\A)$ and we assume that identity \eqref{dim0} holds
\begin{equation}\label{dim01}
\dim C+\dim U_{{\mathcal O}}=\dim \pi+\dim E_\tau(\cdot,s)
\end{equation}
Unfolding the Eisenstein series, we consider the space
$P(F)\backslash H(F)/P_{{\mathcal O}}(F)$. Let $w_0$ be longest Weyl
element in this space. From the relation $P=P_{{\mathcal O}}^a,$ we deduce that $U_{\O,w_0}=U_{{\mathcal O}}$ and that $Q_{w_0}=M_{{\mathcal O}}$. Also,
$w_0^{-1}M_{{\mathcal O}}w_0\bigcap M_{{\mathcal O}}=M_{{\mathcal O}}$.
Hence, the space $Q_{w_0}\backslash M_{{\mathcal O}}/C$ contains one
element, and we define $\nu_0=e$. From this we deduce that integral
\eqref{int45} is given by
\begin{equation}\label{special1}
\int\limits_{C(F)\backslash C(\A)}\varphi_\pi(t)\varphi_{\tau}
(t)dt\notag
\end{equation}
We claim that identity \eqref{id1} holds in this case. Indeed, in
this case, we have $L_{w_0}=C$ and $V=\{e\}$, where the last
equality follows from the fact that $U_{\O,w_0}=U_{{\mathcal O}}$. Recall that $\dim E_\tau(\cdot,s)=\dim \tau+\dim U$. We have
$\dim U= \dim U_{{\mathcal O}}$. Plugging this into
equation \eqref{dim01} we obtain
$$\dim C+\dim U_{{\mathcal O}}=\dim \pi+\dim \tau+\dim U_{{\mathcal O}}$$ Hence equation
\eqref{id1} holds.

We remark that integrals of this type are almost never preWhittaker.

\section{\bf Maximal parabolic subgroups of $F_4$}
\label{s:MaximalParabolicsOfF4}
The group $F_4$ has four standard maximal parabolic subgroups, $P_1, P_2, P_3$ and $P_4.$  (See section \ref{section:   Basic notations}.)
$$
\begin{array}{|l|l|l|}\hline
\text{Levi}& \text{Isomorphic to...}\\\hline
M_1&GSp_6\\
M_2&\{
  (g_1, g_2) \in GL_2 \times GL_3\mid
  \det g_1 \cdot  \det g_2 =1
  \}\\
M_3&\{
  (g_1, g_2) \in GL_3\times GL_2\mid
  \det g_1 \cdot  \det g_2^2 =1
  \}\\
M_4&\GSpin_7\\
\hline\end{array}
$$
Before working an example it will be convenient to pin down a specific
identification each standard maximal Levi subgroups with the matrix group
listed above.

We realize the group
$GSp_6$ as the group generated by $Sp_6=\{ m\in GL_6 :
\,^tmJ_6m=J_6\}$ and the one dimensional torus $\text{diag}
(a,a,a,1,1,1)$. Here
$$J_6=\begin{pmatrix} &&&&&1\\ &&&&1&\\ &&&1&&\\ &&-1&&&\\ &-1&&&&\\
-1&&&&&\end{pmatrix}$$

We fix specific isomorphisms
 by mapping\\
 $\boxed{\mathbf{M_1:}}$
  $$
    x_{\alpha_4}(r)\mapsto
    \bpm
    1&r&&&&\\
    &1&&&&\\
    &&1&&&\\
    &&&1&&\\
    &&&&1&-r\\
    &&&&&1\epm,\qquad
        x_{\alpha_3}(r)\mapsto
    \bpm
    1&&&&&\\
    &1&r&&&\\
    &&1&&&\\
    &&&1&-r&\\
    &&&&1&\\
    &&&&&1\epm
  $$
       $$x_{\alpha_2}(r)\mapsto
    \bpm
    1&&&&&\\
    &1&&&&\\
    &&1&r&&\\
    &&&1&&\\
    &&&&1&\\
    &&&&&1\epm,$$
     $$\alpha_1^\vee(t_1)\alpha_2^\vee(t_2)
\alpha_3^\vee(t_3)\alpha_4^\vee(t_4) \mapsto
   \bpm
    t_4&&&&&\\
    &t_4^{-1}t_3&&&&\\
    &&t_3^{-1}t_2&&&\\
    &&&t_2^{-1}t_3t_1&&\\
    &&&&t_3^{-1}t_1t_4&\\
    &&&&&t_4^{-1}t_1\epm
  $$
 $\boxed{\mathbf{M_2:}}$
  $$
  x_{\alpha_1}(r) \mapsto
 \left( \bpm 1&r \\ 0&1 \epm , I_3\right),\qquad
 x_{\alpha_3}(r) \mapsto
 \left( I_2, \bpm 1& r&\\ &1&\\&&1\epm
 \right),\qquad
x_{\alpha_4}(r) \mapsto
 \left( I_2, \bpm 1& &\\ &1&r\\&&1\epm
 \right),$$
$$ \alpha_1^\vee(t_1)\alpha_2^\vee(t_2)
\alpha_3^\vee(t_3)\alpha_4^\vee(t_4) \mapsto
\left( \bpm t_1&\\&t_1^{-1}t_2
 \epm , \bpm t_2^{-1}t_3 &&\\&t_3^{-1}t_4&\\ &&t_4^{-1} \epm
 \right).
  $$
\\
$\boxed{\mathbf{M_3:}}$
 $$
  x_{\alpha_1}(r) \mapsto
 \left(\bpm 1& r&\\ &1&\\&&1\epm , I_2\right),\qquad
 x_{\alpha_2}(r) \mapsto
 \left( \bpm 1& &\\ &1&r\\&&1\epm,I_2
 \right),\qquad
x_{\alpha_4}(r) \mapsto
 \left( I_3, \bpm 1& r\\ &1\epm
 \right),$$
$$ \alpha_1^\vee(t_1)\alpha_2^\vee(t_2)
\alpha_3^\vee(t_3)\alpha_4^\vee(t_4) \mapsto
\left(\bpm t_1 &&\\&t_1^{-1}t_2&\\ &&t_2^{-1} t_3^2\epm,  \bpm t_4t_3^{-1}&\\&t_4^{-1}
 \epm
 \right).
  $$
 The group $M_4$ is isomorphic to $\GSpin_7,$ as mentioned above.
 This group has an embedding into $SO_8.$  It also has a
 projection to $SO_7$ which restricts to an isomorphism on
 unipotent elements.
 We define $SO_8$ relative to the matrix $J_8$ which has
 ones on the diagonal which goes from top right to bottom left
 and zeros everywhere else.
 We define $e_{ij}$ to be the matrix with a $1$ at the $i,j$ and zeros
 elsewhere, and $e'_{ij}= e_{ij}-e'_{9-j,9-i}.$
 We fix  an isomorphism from $M_4$
 to $\GSpin_7\subset SO_8$  as follows:\\
$\boxed{\mathbf{M_4:}}$
 $$
  x_{\alpha_1}(r) \mapsto I_8 + r e_{12}' + r e_{35}', \qquad
 x_{\alpha_2}(r) \mapsto I_8 + r e_{23}',\qquad
x_{\alpha_3}(r) \mapsto I_8 + r e_{34}',
$$
$$ \alpha_1^\vee(t_1)\alpha_2^\vee(t_2)
\alpha_3^\vee(t_3)\alpha_4^\vee(t_4) \mapsto
\diag(
t_1t_4^{-1} ,
t_1^{-1} t_2 ,
t_2^{-1} t_1 t_3,
t_1^{-1} t_3,
t_3^{-1} t_1t_4,
 t_2 t_1^{-1} t_3^{-1}t_4,
 t_1 t_2^{-1}t_4 ,
 t_1^{-1} t_4^{2}).
  $$
 When working with unipotent subgroups of $M_4,$ we often find it
 more convenient to identify them with their images in $SO_7$ under
 a fixed projection.
 We define $SO_7$ relative to the matrix $J_7$ which has
 ones on the diagonal which goes from top right to bottom left
 and zeros everywhere else.
 We fix the projection $M_4 \to SO_7$ in such a fashion that
 $$
x_{\alpha_1}(r_1) \mapsto I_7 + r_1 e_{12}',
\qquad
x_{\alpha_2}(r_2)
\mapsto I_7 + r_2 e_{23}',\qquad
x_{\alpha_3}(r_3)\mapsto I_7 + r_3 e_{34}-r_3 e_{45}-\frac{r_3^2}2 e_{35},
$$
where now $e_{ij}' = e_{ij} - e_{8-j,8-i}.$
This determines the image of $x_{\alpha}(r)$ for any positive
root $\alpha,$ and we require that the projection of  $x_{-\alpha}(r)$ is $^tx_{\alpha}(r)$ for each root $\alpha.$  This then determines a unique map
 $M_4 \to SO_7$ which, as we say, is not an isomorphism but gives an
 isomorphism when restricted to any unipotent subgroup.

\section{\bf An Example}\label{An Example}
Let $H=F_4$, and consider the unipotent orbit ${\mathcal
O}=\widetilde{A}_2$. The weighted Dynkin diagram for this orbit is $$\ff0002.$$
From the  table  at the end of section \ref{section: constructing global integrals}, it follows that $C$ is the
exceptional group $G_2$. We have $P_{{\mathcal O}}=M_\O \cdot
U_{{\mathcal O}},$ with $M_\O\cong \GSpin_7.$  Since this orbit is labeled with zeros and twos
only, the integral to consider is integral \eqref{int4}.
Write $L_\O^{1,*}$ for the rational representation of $M_\O$ which is
dual to the rational representation of $M_\O$ on $L^{(1)}_\O$ given by
conjugation.  The $F$-points of this representation may be identified
with the space of characters of $\quo{U_\O}.$
The action of $M_\O$ on $L_\O^{1,*}$ may be identified with the representation of $\GSpin_7$ in which $\Spin_7$ acts by the spin representation and scalars act by scalar multiplication.
 There is a
  nondegenerate quadratic form on  $L_\O^{1,*}$ which is preserved by the action of $\Spin_7.$
 An element of $L_\O^{1,*},$ (or its dual, the space of characters) is in
 general position if its image under this quadratic form is nonzero.
 Further, the elements of $L_\O^{1,*}(F)$ in general position are a
 single $M_\O(F)$-orbit.
 See \cite{sato-kimura}, pp. 114-116, and \cite{igusa-spinor}, \S 4 especially proposition 4.

We choose
the character $\psi_{U_{{\mathcal O}}}$ as follows. For $u\in
U_{{\mathcal O}}$ write $u=x_{1111}(r_1)x_{0121}(r_2)u'$
as in subsubsection \ref{sss:conventions for defining characters}.   Then $\psi_{U_\O}(x_{1111}(r_1)x_{0121}(r_2)u')=\psi(r_1+r_2)$.
It is not hard to check that this character corresponds to a point in general position.  The embedding of $G_2$ into $\GSpin_7$ as the stabilizer of
 $\psi_{U_{{\mathcal O}}}$ is a pretty  standard embedding. A maximal unipotent of
this stabilizer is generated by all elements of the form
$$
 \begin{matrix}
 x_{1000}(r)x_{0010}(r), &\qquad
  x_{0100}(r),& \qquad
 x_{1100}(r)x_{0110}(-r),\\
 x_{1110}(r)x_{0120}(-r), &\qquad
x_{1120}(r),&\qquad
x_{1220}(r),\end{matrix}
  \qquad(r \in \G_a).
 $$
 For $1 \le i \le 4,$ we let $w[i]$ be the reflection in the Weyl group attached to
 the simple root $\alpha_i.$    We use the same notation for the
 standard representative for this reflection, which is
$x_{\alpha_i}(1)
x_{-\alpha_i}(-1)
x_{\alpha_i}(1).$
Finally, $w[i_1, i_2, \dots, i_k]$ denotes the product of the corresponding
simple reflections in the Weyl group, or of their representatives in $H(F).$
  Note that
  the Weyl group of $G_2$ may be identified with the subgroup of  that of $F_4$
which is generated by the simple reflection $w[2]$ and the product $w[1,3].$

 Thus, integral \eqref{int4} is given by
$$I=\int\limits_{G_2(F)\backslash G_2(\A)}\int\limits_{
U_{{\mathcal O}}(F)\backslash U_{{\mathcal O}}(\A)}
\varphi_\pi(g)E_\tau(ug,s)\psi_{U_{{\mathcal O}}}(u)\,du\,dg.$$ The
table at the end of section \ref{section: constructing global integrals} indicates that each of the four parabolic subgroups of $F_4$ is an option for the parabolic subgroup $P$ used to form the Eisenstein series. So we have four cases to check.  In each
case we carry out the unfolding process using the notations of
section \ref{section: unfolding global integrals}.
\subsection{$\bf P=P_1$.} First we determine the space $P(F)\backslash
H(F)/P_{{\mathcal O}}(F)$. A set of representatives is $e,\
w[1,2,3,4]$ and $w_0=w[1,2,3,2,1,4,3,2,3,4]$. Hence we obtain that
$Q_{w_0}$ is the maximal parabolic  subgroup of $\GSpin_7$ whose Levi
part is isomorphic to $GL_1\times \GSpin_5$. The group $U_{{\mathcal O},w_0}$ is generated by elements $x_\alpha(r),$ where $\alpha$ is in the set
$$\{(0001);\ (0011);\ (0111);\ (0121);\ (0122);\ (1122);\ (1222);\
(1232);\ (1242);\ (1342)\}.$$ Next we consider the space
$Q_{w_0}(F)\backslash \GSpin_7(F)/G_2(F)$.
Consider the action of $\GSpin_7(F)$ on the projective space consisting of all one-dimensional subspaces of its standard representation.  Then $Q_{w_0}(F)$ is the
stabilizer of the line consisting of all highest-weight vectors.
  Hence,
the space
$Q_{w_0}(F)\backslash \GSpin_7(F)$
may be identified with the orbit of this line.  It is not
hard to see that this orbit consists of all nonzero vectors ``of length $0$,'' i.e., on which the
$\Spin_7$-invariant quadratic form vanishes.
Now,
$Q_{w_0}(F) \bs \GSpin_7(F)/G_2(F)$ may be identified with  the
set of $G_2(F)$-orbits in this $\GSpin_7(F)$-orbit.
But again one
 easily checks that $G_2(F)$ acts transitively on the nonzero vectors of  length zero  in the standard representation.
 Thus $Q_{w_0}(F) \bs \GSpin_7(F)/G_2(F)$
 has only one element and we may take $\nu_0 =e.$

 Let $R$ denote the parabolic  subgroup
 of $G_2$ which preserves the line of highest weight vectors in the action considered above (i.e., in the standard
 representation of $G_2$).
 Thus, $R$ is the standard maximal parabolic subgroup of $G_2$ which contains the $SL_2$ which is
generated by $U_{\pm (0100)}$.
 Then in this case $L_{\nu_0}=R.$
To determine the group $V$ we first
observe that $U_{{\mathcal O}}^{w_0}$ is the group generated
by all $U_\alpha$ such that $\alpha$ is in the set $\{ (1111);\
(1121);\ (1221);\ (1231);\ (2342)\}$. When conjugating by $w_0$ we
obtain the group $V$ which is the group generated by all $U_\alpha$
such that $\alpha$ is in the set $\{ (0001);\ (0011);\ (0111);\
(0121);\ (0122)\}$.
Particular attention must be paid to $U_{1111}$ because it is not in the kernel of $\psi_U.$  We have
\begin{equation}\label{e:F4A2tP1eq1}
w_0 x_{1111}(r) w_0^{-1}
=x_{0001}(r).\end{equation}

We identify the Levi part of $P$ with the group $GSp_6$ as
in section \ref{s:MaximalParabolicsOfF4},
and compute that  $m(v)$ is embedded inside $GSp_6$ as the group $Z$
which consists of all matrices in $GSp_6$ of the form
\begin{equation}\label{uniz}
\left \{ z=\begin{pmatrix} 1&x&y\\ &I_4&x^*\\ &&1\end{pmatrix}\ :\ \
\ \ x\in Mat_{1\times 4};\ \ y\in Mat_{1\times 1}\right \}.
\end{equation}
To describe the projection of $R$ in $GSp_6$ we first notice that
$w_0x_{\pm 0100}w_0^{-1}=x_{\pm 0100}$.
Also,  $w_01010^\vee(t)w_0=1121^\vee(t).$
Here $1010^\vee(t):= \alpha_1^\vee(t)\alpha_3^\vee(t),$
and $1121^\vee$ is defined similarly.
Hence the Levi part of $R$, which is isomorphic to
the group $GL_2$ is embedded in $GSp_6$ as
$$\left \{\begin{pmatrix} |g|I_2&&\\ &g&\\ &&I_2\end{pmatrix}\ \ :\ \ \ g\in
GL_2\right \}$$ As for the unipotent radical of $R$, which we denote
by $U_R$, let
\begin{equation}\label{unig2}
u=x_{1000}(r_1)x_{0010}(r_1)x_{1100}(r_2)x_{0110}(-r_2)x_{1110}(r_3)
x_{0120}(-r_3)x_{1120}(r_4)x_{1220}(r_5)
\end{equation}
Then conjugating by $w_0$, we obtain
$$
w_0 u w_0^{-1} =
x_{1122}(r_1)x_{0010}(r_1)x_{1222}(r_2)x_{0110}(-r_2)
x_{1232}(r_3)x_{0120}(-r_3)x_{1242}(r_4)x_{1342}(r_5).
$$
Projecting into $M,$ we get
$x_{0010}(r_1)x_{0110}(-r_2)
x_{0120}(-r_3).$
Using our identification of $M$ with $6\times 6$ matrices,
we get
$$m(w_0uw_0^{-1})=\begin{pmatrix} 1&&&&&\\ &1&r_1&r_2&-r_3&&\\&&1&&r_2&\\
&&&1&-r_1&\\ &&&&1&\\ &&&&&1\end{pmatrix}$$ Combining all of the
above, integral \eqref{int45} in this case is
\begin{equation}\label{intg22}
\int\limits_{GL_2(F)\backslash GL_2({\A})}\int\limits_{U_R(F)\backslash U_R(\A)}\int\limits_{
Z(F)\backslash Z({\A})}\varphi_\pi(ut)\varphi_\tau(zm(w_0uw_0^{-1})t)\psi_Z(z)dzdudt\notag
\end{equation}
Here $\psi_Z$ is defined as follows. For $z$ as described in
\eqref{uniz}, we set $\psi_Z(z)=\psi(x_{1,1})$.
(It follows from \eqref{e:F4A2tP1eq1} that $\psi_Z(z)
= \psi_{U_\O}( w_0^{-1} z w_0).$)

We count dimensions. We have $L_{\nu_0}=R$ and hence $\dim L_{\nu_0}=9$. Also, $\dim U_{{\mathcal O}}^{w_0}=\dim V=5$. Hence $\dim L_{\nu_0} +\dim V=14$. Since $\pi$
is a generic cuspidal representation of $G_2$, its dimension is 6,
and as follows from the above table, we assumed that $\tau$ is
attached to the unipotent orbit $(42)$ of $GSp_6$ whose dimension
when divided by two is 8. Hence $\dim \pi +\dim \tau=6+8=14$. Hence \eqref{id1} holds.

We state some conclusions.  First, it is immediate
from the discussion above that integral \eqref{int45}
is of open orbit type.  Second,
 it is not hard to check
that the contribution to the integral $I$ from $e$ and $w[1,2,3,4]$
is zero, so
conjecture \ref{conj1} holds in this case.

By choosing $\tau$ to be a
certain Eisenstein series we can further unfold the last integral
and obtain an integral with a Whittaker coefficient of $\pi$. This
shows that the integral is a weakly preWhittaker integral. See
definition 5. Since we are not aware of any way to unfold the above
integral, to obtain a Whittaker coefficient of $\pi$, and this for
any representation $\tau$, we conjecture that the integral is not
preWhittaker.

\subsection{$\bf P=P_2$.} In this case the space $P(F)\backslash
H(F)/P_{{\mathcal O}}(F)$ consists of 5 elements. We can choose as
representatives the Weyl elements
$$e;\ w[2,3,4];\ w[2,1,3,2,3,4];\ w[2,3,4,2,1,3,2,3,4];\ w_0=
w[2,3,1,2,3,4,3,2,1,3,2,3,4].$$ Hence, if we consider $w_0$, then
$Q_{w_0}$ is the maximal parabolic subgroup of $\GSpin_7$ whose Levi
part $M_{w_0}$ is $GL_2\times \GSpin_3$. Also $U_{{\mathcal O}}^{w_0}$ is generated by the elements $x_\alpha(r)$ such that $\alpha\in \{ (1221);\ (1231)\}$. The
group $U_{{\mathcal O},w_0}$ is the product of the groups $U_\alpha$
for the other $13$ roots of $T$ in
 $U_{{\mathcal O}}$, which are
 $$
 0001,\;\;0011,\;\;0111,\;\;0121,\;\;0122,\;\;1111,\;\;1121,\;\;1122,\;\;1222,\;\;1232,\;\;1242,\;\;1342,\;\;2342
 $$
 Next we consider the space
$Q_{w_0}(F)\backslash \GSpin_7(F)/G_2(F)$. This space
can be identified with the set of $Q_{w_0}(F)$-orbits
in the $\GSpin_7(F)$-orbit of the character $\psi_{U_\O}.$
Recall that $\psi_{U_\O}$ is identified with
a point $x_0$ in the representation of $\GSpin_7$ in which $\Spin_7$ acts by the spin representation and scalars act by scalar multiplication.
This is an eight dimensional representation of $\GSpin_7.$
This representation
has a $\Spin_7$-invariant bilinear form $(\, , \, ).$  Assume that  $(\, , \, )$ is normalized so that  $(x_0, x_0)=1.$
  Then the $\GSpin_7$
orbit is $\mathcal{X}_\square :=\{x: (x,x)\text{ is a  square}\}.$
For a suitably chosen basis, $x_0=\,^t (0,0,0,1,1,0,0,0)$ and the image of $Q_{w_0}$
 consists of $8\times 8$ matrices of the form
$$
\bpm _tg_1^{-1}\det g_1& x & * \\ &g_1\otimes g_2 & x'\\&&g_1\det g_2
\epm, \qquad g_1, g_2 \in GL_2,
x \in \Mat_{2\times 4},
$$
where $x'$ is determined by $x,g_1$ and $g_2,$
and $g_1\otimes g_2$ is given by
$$
\bpm a&0&b&0\\0&a&0&b\\ c&0&d&0\\0&c&0&d\epm
\bpm g_2&\\&g_2 \epm , \qquad g_1 = \bpm a&b\\c&d \epm.
$$
Here and throughout, $_tA$ denotes the transpose of the matrix $A$
over the diagonal which runs from top right to lower left.  Thus
$_t \bspm a&b\\c&d \espm = \bspm d&b\\c&a \espm.$
We also introduce the notation $g^* := \, _tg^{-1}.$   Thus
$ \bspm a&b\\c&d \espm^* = \frac1{ad-bc}\bspm a&-b\\-c&d \espm.$
From this description it's easy to see that $\mathcal{X}_\square(F)$
is a union of two $Q_{w_0}(F)$-orbits:  one consisting of those elements of
the form $^t(x,y,z)$ with $x,z \in F^2, y \in F^4,$ such that
$z\ne0,$ and the other of those such that $z=0.$  Clearly, the first of these two orbits is open.
One may check that $\nu_0=w[2,1]$ is an element of it. In this case
$L_{\nu_0}=\nu_0^{-1}Q_{w_0}\nu_0\bigcap G_2=R^0$ which is defined
as follows.  Take the subgroup $R$  defined as in the case $P=P_1$.
Then $R^0 \subset R$  consists of $M_R\cong GL_2$ and all elements $u\in U_R$ as  in
\eqref{unig2} with $r_1=r_2=0$. Thus $\dim L_{\nu_0}=7$.
Thus, in this case, integral \eqref{int45} is given by
\begin{equation}\label{intg23}\begin{aligned}
\int\limits_{GL_2(F)\backslash GL_2({\A})}\int\limits_{(F\backslash \A)^5}\varphi_\pi(x_{1110}(r_1)
x_{0120}(-r_1)x_{1120}(r_2&)x_{1220}(r_3)t)\times
\\&
\varphi_{\tau_1}(t)\varphi_{\tau_2}\begin{pmatrix} 1&r_1&y_2\\
&1&y_1\\ &&1\end{pmatrix}\psi(y_1)dy_idr_jdt\end{aligned}\end{equation}
 Here we wrote
$\tau=\tau_1\times\tau_2$ where $\tau_1$ is an automorphic
representation of $GL_2(\A)$, and $\tau_2$ is an automorphic
representation of $GL_3(\A)$. (Cf. section \ref{s:MaximalParabolicsOfF4}.)
  Also, one can check that all other
elements give zero contribution to $I$. Hence conjecture \ref{conj1}
holds in this case.

We count dimensions.  We have $\dim L_{\nu_0}+\dim V=7+2=9$. According to  the table
the basic equation implies
$\dim \tau=3,$ and then formula \eqref{id1}  holds as well.

  There are two possibilities
for the pair $(\dim \tau_1, \dim \tau_2),$  namely $(1,2)$
and $(0,3).$   We now treat each of these two possibilities.

\subsubsection{$\dim \tau_1=1, \dim \tau_2=2$}
If $\dim \tau_1=1$ and $\dim \tau_2=2,$ then $\tau_1$ is
generic and $\tau_2$ is attached to the orbit $(21)$ of $GL_3.$
We will show that integral \eqref{intg23} can be further expanded and
unfolded to the Whittaker coefficient of $\pi$ with the only
assumption on $\tau$ that ${\mathcal O}(\tau_2)=(21)$. This will
prove that  integral \eqref{int45} is preWhittaker  in this case. Indeed, we have
\begin{equation}\label{eq: Four Exp GL3 in an example, P2}
\int\limits_{(F\backslash \A)^2}\varphi_{\tau_2}\begin{pmatrix} 1&r_1&y_2\\
&1&y_1\\ &&1\end{pmatrix}\psi(y_1)dy_1dy_2= \sum_{\alpha\in F}
\int\limits_{(F\backslash \A)^3}\varphi_{\tau_2}\begin{pmatrix} 1&r_1+z&y_2\\
&1&y_1\\ &&1\end{pmatrix}\psi(y_1+\alpha z)dzdy_1dy_2
\end{equation} If
$\alpha\in F^*$, then we obtain the Whittaker coefficient of
$\tau_2$ as an inner integration. This is zero from the assumption
that ${\mathcal O}(\tau_2)=(21)$. Hence we are left with the
contribution $\alpha=0$. Thus, changing variables, integral
\eqref{intg23} is equal to
\begin{equation}\label{intg34}
\begin{aligned}
\int\limits_{GL_2(F)\backslash GL_2(\A)}\int\limits_{(F\backslash \A)^6}\varphi_\pi(x_{1110}(r_1)
x_{0120}(-r_1)x_{1120}(r_2)&x_{1220}(r_3)t)\times
\\
&\varphi_{\tau_1}(t)\varphi_{\tau_2}\begin{pmatrix} 1&z&y_2\\
&1&y_1\\ &&1\end{pmatrix}\psi(y_1)dzdy_idr_jdt\end{aligned}\end{equation}
 Next we expand
$\varphi_\pi$ along the group $U_R/U_R^0$. We recall that $U_R$ is the
unipotent radical of the maximal parabolic subgroup of $G_2$ which
preserves a line, and $U_R^0$ is the unipotent radical of the group
$R^0$ defined above. The group $GL_2(F)$, as embedded in $R^0$ acts
on this expansion with two orbits. The trivial orbit contributes
zero by cuspidality. Thus, the above integral is equal to
\begin{equation}\label{intg35}
\int\limits_{B_0(F)\backslash GL_2(\A)}
\int\limits_{U_R(F)\backslash U_R(\A)}\varphi_\pi(ut)\psi_{U_R}(u)du
\varphi_{\tau_1}(t)\int\limits_{(F\backslash \A)^3}\varphi_{\tau_2}\begin{pmatrix} 1&z&y_2\\
&1&y_1\\ &&1\end{pmatrix}\psi(y_1)dzdy_idt
\end{equation}
Here $\psi_{U_R}$ is defined as follows. If we write $u$ as a
product given in \eqref{unig2}, then $\psi_{U_R}(u)=\psi(r_1)$.
Also, in integral \eqref{intg35} the group $B_0$ is the subgroup of
$GL_2$ consisting of all matrices of the form $\begin{pmatrix} a& r\\
&1\end{pmatrix}$ where $a\in F^*$ and $r\in F$. This group has the
factorization $B_0=TN$ where $T$ is a torus and $N$ is
unipotent. In terms of the group $F_4$, the
group $N$ is identified with $U_{0100}.$  Factoring the measure in integral
\eqref{intg35} we obtain
\begin{equation}\label{intg36}\begin{aligned}
\int\limits_{T(F)N(\A)(F)\backslash GL_2(\A)}
& \left[\int\limits_{U_R(F)\backslash U_R(\A)}
\iFA\varphi_\pi(ux_{0100}(r)t)\psi_{U_R}(u)
\varphi_{\tau_1}(x_{0100}(r)t)dr\,du\right]\, dt\ \times
\\ &\int\limits_{(F\backslash \A)^3}\varphi_{\tau_2}\begin{pmatrix} 1&z&y_2\\
&1&y_1\\ &&1\end{pmatrix}\psi(y_1)dzdy_i
\end{aligned}
\end{equation}
 Next we plug the
Fourier expansion of $\varphi_\pi$ along $U_{0100}$  into
the expression in brackets.   It is equal to
\begin{equation}\label{intg37}
\sum_{\alpha\in F}\int\limits_{(F\backslash \A)^2}
\int\limits_{U_R(F)\backslash U_R(\A)}\varphi_\pi(ux_{0100}(r+z)t)\psi_{U_R}(u)\psi(\alpha z)
\varphi_{\tau_1}(x_{0100}(r)t)dzdrdu\notag
\end{equation}
It follows from the cuspidality of $\pi$ that the contribution from
$\alpha=0$ is zero. Plugging the summation in integral
\eqref{intg36}, collapsing summation and integration, we obtain the
integral
\begin{equation}
\int\limits_{N(\A)(F)\backslash GL_2(\A)}\int\limits_{(F\backslash \A)^3}
W_{\varphi_\pi}(t)W_{\varphi_{\tau_1}}(t)
\varphi_{\tau_2}\begin{pmatrix} 1&z&y_2\\
&1&y_1\\ &&1\end{pmatrix}\psi(y_1)dzdy_idt\notag
\end{equation}
Here, for a vector $\vph_\sigma$ in a generic
representation $\sigma$ we denote by
$W_{\varphi_\sigma}$ its Whittaker coefficient, with the precise
generic character of the maximal unipotent subgroup
to be deduced from context. Thus integral \eqref{int45} is preWhittaker in this case.
\subsubsection{$\dim \tau_1=0, \dim \tau_2=3$}
(Equivalently, $\tau_1$ is a character, and $\tau_2$ is generic.)
In this case, we begin in the same fashion, by plugging in the
expansion \eqref{eq: Four Exp GL3 in an example, P2}.
The group $GL_2$ (once identified with the Levi of the parabolic subgroup
$R \subset C$) acts on the characters in this expansion with two orbits.
That is, $GL_2$ normalizes $U_{1110}$ and if we define
$\psi_a(x_{1110}(r)) = \psi(ar),$ then
$\psi_a( gx_{1110}(r)g^{-1}) = \psi_{a\det g}(r).$  It follows that in this
case
$I_{w_0, \nu_0}$ can be written as $I_{w_0, \nu_0}^0+ I_{w_0, \nu_0}^1,$
where
$$\begin{aligned}
I_{w_0, \nu_0}^0=\int\limits_{GL_2(F)V(\A)\backslash G_2(\A)}\int\limits_{(F\backslash \A)^3}\varphi_\pi(&x_{1110}(r_1)
x_{0120}(-r_1)x_{1120}(r_2)x_{1220}(r_3)g)
\psi(-r_1) \, dr
\times
\\
&\iiFA 3f_\tau\left(\left( I_2, \begin{pmatrix} 1&z&y_2\\
&1&y_1\\ &&1\end{pmatrix} \right)g, s\right)
\psi(y_1)\,dz\,dy_i\,dg\end{aligned}
$$

$$
\begin{aligned}
I_{w_0, \nu_0}^1=\int\limits_{SL_2(F)V(\A)\backslash G_2(\A)}\int\limits_{(F\backslash \A)^3}\varphi_\pi&(x_{1110}(r_1)
x_{0120}(-r_1)x_{1120}(r_2)x_{1220}(r_3)g)
\psi(-r_1) \, dr
\times
\\
&\iiFA 3f_\tau\left(\left( I_2, \begin{pmatrix} 1&z&y_2\\
&1&y_1\\ &&1\end{pmatrix} \right)g, s\right)
\psi(z+y_1)\,dz\,dy_i\,dg.\end{aligned}
$$
Here $V=U_{1110}U_{0120}U_{1120}U_{1220}.$
\begin{lem}
The integral $I_{w_0, \nu_0}^0$ is zero for all cusp
forms $\varphi_\pi$
and all $f_\tau \in \Ind_{P(\A)}^{H(\A)} \tau\delta_P^s,$
where $ s\in \C$ and $\tau$ is attached to the orbit $(1^2|3)$
of $GL_2\times GL_3.$
\end{lem}
\begin{proof}
The proof relies on the fact that $\varphi_\pi$ is cuspidal,
while  $f_\tau$ is invariant by $U_{1000}.$ First expand
$$
\varphi_\pi^{V}:= \iq{V} \vph_\pi(vg)\; dg
$$
along $\quo{U_{1000}U_{1100}}.$  The constant term in this
Fourier expansion is a constant term of $\varphi_\pi,$ and hence
zero.  The remaining terms are permuted transitively by $GL_2$
(embedded as a subgroup of $R$).  As a representative
we choose $\x r {1000} \x r {1100} \mapsto \psi(r_{1000}),$ and the stabilizer in $GL_2$ is the product of $U_{0100}$ and a one-dimensional torus.
Thus
$$I_{w_0, \nu_0}^0=\int\limits_{T_1(F)U_{0100}V(\A)\backslash G_2(\A)}\int\varphi_\pi^{(V', \psi_V')}(g)
\iiFA 3
f_\tau\left(\left( I_2, \begin{pmatrix} 1&z&y_2\\
&1&y_1\\ &&1\end{pmatrix} \right)g, s\right)
\psi(y_1)\,dz\,dy_i\,dg,$$
where $V' = U_{1000}U_{1100}V,$ and $\psi_V'(v')= \psi( v'_{1000}).$
Factor the integration over $\quo{U_{0100}}.$   Since $w_0\nu_0 \cdot 0100
= 1000,$ it follows that $h \mapsto f_\tau(w_0\nu_0 h,s)$ is invariant
by $U_{0100}(\A).$  And so the function $\vph_\pi \mapsto I_{w_0, \nu_0}^0$ factors through the constant term attached to the standard maximal  parabolic subgroup
of $C$ whose unipotent radical contains $U_{0100}.$  This completes
the proof.
\end{proof}

\subsection{$\bf P=P_3$.} In this case the space $P(F)\backslash
H(F)/P_{{\mathcal O}}(F)$ consists of 7 elements. We can choose as
representatives the Weyl elements
$$e;\ w[3,4];\ w[3,2,3,4];\ w[3,2,1,3,2,3,4];\ w[3,2,1,4,3,2,3,4];$$
$$w[3,2,3,4,3,2,1,3,2,3,4];\ w_0=w[3,2,1,3,2,3,4,3,2,1,3,2,3,4].$$ The
group $Q_{w_0}$ is the maximal parabolic subgroup of $\GSpin_7$ whose
Levi part is $GL_3\times GL_1$. A similar argument shows that the space $Q_{w_0}(F)\backslash
\GSpin_7(F)/G_2(F)$ is finite and we may choose $\nu_0=w[3,2,1]$.
Conjugating by $w_0\nu_0$ we obtain that $L_{\nu_0}=SL_3$ embedded
in $G_2$ as the group generated by
the groups $U_\alpha$ attached to the
 long roots. Hence, integral \eqref{int45} is given by
\begin{equation}\label{intg24}
\int\limits_{SL_3(F)\backslash SL_3(\A)}\varphi_\pi(t)\varphi_{\tau_1}
(t)dt\int\limits_{F\backslash \A}\varphi_{\tau_2}
\begin{pmatrix} 1&x\\ &1\end{pmatrix}\psi(x)dx.\notag
\end{equation}
Here $\tau=\tau_1\times \tau_2$ where $\tau_1$ is an automorphic
representation of $GL_3$ and $\tau_2$ of $GL_2$. As in the two
previous cases conjecture \ref{conj1} holds in this case, and this
is an open orbit type of integral.\\
\subsection{$\bf P=P_4$.}
This case is an instance of the special case discussed in subsection \ref{Subsection: A Special Case}.
In this case the space $P(F)\backslash
H(F)/P_{{\mathcal O}}(F)$ consists of the following 5 elements,
$$e;\ w[4];\ w[4,3,2,3,4];\ w[4,3,2,1,3,2,3,4];\
w_0=w[4,3,2,1,3,2,3,4,3,2,1,3,2,3,4]$$ Thus $Q_{w_0}=\GSpin_7$, and
hence $Q_{w_0}(F)\backslash \GSpin_7(F)/G_2(F)$ consists of one
element. Thus $\nu_0=e$, and integral \eqref{int45} is
\begin{equation}\label{intg25}
\int\limits_{G_2(F)\backslash G_2(\A)}\varphi_\pi(t)\varphi_{\tau} (t)dt\notag
\end{equation}
As before conjecture \ref{conj1} holds in this case, and this is an
 open orbit type of integral.

\section{\bf Dealing with theta functions}
\label{section: dealing with theta functions}
\subsection{A lemma on Heisenberg groups}
\label{subsection: A lemma on Heisenberg groups}
We often use the fact that some unipotent subgroup of $F_4$ is isomorphic
to a Heisenberg group.  The point of this section is to record a simple lemma regarding what choices one has to make in order to pin down a specific isomorphism.

By a Heisenberg group, we mean a two-step nilpotent, unipotent linear algebraic group $U$ such that the center $Z$ and the quotient $U/Z$ are vector groups, and the skew-symmetric form $U/Z\times U/Z \to Z$
defined by the commutator is nondegenerate.  The
standard Heisenberg group $\cH_{2n+1}$ in $2n+1$ variables is defined by equipping $\G_a^n \times \G_a^n \times \G_a$ with the operation
$$(x_1|y_1|z_1)\cdot (x_2|y_2|z_2)
 = \left(x_1+x_2\;\left|\;y_1+y_2\;\left|\;z_1+z_2 + \frac
 12 (x_1 \cdot\,_t y_2 - y_1 \cdot \, _tx_2) \right.\right.\right).$$
 (For $(x|y|z) \in \cH_{2n+1},$ we think of $x$ and $y$ as row vectors.)
 We normally write an element of $\cH_{2n+1}$
 as $(x|y|z),$ with $x,y \in \G_a^{n}$ and $z \in \G_a.$
 We shall also use the notation $(x|y|z)_{2n+1}$
 when it is desirable to reflect the size of the group in the
 notation.

\begin{lemma}\label{l:HeisenbergIsomorphisms}
If $U$ is any Heisenberg group, then $U$ is isomorphic to the standard Heisenberg group of the same dimension, and an $F$-isomorphism is uniquely determined by an ordered collection
$(x_1, \dots, x_n, y_1, \dots, y_n)$
of injective $F$-morphisms $\G_a\to U$ satisfying
$$
(x_i(r), x_j(s)) =1_U \;\;\forall i,j, \qquad (x_i(r),y_j(s)) = 1_U \;\;\forall i,j \st i+j \ne n+1,$$ $$(y_i(r),y_j(s))=1_U \;\;\forall i,j,\qquad
(x_i(r), y_{n+1-i}(s)) = (x_j(r), y_{n+1-j}(s))\ne 1_U\;\; \forall i,j,
$$
for all $r, s\in \G_a.$
(Here, $(\;,\;)$ denotes the commutator and $1_U$
denotes the identity element of the group $U.$)
\end{lemma}
The proof is straightforward, and it is omitted.
We note that an isomorphism $U \to \cH_{2n+1}$ also determines
 an isomorphism of the group of all automorphisms of $U$ which fix $Z$ pointwise with $Sp_{2n}.$
 Suppose that a Fourier coefficient is defined as in
 section \ref{section: construction of fourier coefficients},
 by exploiting the Heisenberg group structure
 of $U$ to define theta functions.  Then the group $C$
 acts on $U$ by automorphisms which fix $Z$ pointwise,
 and hence a choice of isomorphism $U \to \cH_{2n+1}$
 also determines a homomorphism $C \to Sp_{2n}.$
One may make this description more explicit in terms
of the semidirect product $\cH_{2n+1}\rtimes Sp_{2n}$:
given an isomorphism $l: U \to \cH_{2n+1},$ there
is a unique homomorphism $\rho: C \to Sp_{2n}$
such that $l(gug^{-1}) = \rho(g) l(u) \rho(g^{-1})$
for all $g \in C$ and $u \in U.$

Write $\Phi(U/Z,T)$ for the set of roots of $T$ in the quotient
$U/Z,$ i.e., for $\Phi(U) \smallsetminus \Phi(Z).$
The simplest way to choose a family of morphisms in lemma \ref{l:HeisenbergIsomorphisms}
is to
\begin{itemize}
\item pair up the roots of $T$ in $U/Z$:  for each $\alpha$ in $\Phi(U/Z,T)$ there is a unique ``partner,'' such that the sum the unique root of $T$ in $Z.$
\item order the roots in such a way that for each $i,$
the root at position $2n+1-i$ is the partner (as above)
of the root at position $i,$
\item for $i =1$ to $n,$ let $x_i = x_{\beta_i},$
where $\beta_i$ denotes the root at position $i$ in
the ordering just fixed,
\item for $i =1$ to $n,$ let $y_i(r) = x_{\beta_{2n+1-i}}(r/N({\beta_i, \beta_{2n+1-i}})).$
\end{itemize}
Here, $N({\beta_i, \beta_{2n+1-i}})$ denotes the structure constant, defined as in
\cite{Gilkey-Seitz}.
This will be our ``standard'' method of determining an
isomorphism $U \to \cH_{2n+1}$ based on listing the
roots of $T$ in $U/Z$ in a specific order.
Note that the roots of $T$ in $U/Z$ come equipped with
a partial ordering associated to the base of simple roots
for $\Phi(G,T).$  That is, $\alpha < \beta$ if $\beta - \alpha$ is a positive root.  We shall normally
order the elements of $U/Z$ in a fashion which is compatible with this partial ordering.

As an example, consider the unipotent orbit
${\bf A_1}$ in the exceptional group $F_4.$
It's diagram is $$\ff1000.$$
The corresponding parabolic subgroup $P$ is $P_1.$
The unipotent radical $U$ of this parabolic subgroup is isomorphic
to $\cH_{15}.$  Indeed, it has a one dimensional center,
which is the root subgroup $U_{2342}.$ The other $14$
roots of $T$ in $U$ are
$$
1000,\;1100,\;1110,\;1120,\;1111,\;1220,\;1121,
\;1221,\;1122,\;1231,\;1222,\;1232,\;1242,\;1342.
$$
One easily checks that for each root $\alpha$ which appears on this list, the ``partner'' $2342-\alpha$ is also present.  For example, the partner
of $1120$ is $1222.$
Also, the roots are ordered in such a way that each root and its partner
are in symmetric positions.  For example, 1120 is the fourth root,
and 1222 is fourth-last.
In order to fix an isomorphism  $l:U\to \cH_{15},$   we require that
$l( x_{2342}(z)) = (0|0|z)$ and
$$l(x_{1000}(r_1)x_{1100}(r_2)x_{1110}(r_3)x_{1120}(r_4)x_{1111}(r_5)x_{1220}(r_6)x_{1121}(r_7))=(r_1,r_2,r_3,r_4,r_5,r_6,r_7|0|0).
$$
The relevant structure constants from \cite{Gilkey-Seitz} are
$$
N(1000,1342)=-1,\quad
N(1100,1242)=1,\quad
N(1110,1232)=-2,\quad
N(1120,1222)=1,$$
$$N(1111,1231)=2,\qquad
N(1220,1122)=-1,\qquad
N(1121,1221)=-2.
$$
This means that
$$
(x_{1000}(r), x_{1342}(s)) = x_{2342}(-rs), \qquad
(x_{1121}(r), x_{1221}(s)) = x_{2342}(-2rs), \quad \text{ etc.}
$$
It follows that the preimage of $(0|y_1,y_2,y_3,y_4,y_5,y_6,y_7|0)$
under $l$ must be
$$
 x_{1221}\left(-\frac{y_1}2\right)x_{1122}(-y_2)x_{1231}\left(\frac{y_3}2\right)x_{1222}(y_4)x_{1232}\left(-\frac{y_5}2\right)x_{1242}(y_6)x_{1342}(-{y_7}).
$$
Also in this case there are
six
other orderings of the roots which are compatible with the partial ordering
inherited from our choice of a base of simple roots:
$$
1000,\;1100,\;1110,\;1120,\;1220,\;1111,\;1121,\;\dots$$
$$1000,\;1100,\;1110,\;1120,\;1111,\;1121,\;1122,\;\dots$$
$$1000,\;1100,\;1110,\;1120,\;1111,\;1121,\;1220,\;\dots$$
$$1000,\;1100,\;1110,\;1111,\;1120,\;1220,\;1121\;\dots$$
$$1000,\;1100,\;1110,\;1111,\;1120,\;1121,\;1122\;\dots$$
$$1000,\;1100,\;1110,\;1111,\;1120,\;1121,\;1220\;\dots$$
Where ``$\dots$'' indicates that the remaining seven roots are
ordered by symmetry.

In practice, it makes sense to vary the choice of ordering
(i.e., the choice of isomorphism $U_\O \to \cH_{2n+1}$)
based on the situation.

Returning to the general case, suppose that if $l_1, l_2: U \to \cH_{2n+1}$
are two isomorphisms, defined over $F,$
such that the restrictions to the center
of $U$ agree.  Define $\rho_1, \rho_2: C \to Sp_{2n}$  as before, by requiring that $\rho_i(g) l_i(u) \rho_i(g^{-1})
= l_i( gug^{-1})$ for $i=1,2.$  Then there is an element
$\sigma \in Sp_{2n}(F)$ such that $l_2(u) \rho_2(g)= \sigma l_1(u) \rho_1(g) \sigma^{-1}.$
Hence $\theta_\phi^\psi( l_1(u) \rho_1(g)) = \theta_{\omega_\psi(\sigma)\phi}^\psi( l_2(u)\rho_2(g))$ for all $\phi \in S(\A^n).$   Thus, one may vary the choice of isomorphism
$U \to \cH_{2n+1}$ fairly freely.

We shall refer to an algebraic  group $U$ as a {\bf generalized Heisenberg group} if
\begin{enumerate}
\item  it is two step nilpotent, with center $Z$ equal to
its commutator subgroup $(U,U),$
\item both $Z$ and $U/Z$ are ``vector groups,''
\item
there exists a linear form $\ell:Z \to \G_a$
such that the skew symmetric
bilinear form $U/Z \to \G_a$ induced by
composing $\ell$ with the commutator
$U \to Z$ is nondegenerate.
\end{enumerate}
Note that the existence of $\ell$ as in the last condition clearly requires $\dim U/Z$
to be even, and that any choice of $\ell$
induces a projection from $U$ onto
$U/\ker \ell,$ which is a Heisenberg group.

\begin{lem}
Let $\O$ be an orbit and define $V_\O^{(i)}$ as in section
\ref{section: construction of fourier coefficients}.
Suppose that   $\ell: V^{(2)}_\O/V_\O^{(3)} \to \G_a$ is in general position.
Then $(x,y) \mapsto \ell(xyx^{-1}y^{-1})$ is a nondegenerate skew-symmetric
form on $V^{(1)}_\O/V^{(2)}_\O.$
\end{lem}
\begin{proof}
  To explain this phenomenon, one needs to
review a bit more of the theory underpinning the construction of
Fourier coefficients given in section \ref{section: construction of fourier coefficients}.  Let $\O$ be a unipotent conjugacy class in $H.$  Then,
$\O$ can be identified via the exponential map
with an orbit for the adjoint action of $H$ on its Lie algebra
$\frak h.$  Such an orbit can be identified, via the Jacobson-Morozov
theorem, with an orbit of Lie algebra homomorphisms $\fsl_2 \to \fh.$
Let $L: \fsl_2\to \fh$ be a homomorphism in the orbit attached to $\O,$
with the property that $h_L:=L\bspm 1&0\\0&1 \espm$ lies in the Lie algebra
$\frak t$ of our fixed maximal torus $T.$
Then the weight of a simple root $\alpha_i$
 in the weighted Dynkin diagram of    $\O$
is simply the eigenvalue for $h_L$ on the one-dimensional subalgebra
$\fh_{\alpha_i}.$
Furthermore, $L$ can be chosen so that $\ell(\exp X) = \beta\left( X, f_L \right),$ for all $X\in \mathfrak{v}^{(2)}_\O$
(the Lie algebra of $V_\O^{(2)}$),
 where $f_L:=L\bspm 0&0\\1&0 \espm.$  Write $\fh_i$ for the $i$-eigenspace of $h_L$ in $\fh.$
By viewing $\fh$ as a direct sum of $L(\fsl_2)$-modules, we see that
$X \mapsto [f_L, X]$ is an isomorphism $\fh_{1} \to \fh_{-1}.$
Now, given $X \in \fh_1$ there exists $Y  \in \fh_{-1}$ such
that $\beta(X,Y)\ne 0.$
But then for $X' \in \fh_{1}$ with $Y= [f_L, X'],$ we have
$\beta(X,[f_L,X'])=\beta( [X,X'], f_L) \ne 0.$   Nondegeneracy
of the form $(x,y) \mapsto \ell(xyx^{-1}y^{-1})$  follows easily.
\end{proof}

The procedure outlined above
for fixing a specific isomorphism from a Heisenberg group
 to the standard Heisenberg group
of the appropriate dimension
may be used on generalized Heisenberg
groups to extend $\ell$ to a projection $l:U \to \cH_{2n+1}$
for suitable $n$,
 {\sl if} the function $\ell: Z \to \G_a$ is supported on a single
root subgroup $U_\alpha.$

As an example consider the unipotent subgroup $U=U_4$ of $F_4.$
It is two step nilpotent and $15$-dimensional.
The center, $Z,$ is seven-dimensional, and we have
$$
\Phi(Z,T) = \{0122,1122,1222,1232,1242,1342,2342\}.$$
The group $U/Z$ is eight-dimensional, and we have
$$\Phi(U/Z,T)=\{0001,0011,0111,1111,0121,1121,1221,1231\}
$$
Observe that for each $\alpha \in \Phi(U/Z,T),$ the root
$1232-\alpha$ is also in $\Phi(U/Z,T).$  Hence we can define a
surjective homomorphism $l: U \to \cH_9$ such that
$$
l\left(x_{0122}(r_1)x_{1122}(r_2)x_{1222}(r_3)x_{1232}(r_4)x_{1242}(r_5)x_{1342}(r_6)x_{2342}(r_7)\right) =(0|0|r_4).
$$
As in the case of an ordinary Heisenberg group, there are many such
isomorphisms, and to pin down a specific one, we may use a
choice of ordering on the roots and the homomorphisms $x_\alpha,$
as well as the associated structure constants.
For example we may order $\Phi(U/Z,T)$ as
$$
0001,0011,0111,1111,0121,1121,1221,1231.
$$
Note that for each root $\alpha,$ the ``partner'' root $1232-\alpha$
is in a symmetric position.
The relevant structure constants are
$$
N(0001, 1231)= -1,\;
N(0011, 1221)= 1,\;
N(0111, 1121)= 1,\;
N(1111, 0121)= -1.$$
So, we may define $l:U \to \cH_9$ such that
$$
l\left( x_{0001}(r_1)x_{0011}(r_2)x_{0111}(r_3)x_{1111}(r_4)\right)
=(r_1, r_2, r_3, r_4|0|0),
$$
$$
l\left( x_{0121}(r_1)x_{1121}(r_2)x_{1221}(r_3)x_{1231}(r_4)\right)
=(0|-r_1, r_2, r_3, -r_4|0).
$$
In this case, there is only one other possible ordering,
which is
$$
0001,0011,0111,0121,1111,1121,1221,1231.
$$
Note that  the above construction does not work
if $1232$ is replaced by one of the other roots in $\Phi(Z,T)$:
for each other root $\beta \in \Phi(Z,T)$ there exists $\alpha
\in \Phi(U/Z,T)$ such that $\beta-\alpha$ is not a root.
On the other hand, one can check that $$
l\left(x_{0122}(z_1)x_{1122}(z_2)x_{1222}(z_3)x_{1232}(z_4)x_{1242}(z_5)x_{1342}(z_6)x_{2342}(z_7)\right)=(0|0|z_1+z_7)$$
$$
l\left( x_{0001}(r_1)x_{0011}(r_2)x_{1111}(r_3)x_{1121}(r_4)\right)
=(r_1, r_2, r_3, r_4|0|0),
$$
$$
l\left(x_{1221}(s_1)x_{1231}(s_2)x_{0111}(s_3) x_{0121}(s_4)\right)
=(0|-2s_1, 2s_2, -2s_3, -2s_4|0),
$$
gives a well defined homomorphism  $l: U \to \cH_9.$
(Note that
$$
x_{0122}(z_1)x_{1122}(z_2)x_{1222}(z_3)x_{1232}(z_4)x_{1242}(z_5)x_{1342}(z_6)x_{2342}(z_7)\mapsto z_4
$$
and
$$x_{0122}(z_1)x_{1122}(z_2)x_{1222}(z_3)x_{1232}(z_4)x_{1242}(z_5)x_{1342}(z_6)x_{2342}(z_7)\mapsto z_1+z_7$$
are in general position for the action of $M_4 \cong \GSpin_7$ on
the space of linear maps $Z \to \G_a,$ but
$$
x_{0122}(z_1)x_{1122}(z_2)x_{1222}(z_3)x_{1232}(z_4)x_{1242}(z_5)x_{1342}(z_6)x_{2342}(z_7)\mapsto z_i
$$
is not for $i \ne 4.$)

\subsection{Explicit formulae for the Weil representation}
In this short subsection, we record for convenience some explicit
formulae for the action of the Weil representation of
$\cH_{2n+1}(\A) \rtimes \Sp_{2n}(\A)$ which we shall use repeatedly
when dealing with cases involving theta functions. A reference is \cite{Perrin}.  We write $\Mat_{n\times
n}^{0}$ for the space of $n\times n$ matrices $X$ such that $_tX=X.$

For $x,y \in \A^n$ (realized as row vectors), $z \in \A,$ $g \in
GL_n(\A), \varepsilon =\pm 1,$ and $X \in \Mat_{n\times
n}^{0}(\A):$
\begin{equation}\label{E:SchrodingerRep-Formulae}
  \omega_\psi(0|0|z)\phi = \psi(z) \phi
\qquad  \omega_\psi(x|0|0)\phi(\xi) = \phi( \xi+x)
\qquad  \omega_\psi(0|y|0) \phi(\xi) =
  \psi( y  \,_t \xi)\phi(\xi)
  \end{equation}
          \begin{equation}\label{E:WeilRepActionOfSiegelParabolic}
     \omega_\psi\left(\bpm g&\\&g^*\epm
 ,\varepsilon \right)\phi(x) = \varepsilon\gm_{\det g}
 \phi(\xi g)\qquad \qquad  \omega_\psi\bpm I&X\\&I \epm \phi(\xi)
 = \psi\left( \frac12 \xi X \,_t\xi\right)\phi(\xi), \end{equation}
 where $\gm_a$ is given in terms of the function 
 $\omega$ defined on p. 377 of \cite{P} by the formula $\gm_a=\omega(1)\omega(-1/a).$ 
 
\begin{lemma}\label{Lem: Reduction to theta function on smaller symplectic group-- partial theta}
If $k < n$ are integers, embed $Sp_{2k}
\hookrightarrow Sp_{2n}$  and $\cH_{2k+1}\hookrightarrow \cH_{2n+1}$ by
$$
g \mapsto
\bpm I_{n-k} &&\\&g&\\&&I_{n-k}\epm,
\qquad
(x|y|z)_{2k+1} \mapsto (0,x|y,0|z)_{2n+1}.
$$

This induces an embedding $\iota: \cH_{2k+1} \rtimes \Sp_{2k}(\A)
\to \cH_{2n+1} \rtimes\Sp_{2n}(\A)$ such that $\omega_\psi^n \circ
\iota = \omega_\psi^k.$ Here $\omega_\psi^n$ and $\omega_\psi^k$
denote the Weil representations of $\cH_{2n+1} \rtimes\Sp_{2n}(\A)$
and $\cH_{2k+1} \rtimes \Sp_{2k}(\A),$ respectively. If    $\phi$ is
an element of the Schwartz space $S(\A^n),$ $u \in \cH_{2n+1}(\A),$
and $ \wt g \in
\Sp_{2n}(\A),$
define
\begin{equation}\label{partialTheta}
\wt \theta_{\phi|_k}^\psi (u\wt g) :=
  \sum_{\xi \in F^{k}}
 \omega_\psi (u\wt g )\phi( 0, \xi).
\end{equation}
Then
\begin{description}
\item[(i)]
$$
\iiFA {n-k} \wt \theta_\phi^\psi
 (( 0|0, y|0) u\wt  g) dy=\theta_{\phi|_k}^\psi (u\wt g)
 \qquad (\forall \; u \in \cH_{2n+1}(\A), \wt  g \in \Sp_{2n}(\A)).
 $$
\item[(ii)] For $u \in \cH_{2k+1}(\A), \wt  g \in \Sp_{2k}(\A),$ we have
$$\wt \theta_{\phi|_k}^\psi (\iota (u\wt g)) =\wt \theta^{\psi , 2k}_{\phi_1} (u\wt g),$$
where $\wt \theta^{\psi, 2k}_{\phi_1}$ is
 the theta function  on $\cH_{2k+1} \rtimes \Sp_{2k}(\A)$ defined using the function
$\phi_1(x) := \phi(0,x),$ which lies in $S(\A^k).$
Here $x \in \A^k$ and $0$ denotes the zero element of $\A^{n-k}.$
\item[(iii)]
The function $\wt \theta_{\phi|_k}^\psi$
satisfies
 \begin{equation}\label{partial theta-- GL equivariance}
 \wt \theta_{\phi|_k}^\psi \left[\bpm g_1&&\\&I_{2k}&\\&&_tg_1^{-1} \epm \bpm I_{n-k}&X&Y\\&I_{2k}&X'\\&&I_{n-k}
 \epm u\wt g\right]
 =\gm_{\det g_1} \wt \theta_{\phi|_k}^\psi (u\wt g),
 \end{equation}for all  $g_1 \in GL_{n-k}(\A), u \in \cH_{2n+1},
 \wt g \in
\Sp_{2n}(\A),$ and for all $X \in \Mat_{(n-k)\times(2k)}(\A),$
$X' \in \Mat_{(2k)\times(n-k)}(\A),$
 $Y \in
\Mat_{(n-k)\times(n-k)}(\A),$ such that the matrix above
is in $Sp_{2n}.$
 \end{description}
\end{lemma}
\begin{proof}
This follows from the formulas for the Weil
representation given above.
\end{proof}

\subsection{Unfolding in the presence of theta functions}
\label{subsection: Unfolding in the presence of theta functions}
In this subsection, we give a discussion similar to section \ref{Subsection:
The General Process}, for an integral of the type defined  in equation \eqref{int6}.
As in
subsection \ref{Subsection:  The General Process}, we express the
integral  \eqref{int6} as a double sum indexed first by $w \in P\bs H
/ P_\O,$ and then by $\nu \in Q_{w}\bs M / C,$ with $Q_{w} = M \cap
w^{-1} P w.$ For $\nu \in M$ we define $L_\nu = H \cap (w\nu)^{-1} P
(w\nu)$ and $V=Q \cap wU w^{-1}.$ Then in this case $$\begin{aligned}
I_{w,\nu}
&=
\il_{L_\nu(F) \bs C(\A)}\il_{U^{w\nu}_\O(F) \bs U_\O(\A)}
\varphi_\pi(g) f_\tau( w\nu ug,s)
\wt \theta_\phi^\psi(l(u)\rho(g))\psi_{U_\O}( l'(u))
 \, du\, dg\\
&=
\il_{L_\nu(F) \bs C(\A)}
\varphi_\pi(g)
\il_{U^{w\nu}_\O(\A) \bs U_\O(\A)}
\iq{U^{w\nu}_\O}
 f_\tau(w\nu u_1 ug,s)
\wt \theta_\phi^\psi(l(u_1u)\rho(g))\psi_{U_\O}( l'(u_1u))
\, du_1\, du\, dg .
\end{aligned}
$$
Here $l$ and $l'$ are the projections of $U_\O/V_\O^{(3)}$
onto its Heisenberg and abelian factors as in section
\ref{section: constructing global integrals}.
The definition of ``open orbit type'' for an integral of type
\eqref{int6}, of course reflects the dimension of the theta
representation.
\begin{definition}
An integral of the type  \eqref{int6} is said to be of {\bf open
orbit type} if there is a unique $w_0 \in P \bs H/ P_\O$ and a
unique $\nu_0 \in Q_{w_0}\bs M/C$ such that
$$
\dim L_{\nu_0} + \dim U^{w_0\nu_0}_\O
= \dim \pi+ \dim \tau + \dim \Theta^\psi.
$$
It is weakly open orbit type if the number $\nu_0$ satisfying this identity
is greater than one, but the integral $I_{w_0, \nu_0}$ vanishes for
all $\nu_0$ but one.  It is preWhittaker if
\begin{enumerate}
\item it is weakly open orbit type,
\item $I_{w, \nu}$ vanishes for all pairs $(w, \nu)\ne (w_0, \nu_0),$
\item $I_{w_0, \nu_0}$ is equal to a Whittaker integral, with no conditions on $\tau$ beyond those listed in the table at the end of section \ref{section: constructing global integrals}.
\end{enumerate}
It is weakly preWhittaker if this the first two conditions hold, and $I_{w_0, \nu_0}$ is equal to a Whittaker integral, under some additional
condition on $\tau$ (e.g., that it is an Eisenstein series).
\end{definition}

In this case, we shall refer to the
integral $I_{w_0, \nu_0}$ as the ``contribution
from the open orbit,'' and, introduce a corresponding
``inner period,'' analogous to \eqref{int45}, and
given by
\begin{equation}\label{Inner Period-- Theta Function Case}
\iq{L_\nu}
\varphi_\pi(h)
\iq{U^{w\nu}_\O}
 \varphi_\tau(m( w\nu uh (w\nu)^{-1}  ))
\wt \theta_\phi^\psi(l(u)\rho(h))\psi_{U_\O}( l'(u))
\, du\, dh.
\end{equation}

The integral $I_{w,\nu}$ can be simplified
in some cases
using lemma \ref{Lem: Reduction to theta function on smaller symplectic group-- partial theta}.

It should be noted that for integrals of this type, it is sometimes
necessary to take $\pi$ to be a genuine representation on a covering group of $C(\A).$  In this case, uniqueness of local Whittaker models, and
as a consequence Whittaker integrals need not be Eulerian.
Thus, the question
of whether a given integral is preWhittaker may not be directly
relevant to the question of whether it is Eulerian.  Nevertheless, particularly in  view of \cite{WMD1},\cite{WMD2},\cite{WMD3},\cite{WMD4}, it
seems that preWhittaker integrals which involve a cusp form on a covering
group are certainly not without interest.

\section{\bf An example with a theta function}\label{sec: Example with theta}

Take $\O =\mathbf{A_1}.$
The corresponding weighted Dynkin diagram is
$$\ff1000.$$
Hence $P_\O=P_1.$
The unipotent radical
of $P_1$ is a Heisenberg group in $15$
variables. The center is $U_{2342}.$
The Levi, $M_\O$ is isomorphic to $GSp_6.$

Define a character
 of $U^{(2)}_\O$ by $\psi_{U_\O}(u) = \psi(u_{2342}).$
 The stabilizer $C$ of this character in $M_\O$ is
 the derived group of $M_\O,$ isomorphic to
 $Sp_6.$  We identify it with $Sp_6$
 using the isomorphism $M_\O \to GSp_6$
 given in section \ref{s:MaximalParabolicsOfF4}.

 Write $\Theta^\psi_7$ for the theta representation
 of $\cH_{15}(\A) \rtimes \Sp_{14}(\A),$
 and let $l: U_\O \to \cH_{15}$ be an isomorphism
 such that $l(x_{2342}(z))=(0|0|z).$  It will be convenient
 to choose different isomorphisms $l$ for different cases below.
The isomorphism $l$ also determines  a homomorphism $C \to Sp_{14}.$
 Since the action of $C$ on $U_\O/U_{2342}$ is equivalent
 to the third fundamental representation of $Sp_6,$ we denote
 this homomorphism by $\varpi_3.$
 See \cite{sato-kimura},
pp. 107-108
 \cite{igusa-spinor}, \S 5 for some discussion of this representation.
 Note that on the level
 of matrices, $\varpi_3$ depends on $l,$ which will vary from case
 to case.

The image of $Sp_6(\A)$ in $Sp_{14}(\A)$
does not split under the covering map $\Sp_{14}(\A)\to Sp_{14}(\A).$
Thus an element of $\Theta^{\psi}_7$ is a genuine function defined
on the double cover of $C(\A)\cong Sp_6(\A),$ which we denote $\bC.$
   In addition let
$\bG$ denote  the metaplectic double cover of $F_4(\A)$ defined
using the Matsumoto cocycle. Then $U(\A) \rtimes \bC \hookrightarrow
\bG.$

The options in this case were discussed in detail in section
\ref{section: constructing global integrals}.  In brief,
since each vector $\wt \theta^\psi_\phi$ in the space of $\Theta^{\psi}_7$ restricts to  a genuine function on $\Sp_6(\A),$ our global integral should include either an Eisenstein
series which is genuine and a cusp form which is not, or an Eisenstein series
which is not genuine and a cusp form which is.  If
$P=P_2$ or $P_3$ then $\dim \tau =0,$ which means that $\tau$
is a character, and this is incompatible with $E_\tau$ being genuine.  Therefore, in these cases it must be the cusp form which is genuine.
If $P=P_1$ or $P_4,$ then as far as we know both options are available.
However, our arguments will be the same regardless of which function
is genuine.  For concreteness, we assume that the cusp form is genuine
in these cases as well.  Thus, our integral is
 $$
 \int_{C(F)\bs C(\A)} \wt\varphi_{\pi}(g)
 \int_{U(F)\bs U(\A)}
 \wt\theta^\psi_\phi(l(u)\varpi_3(g)) E_\tau(ug,s)\; du \; dg,
 $$
 where $\wt \varphi_\pi$ is a cusp form in the space of some genuine
irreducible  cuspidal automorphic representation of $\Sp_6(\A),$
and $E_\tau$ is an Eisenstein series defined on the group $F_4(\A),$
and $\wt \theta_\phi^\psi$ is a theta series from the representation
$\Theta_7^\psi$ of $\cH_{15}(\A) \rtimes \Sp_{14}(\A).$

  Referring to the unfolding process sketched above, notice that
 $C$ contains every unipotent element of $M_\O.$  It follows easily that for all standard
 parabolic subgroups $P$ we have
 $$
 P \bs H/ CU_\O = P\bs G/P_\O = W(M,T) \bs W/W(M_\O,T).
 $$
 In particular, the integral is of open orbit type for all $P,$
 and one can take $\nu_0$ to be the identity in all cases.
 Since $\nu_0$ is the identity, it follows that
 $L_{\nu_0} = Q_{w_0} \cap C.$  This is a parabolic subgroup
 of $C$ and it's Levi part is equal to the intersection
 of $C$ with the Levi of $Q_{w_0}.$
 \subsection{$\bf P = P_1$}
This case is an example of the phenomenon discussed
in subsection \ref{Subsection: A Special Case}.
In this case, $w_0=w[1,2,3,2,1,4,3,2,1,3,2,4,3,2,1],$ while
$Q_{w_0}=M_\O = M_1.$  As already mentioned, the integral
\eqref{int6} is of open orbit type for this Fourier coefficient,
regardless of the parabolic subgroup used to form the Eisenstein series.
Further, $\nu_0 =$ identity, $L_{\nu_0}=M_\O = M_1,$ and $V$ is
trivial. Hence, $I_{w_0,\nu_0}$ is given by
$$
\iq{C}
\il_{U_\O(\A)}
f_\tau(w_0gu, s) \wt\varphi_\pi( g) \wt\theta_\phi^\psi
( \varpi_3(g)l(u))
\, du\, dg,
$$
while the inner period
\eqref{Inner Period-- Theta Function Case}
is given by
$$
\iq{C} \varphi_\tau(g)  \wt\varphi_\pi(g)
\wt\theta_\phi^\psi(g)\, dg.
$$

\begin{lem}\label{lem: conj one for theta case p1}
Conjecture \ref{conj1} is satisfied in this case.
\end{lem}
\begin{proof}
The space $P \bs H/ C U_\O =P\bs H/P_\O$ is represented by
     $$e,
     w[1], \;
     w[1,2,3,2,1], \;
     w[1,2,3,2,4,3,2,1],\;
     \text{ and } w_0 .$$
Let $w$ denote one of these representatives, and let
$$
I_{w}= \il_{Q^0_w(F) \bs Sp_6(\A)}
\wt \vph_\pi(g)
\il_{U_\O^w(F) \bs U_\O(\A)}
f_\tau( wug, s) \wt\theta_\phi^\psi(l(u)\varpi_3(g)) \, du\, dg.
$$
Here $Q^0_w = Sp_6 \cap w^{-1} P w.$
We must show that $I_w=0$ for all $w\ne w_0.$

If
$w=e$ or $w[1],$ then
 $w\cdot 2342\in \Phi(U,T)$ which means that
 $f_\tau(w u g, s)$ is invariant by $U_{2342}(\A)$ on the left.
 This clearly forces the integral $I_{w}$ to be zero.

If $w= w[1,2,3,2,1],$ then $Q_w^0$ is the maximal parabolic subgroup
of $Sp_6$ with Levi isomorphic to $Sp_4 \times GL_1.$
Fix an isomorphism of $U_\O$ with $\cH_{15}$ by ordering the roots as
$$
1000, 1100, 1110, 1120, 1220;  \quad 1111, 1121| 1221, 1231; \quad1122, 1222, 1232, 1242, 1342.
$$
Here, extra spacing has been used to reflect the way that
$L_\O^{(1)}$ decomposes as a direct sum of three invariant
subspaces under the action of the standard Levi subgroup
of $Q_w.$   Also,
the middle of the list is marked with a vertical
bar.  This will be standard practice.

Note that the three $M_\O$-invariant subspaces
actually form a flag which is preserved by the action of the unipotent
radical of $Q_w.$  Thus, if $R$ is the maximal standard parabolic of $Sp_{14}$ which
has Levi isomorphic to $GL_5 \times Sp_4,$ and if $\varpi_3$ is defined
by ordering the roots as above, then $\varpi_3$ will map
the Levi of $Q_w$ into the Levi of $R$ and the unipotent radical
of $Q_w$ into the unipotent radical of $R.$

Then $w\alpha \in \Phi(U,T)$ for all of the last five roots.
Therefore, by lemma \ref{Lem: Reduction to theta function on smaller symplectic group-- partial theta},
$$
I_w = \il_{Q^0_w(F) \bs Sp_6(\A)}
\wt \vph_\pi(g)
\il_{U_\O^w(F)V_1(\A) \bs U_\O(\A)}
f_\tau( wug, s) \wt\theta_{\phi|_2}^\psi(l(u)\varpi_3(g)) \, du\, dg,
$$
where $V_1$ is the five dimensional abelian unipotent subgroup
of $U$ corresponding to the last five roots listed above.

 Since both $h \mapsto f_\tau(wh)$ and $h \mapsto \wt\theta_{\phi|_2}^\psi(\varpi_3(h))$ are invariant by the unipotent radical of $Q_w^0,$ the integral $I_w$ vanishes by the cuspidality of $\wt \vph_\pi.$

 Finally, suppose $w=[1,2,3,2,4,3,2,1].$  Then $Q_{w}^0$ is the
 maximal standard parabolic subgroup of $Sp_6$ whose Levi subgroup
 is isomorphic to $GL_3.$
 In order to analyze $I_w$ in this case, we shall use a different isomorphism
 $U \to \cH_{15},$ obtained by ordering the roots as follows:
       \begin{equation}\label{F4A1P3-U->H}
1000;
\;\;
1100,
1110,1111,1120,1121,1122|\;\;1220,1221,1222,1231,1232,
1242; \;\; 1342.
\end{equation}
Thus
$$\begin{aligned}
&l(x_{1000}(r_1)x_{1100}(r_2),...,x_{1122}(r_7))=(r_1,r_2,r_3,r_4,r_5,r_6,r_7|0|0)
\qquad l(x_{2342}(z)) = (0|0|z),\\
&l(x_{1220}(y_1)x_{1221}(y_2) x_{1222}(y_3)
x_{1231}(y_4)x_{1232}(y_5) x_{1242}(y_6) x_{1342}(y_7))\\
&\hskip 3in=(0|y_1,-2y_2,y_3,2y_4,-2y_5,y_6,-y_7|0).
\end{aligned}$$

Since $w\cdot 1342=1342 \in \Phi(U,T),$ lemma
\ref{Lem: Reduction to theta function on smaller symplectic group-- partial theta} implies that
$$
I_w = \il_{Q^0_w(F) \bs Sp_6(\A)}
\wt \vph_\pi(g)
\il_{U_\O^w(F)U_{1342}(\A) \bs U_\O(\A)}
f_\tau( wug, s) \wt\theta_{\phi|_6}^\psi(l(u)\varpi_3(g)) \, du\, dg.
$$
Next we unfold the partial theta function
$$
\wt\theta_{\phi|_6}^\psi(h) = \sum_{\xi \in F^6}
[\omega_\psi(h) \phi]( 0, \xi).
$$
The standard Levi subgroup of $Q_w^0,$ isomorphic to $GL_3,$
acts on $F^6$ (which may be identified with the abelian
subgroup of $U_\O$ corresponding to the six roots
$1100,\dots 1122$) by a representation which is equivalent to the
symmetric square representation of $GL_3.$
Choose a collection $S$ of orbit representatives, and for each representative $\xi$ let $O_\xi$ denote the stabilizer in $GL_3.$
Let $N$ denote the unipotent radical of $Q_w^0.$
Then we have
$$
I_w = \sum_{\xi \in S}
\il_{O_\xi(F)N(F) \bs Sp_6(\A)}
\wt\vph_\pi(g)
\il_{U_\O^w(F) U_{1342}(\A)\bs U_\O(\A)}
[\omega_\psi( l(u) \varpi_3(g) ) \phi](0, \xi)
f_\tau( w ug,s)\, du \, dg.
$$
Let $N'(F)$ denote the six-dimensional abelian subgroup
of $U_\O(\A)$ corresponding to the roots
$1220,1221,1222,1231,1232,
1242.$  Then $wN'w^{-1} =N.$    One the other hand,
$wNw^{-1}$ lies in $U.$ Factoring the integration on $N$ and $N',$
we find that
$$
I_w = \sum_{\xi \in S}
\il_{O_\xi(F)N(\A) \bs Sp_6(\A)}
\wt\vph_\pi^{(N, \psi_{\xi})}(g)
\il_{U_\O^w(F)N'(\A)U_{1342}(\A)\bs U_\O(\A)}
[\omega_\psi( l(u) \varpi_3(g) ) \phi](0, \xi)
f_\tau^{(N, \psi'_\xi)}(w ug,s)\, du \, dg,
$$
where
$$
\wt\vph_\pi^{(N, \psi_{\xi})}(g) =\iq{N} \vph_\pi(ng)\psi_\xi(n)\, dn,
\qquad
f_\tau^{(N, \psi'_\xi)}(w ug,s) =\iq{N} f_\tau(nw ug,s) \psi_\xi'(n)\, dn,
$$
and $\psi_\xi$ and $\psi_\xi'$ are two characters of $N$ which
depend on $\xi \in F^6.$   We need to calculate this dependence
precisely.

First, take
$$
n'(y) = x_{1220}(y_1) x_{1221}(y_2) x_{1222}(y_3) x_{1231}(y_4) x_{1232}(y_5) x_{1242}(y_6)
$$
Then
$$
l(n'(y)) = (0|y_1,-2y_2,y_3,2y_4,-2y_5,y_6,-y_7|0),
$$
so
$$
[\omega_\psi( n'(y)h) \phi_1]( 0,\xi)
= \psi( y \,_t\xi) = \psi( y_1\xi_7-2y_2\xi_6+y_3\xi_5+2y_4\xi_4-2y_5\xi_3 + y_6\xi_2),$$
for  $\xi = (\xi_2, \dots, \xi_7) \in F^6.$
On the other hand
$$
wn'(y) w^{-1} = x_{0100}(y_1) x_{0110}(y_2)
x_{0120}(y_3) x_{0111}(y_4) x_{0121}(y_5)x_{0122}(y_6),
$$
which is identified with
$$
\bpm I_3&Y \\ &I_3 \epm \in Sp_6, \qquad Y = \bpm y_4&-y_5&y_6\\
-y_2& y_3& -y_5 \\ y_1& -y_2& y_4 \epm.
$$
Thus
$$
\psi_\xi'\bpm I_3&Y \\ &I_3 \epm
=\psi( \Tr( \Xi \cdot Y)), \qquad \text{ where }\Xi = \bpm \xi_4&\xi_6&\xi_7\\ \xi_3&\xi_5 & \xi_6\\
\xi_2&\xi_3&\xi_4 \epm.
$$

In order to describe $\psi_\xi,$ one needs to compute the restriction
of $\varpi_3$ to $N.$
Let
$$e_{ij}'=e_{ij}-e_{15-j,15-i}, \qquad
e_{ij}'' = e_{ij}+e_{15-j,15-i}.$$
Here $e_{ij}$ denotes the $14\times 14$ matrix with a $1$ at $i,j$ and $0$'s everywhere else.
Thus $Sp_{14}$ contains $I_{14}+re_{ij}'$ if $1 \le i,j\le 7,$
$ I_{14}+re_{ij}''$ if $1 \le i \le 7 < j < 15-i,$ and $I_{14}+re_{ij}$ if $i+j=15.$
(Here, $I_{14}$ is the $14\times 14$ identity matrix.)
We have
\begin{equation}\label{Example With a Theta; P2 : varpi3 of SiegelParabolic}
\begin{aligned}
\varpi_3( x_{0110}(r)) = I_{14}-re'_{13}-r^2e''_{18}-2r e''_{38}+2re''_{49},&\qquad
\varpi_3( x_{0100}(r)) = I_{14}+re'_{12}+re''_{58}-2e_{69},
\\ \varpi_3( x_{0111}(r)) = I_{14}-re'_{14}-r^2e''_{1,10}+2re''_{39}-2re''_{4,10}&\qquad
\varpi_3( x_{0120}(r)) = I_{14}+re_{15}'+re_{28}''-2re_{4,11}\\\varpi_3( x_{0121}(r)) = I_{14}+re_{16}'-r^2e_{1,13}''-2re_{29}''+2re_{3,11}''&\qquad
\varpi_3( x_{0122}(r)) = I_{14}+re_{17}'+re_{2,10}''-2re_{3,12}
\end{aligned}
\end{equation}
It follows from \eqref{E:WeilRepActionOfSiegelParabolic},\eqref{Example With a Theta; P2 : varpi3 of SiegelParabolic},  that
$$\begin{aligned}&
\omega_\psi\left( \varpi_3\left[
    \bpm I_3&X\\&I_3\epm g \right]
    \right)\phi( 0, \xi_2, \dots , \xi_7) \\&\qquad
    =\psi(
-2 (\xi_4 \xi_5 - \xi_3 \xi_6) x_{1,1}
+2 (\xi_2 \xi_6 -\xi_3 \xi_4) x_{1,2}
-( \xi_3^2- \xi_2 \xi_5 )x_{1,3}
\\&\hskip 2in
+2( \xi_3 \xi_7 -\xi_4 \xi_6) x_{2,1}
-( \xi_4^2
- \xi_2 \xi_7) x_{2,2}
- (\xi_6^2- \xi_5 \xi_7 )x_{3,1}
 )\\&\qquad\qquad \times
   \omega_\psi\left( \varpi_3\left(
    g \right)
    \right)\phi( 0, \xi_2, \dots , \xi_7).\end{aligned}
$$
Thus
$\psi_\xi \bpm I_3&X\\&I_3\epm =
\psi(-\Tr(\Xi^{\ad} X))$ where $\Xi$ is as above,
and $\Xi^{\ad}$
is  the matrix whose $i,j$ entry is the determinant of the $i,j$ minor
of $\Xi.$

Now, if $\Xi$ is trivial or of rank one, then $\Xi^{\ad}$ is trivial, and hence
$\wt \vph_\pi^{(N, \psi_\xi)}$ vanishes identically.
On the other hand, if $\Xi$ is of rank three, then $f_\tau^{(N, \psi'_\xi)}$
is a Fourier coefficient of $f_\tau,$ which is attached to the
orbit $(2^3).$  According to the table, $\tau$ is attached to $(2^21^2),$
so $f_\tau^{(N, \psi'_\xi)}$ vanishes in this case.

This leaves the set of symmetric matrices of rank two.  This set
 is a single orbit
under the action of $GL_3,$ and we may choose the matrix
$\Xi$ such that $\xi_6=1$ and the rest are zero as a representative.

We claim that $f_\tau^{(N, \psi_\xi)}$ is invariant by the group $U_{0010}U_{0011}$ for this choice of $\xi.$  To see this, consider the Fourier
expansion
 $$f_\tau^{(N, \psi_\xi)}(g,s)
 = \sum_{a,b \in F} f_\tau^{(N'', \psi_{\xi, a,b})}(g,s), $$
$$ f_\tau^{(N'', \psi_{\xi, 0,b})}(g,s):=\iiFA 2 f_\tau^{(N, \psi_\xi)}(x_{0010}(r_1)x_{0011}(r_2) g,s) \psi(ar_1+br_2) \, dr.$$
If $a$ and $b$ are both nonzero, then $f_\tau^{(N'', \psi_{\xi, a,b})}$ is a Fourier coefficient attached to the orbit $(42)$ of $Sp_6.$
Such a coefficient vanishes identically on the space of $\tau.$
If one of $a,b$ is zero and the other is nonzero, then after
conjugating by $w[3]$ if necessary, we may assume that
 $a=0$ and $b$ is nonzero.  Then one may rewrite $f_\tau^{(N'', \psi_{\xi, a,b})}$ as an iterated integral with the inner integral being
$$
f_\tau^{(V, \psi_V)}(g,s):=\iq{V} f_\tau(vg,s) \psi_V(v) \, dv,
\qquad V = \left\{
\bpm I_2&X&*\\&I_2&X' \\ &&I_2\epm : X = \bpm x_{11} & x_{12}\\ 0 & x_{22} \epm
\right\}\subset Sp_6,
$$
$$
\psi_V\bpm I_2&X&*\\&I_2&X' \\ &&I_2\epm=\psi(bx_{1,1}+x_{2,2}).
$$
Now, let $V'$ be the group defined in the same manner as $V$
but with $X$ being arbitrary.  Then every extension of $\psi_V$
to a character of $V'$ is in general position.  Integration
  over $V'$ against a character in general position is a Fourier
coefficient attached to the orbit $(3^2).$ Such a Fourier coefficient
vanishes identically on the space of $\tau.$  From this we deduce
that $f^{(V, \psi_V)}$ vanishes identically, and thence
so does $f_\tau^{(N'', \psi_{\xi, 0,b})}.$

This leaves only the constant term in the expansion of $f_\tau^{(N, \psi_\xi')}$ along $U_{0010}U_{0011},$ which proves that $f_\tau^{(N, \psi_\xi')}$
is invariant by this group.
This permits us to express $I_w$ as an iterated integral
with the inner integral being the
 constant term of $\wt \vph_\pi$
along the maximal parabolic subgroup of $Sp_6$ with Levi
isomorphic to $GL_2 \times SL_2.$  Thus $I_w$ vanishes.
\end{proof}

 \subsection{$\bf P = P_2$}
The basic data in this case is as follows:
$$\begin{aligned}
 w_0&=w[2,3,2,1,4,3,2,1,3,2,4,3,2,1]\\
 M_{w_0}&\cong GL_3\times GL_1  \qquad (\Delta\cap \Phi(M_{w_0},T)=\{ \alpha_3, \alpha_4\}),
 \\
 \nu_0 &= \text{ identity,}\\
 L_{\nu_0}&= C \cap Q_{w_0}\\
 &=\text{the maximal parabolic subgroup of $Sp_6$ with Levi}\cong GL_3 \\
 m(w_0 L_{\nu_0}w^{-1}_0 ) &\mapsto
 \left\{ \left(\bpm \det g^{-1}&\\&1\epm,\; g\right): g\in GL_3 \right\}
  \qquad ( M \hookrightarrow GL_2 \times GL_3)\\
  U^{w_0} &= U_{1342}\\
 V&= U_{1000}\mapsto
 \left\{ \bpm 1&x \\ &1 \epm , I_3\right\}
 \qquad ( M \hookrightarrow GL_2 \times GL_3)\\
 \dim \tau &=0.
  \end{aligned}$$
Here, $M_{w_0}$ is the Levi factor of $Q_{w_0}.$

Fix an isomorphism $l:U \to \cH_{15}$
as described in section \ref{section: dealing with theta functions},
based
    on ordering the roots as in \eqref{F4A1P3-U->H}.
 We consider the restriction of the corresponding embedding $\varpi_3:Sp_6 \to Sp_{14}$ to the parabolic subgroup $L_{\nu_0}.$
First the Levi factor of $L_{\nu_0}$ is  isomorphic to $GL_3,$ and we may
write
$$
\varpi_3(g) = \diag(\det g, \sym^2(g) , \sym^2(g)^*, \det g^{-1}),
$$
where the $6\times 6$ matrix $\sym^2(g)$ denotes the matrix of $g$ acting on the symmetric square representation for a suitable choice of basis.

Since $\tau$ is a character in this case,
    $f_\tau$ will be left-invariant by $V.$
     Hence, the integral $I_{w_0,\nu_0}$ for this case is
    $$
    \il_{U_{1342}(\A)\bs U_\O(\A)}
    \il_{(P_2\cap C)(F) \bs C(\A)}
    \wt\varphi_\pi(g)
    f_\tau( w_0 ug, s)
    \iFA
    \wt\theta_\phi^\psi(l(x_{1342}(r) u)\varpi_3(g))
    \, dr
    \, dg
    \, du.
    $$

   By  lemma \ref{Lem: Reduction to theta function on smaller symplectic group-- partial theta},
    part (i)
    $$I_{w_0, \nu_0}=
    \il_{(P_2\cap C)(F) \bs C(\A)}\il_{U_{1342}(\A)\bs U_\O(\A)}
    \wt\varphi_\pi(g)
    f_\tau( w_0 ug, s)
    \wt\theta_{\phi|_6}^\psi(l(u)\varpi_3(g))
    \, du\, dg,
    $$
   where $\wt\theta_{\phi|_6}^\psi$ is defined as in Lemma
    \ref{Lem: Reduction to theta function on smaller symplectic group-- partial theta}.
The corresponding inner period \eqref{Inner Period-- Theta Function Case} is given by
$$ \begin{aligned}
 \iq{GL_3}\iq{\Mat_{3\times 3}^{0}}
    \wt\varphi_\pi&\left(
    \bpm I_3&X\\&I_3\epm \bpm g&\\&_tg^{-1} \epm
    \right)
    \tau(w_0gw_0^{-1})\\&
    \wt\theta_{\phi|_6}^\psi\left(\varpi_3\left[
    \bpm I_3&X\\&I_3\epm \bpm g&\\&_tg^{-1} \epm \right]
    \right) \, dX
    \, dg,\end{aligned}
$$
(Keep in mind that $\tau$ is a character.)
Identify $\Mat_{3\times 3}^{0}$ with the
unipotent radical of $C \cap P_2$ via the mapping $X\mapsto \bspm I_3&X \\ &I_3\espm$ and the isomorphism
$Sp_6 \to C.$  Also, identify $GL_3$ with
the Levi of $C \cap P_2$
via the mapping $g \to \bspm g&\\& _tg^{-1}\espm$
and the isomorphism
$Sp_6 \to C.$

Observe that for $g_1 \in GL_3(F), \wt g \in \Sp_{14}(\A),$ if $\xi' := (\xi_2, \dots , \xi_7)\cdot \sym^2(g_1),$ then
$$\begin{aligned}
\omega_\psi\left( \varpi_3\left[
    \bpm I_3&X\\&I_3\epm \wt g \right]
    \right)\phi( 0, \xi')
    =\omega_\psi\left( \varpi_3\left[\bpm g_1&\\& _tg_1^{-1} \epm \bpm I_3&X\\&I_3\epm \wt g \right]
    \right)\phi( 0,\xi_2, \dots , \xi_7)&
    \\=
    \omega_\psi\left( \varpi_3\left[\bpm I_3&g_1X\, _tg_1\\&I_3\epm  \bpm g_1&\\& _tg_1^{-1} \epm   \wt g \right]
    \right)\phi( 0,\xi_2, \dots , \xi_7).&
\end{aligned}$$
Hence, if we write $\Xi * g_1$ for the matrix analogous to $\Xi$ corresponding to the vector\\
$(\xi_2, \dots , \xi_7)\cdot \sym^2(g_1),$ then we have
$$
\psi( \Tr( ( \Xi*g_1)^{\ad} X))
= \psi( \Tr(\Xi^\ad  (g_1X\,_tg_1)))
= \psi( \Tr((\,_tg_1\Xi^\ad g_1)X))
=\psi( \Tr((g_1 \Xi  \,_tg_1)^\ad X)).
$$
It follows that
the polynomial identity
$\Xi*g_1=g_1 \Xi  \,_tg_1$
holds
for all $g_1 \in GL_3(F),\Xi \in \Mat^{0}_{3\times 3}(F).$

Now unfold the theta function, identify $(\xi_2, \dots, \xi_7) \in F^6$ with $\Xi \in \Mat^{0}_{3\times 3}(F),$
and split the sum over $\Xi$ up into orbits for the
action of $GL_3(F).$  It is clear that $\rank(\Xi)$ is an invariant for this action.
Define
$$\begin{aligned}
I_{w_0, \nu_0, \Xi}=
\il_{\Stab_{\Xi}(F) \Mat_{3\times 3}^{0}(\A)\bs Sp_6(\A)}\il_{U_{1342}(\A)\bs U_1(\A)}&
\iq{\Mat_{3\times 3}^{0}}
    \wt\varphi_\pi\left(
    \bpm I_3&X\\&I_3\epm g\right)
    \psi(\Xi^{\ad} X)
    \, dX\\&
    f_\tau( w_0 ug, s)
    \omega_\psi(l(u)\varpi_3(g))\phi( 0, \xi_2, \dots, \xi_7)
    \, du\, dg,
    \end{aligned}$$
where $\Stab_{\Xi}$ denotes the stabilizer of $\Xi$
in $GL_3.$
\begin{prop}\label{Example with theta: vanishing of orbits for rank < 3}
The integral
$I_{w_0, \nu_0, \Xi}$
vanishes unless $\Xi$ is of rank three.
\end{prop}
\begin{proof}
If $\Xi$ is trivial or of rank one, then $\Xi^{\ad}$ is trivial and  the assertion is obvious.
Suppose then that $\Xi$ is of rank two.
We may assume that
$\xi_2=\xi_3=\xi_4=\xi_6=0,$  since each $GL_3(F)$-orbit
contains elements with this property.
The stabilizer $\Stab_\Xi$ then contains the two-dimensional unipotent subgroup $U_{0010}U_{0011}$ which corresponds to
$$
\left\{ \bpm 1&0&r_2\\0&1&r_1\\0&0&1\epm \right\}
\subset GL_3.
$$
Since
$$
\varpi_3(x_{0010}(r_1)x_{0011}(r_2))
=I+r_1e'_{2,3}+r_2e'_{24}+r_1^2e'_{25}+r_1r_2e'_{26}
+r_2^2 e_{27}'
+2r_1 e'_{35}+r_2e'_{36}+r_1e'_{46}+2r_2e'_{47},
$$
it follows that
$\omega_\psi(l(u)\varpi_3(g))\phi( 0, 0,0,0,\xi_5, 0,\xi_7)$
is invariant by $U_{0010}U_{0011}.$
Since $\tau$ is a character, it follows that $f_\tau$
is also invariant by $U_{0010}U_{0011}.$
Further, the conditions we have placed on $\xi$
imply that $\Tr(\Xi^{\ad} X)=-\xi_5\xi_7 x_{3,1}.$
Factoring the integration over $\quo{U_{0010}U_{0011}},$
we obtain the constant term of $\wt\varphi_\pi$ along the
parabolic subgroup of $Sp_6$ with Levi $GL_2\times SL_2$
as in inner integral.  Hence, $I_{w_0, \nu_0, \Xi}$ vanishes in the rank two case as well.
\end{proof}
In view of proposition \ref{Example with theta: vanishing of orbits for rank < 3}, we may write
$$I_{w_0, \nu_0}=\sum_\Xi I_{w_0, \nu_0,\Xi}$$
with the sum being over representatives for the distinct
orbits of nondegenerate symmetric matrices.
This is similar to an integral where $P\times U_\O C$ has an open orbit $\mathcal X$
in $H$  such that $\mathcal X(F)$ is a union of infinitely many $P(F) \times U_\O(F) C(F)$-orbits.  Further, there is no apparent reason why any of
the terms in the above sum should vanish.

For the sake of completeness, we verify conjecture \ref{conj1}
in this case as well.
\begin{lem} Conjecture \ref{conj1}  holds in this case.
\end{lem}\begin{proof}
There are seven elements of
$P\bs H / CU_\O,$ represented by the Weyl elements
$$e,\;
     w[2,1],\;
     w[2,3,2,1],\;
     w[2,3,2,4,3,2,1],\;
     w[2,1,3,2,4,3,2,1],\;$$
$$     w[2,1,3,2,1,3,2,4,3,2,1],\;
     w[2,3,2,1,4,3,2,1,3,2,4,3,2,1].$$
     Define $I_w$ as in lemma \ref{lem: conj one for theta case p1}.
The first  three Weyl elements listed above map $2342$ into the unipotent radical of $U.$  It easily follows that $I_w$ is zero for all choices of data in these cases.

To study the other cases, let $Q_w^0 = C \cap Q_w,$ (a parabolic subgroup of $Sp_6$).  For $w= [2,3,2,4,3,2,1],$ the group
$Q_w^0$ has Levi isomorphic to $GL_3.$ Identify
$U_\O$ with $\cH_{15}$ using \eqref{F4A1P3-U->H}.  Then
the last seven roots are all mapped into $\Phi(U,T).$
As in lemma \ref{Lem: Reduction to theta function on smaller symplectic group-- partial theta}, this kills every nonzero term in the
sum over $\xi \in F^7$ which defines $ \theta_\phi^\psi,$ leaving a
function which is invariant by the full maximal unipotent
subgroup of $Sp_6.$  Factoring the integration
over the unipotent radical of $Q_{w},$ we obtain the constant
term of $\wt\vph_\pi$ along this unipotent radical.  This proves that $I_w=0$
in this case.

If $w=w[2,1,3,2,4,3,2,1],$ then
the situation is similar.  Indeed,
$w\cdot \alpha$ is positive for
$\alpha = 1342, 1242, 1232, 1222, 1122, 1221,1222.$
From the table, we know that $\tau$ is a character in this case,
and hence, $f_\tau$ is again invariant by the seven dimensional
unipotent group corresponding to these seven roots.

Finally, we consider $w=w[2,1,3,2,1,3,2,4,3,2,1].$
For this case, the standard Levi subgroup $Q_w^0$ of $L_w^0$
is isomorphic to $GL_2\times SL_2.$  Further, $w\cdot \alpha>0$
for $\alpha =1231,1232,1242,1342.$  Let $V_1$ denote the product
of the groups $U_\alpha$ for these four roots.
Define $l$ and $\varpi_3$ by ordering the roots as follows:
$$
1000, 1100, 1110, 1111;\;\; 1120, 1121, 1122| 1220, 1221, 1222;\;\; 1231, 1232, 1242, 1342
$$
Then using lemma \ref{Lem: Reduction to theta function on smaller symplectic group-- partial theta}, one obtains
$$
I_w = \il_{Q_w^0(F)\bs Sp_6(\A)}
\wt \vph_\pi(g)
\il_{V_1(\A)\bs U_\O(\A)}
\wt\theta_{\phi|_3}^\psi(l(u)\varpi_3(g)) f_\tau(ug,s)
\, du
\,dg,
$$
further, $g\mapsto \wt\theta_{\phi|_3}^\psi(\varpi_3(g))$ is invariant
by the unipotent radical of $Q_w^0,$ as is $g\mapsto f_\tau(wg, s).$
Factoring the integration over this group, we obtain zero in this case as well.
\end{proof}

 \subsection{$\bf P = P_3$}
 The basic data for this case is
 $$\begin{aligned}
 w_0&=w[3,2,1,4,3,2,1,3,2,4,3,2,1]\\
 M_{w_0}&\cong GL_2\times GL_2 \qquad (\Delta\cap \Phi(M_{w_0},T)=\{ \alpha_2, \alpha_4\}), \\
  \nu_0 &= \text{ identity,}\\
 L_{\nu_0}&= C \cap Q_{w_0}\\
 &=\text{the maximal parabolic subgroup of $Sp_6$ with Levi}\cong GL_2\times SL_2 \\
 m( w_0\nu_0 L_{\nu_0} \nu_0^{-1}w_0^{-1})
 &\mapsto\left\{\left(
 \bpm g_1&\\&1\epm,g_2\right): g_1, g_2 \in GL_2, \det g_1=\det g_2^{-2}
 \right\},
 \qquad ( M \hookrightarrow GL_3\times GL_2)\\
 U^{w_0}&=U_{1242}U_{1342}\\
 V&= U_{0100}U_{1100}\\
 &\to
 \left\{ \left( \bpm 1&&x_1\\&1&x_2\\&&1\\ \epm, \; I_2 \right)\right\} \qquad ( M \hookrightarrow GL_3\times GL_2)
  \end{aligned}$$
  \begin{lem}\label{conj 1 for theta case P3}
  Conjecture \ref{conj1} holds in this case.
  \end{lem}
  \begin{proof}
  Define $I_w$ as in lemma \ref{lem: conj one for theta case p1}.
  The set $P\bs H/ C U_\O$ has five elements, $w_0$ and the
  following four others:
  $$e, \;\;
     w[3,2,1], \;\;
     w[3,2,4,3,2,1], \;\;
     w[3,2,1,3,2,4,3,2,1].
$$
Once again, by the table, we know that $\tau$ is a character.
The first three coset representatives above map $2342$ to a positive
root.  It follows that $I_w$ vanishes in all these cases.

Now assume $w=w[3,2,1,3,2,4,3,2,1].$  Then $w\alpha >0$
for $\alpha =1122,1222,1231,1232,1242,1342.$  Order
the roots of $T$ in $U_\O/(U_\O, U_\O)$ so that these six roots come last.
Then
$$
I_w =\il_{Q_w^0(F)\bs Sp_6(\A)}
\wt \vph_\pi(g)
\il_{U_\O^w(\A)\bs U_\O(\A)}
f_\tau(wug,s)\theta_{\phi|_1}^\psi(l(u)\varpi_3(g))
\, du
\, dg,
$$
and both $g\mapsto f_\tau(wg,s)$ and $g\mapsto \theta_{\phi|_1}^\psi(\varpi_3(g))$ are invariant by the unipotent radical of $Q_w^0,$
which proves that $I_w$ is zero, since $\wt \vph_\pi$ is cuspidal.
  \end{proof}

  We now proceed to analyze $I_{w_0}.$
Using once again the fact that $\tau$
is a character,  we obtain
 $$
 I_{w_0, \nu_0}= \il_{U_{1242}U_{1342}(\A)\bs U_\O(\A)}
    \il_{(P_3\cap C)(F) \bs C(\A)}
    \wt\varphi_\pi(g)
    f_\tau( w_0 ug, s)
     \wt\theta_{\phi|_5}^\psi(ug)
    \, dg
    \, du.
$$

Define $l:U\to\cH_{15}$ and $\varpi_3:Sp_6 \to Sp_{14}$
using
 the
 ordering given in equation \eqref{F4A1P3-U->H}.
  Identify $L_{\nu_0}$ with $Sp_6$ using the isomorphism  fixed in section \ref{s:MaximalParabolicsOfF4}.
  Recall that $L_{\nu_0}$ is a standard parabolic subgroup of $Sp_6$ in this case.  The  Levi factor of $L_{\nu_0}$
  is isomorphic to $GL_2 \times SL_2$ and consists of
  all matrices of the form $\diag( g_1, g_2,\,_tg_1^{-1}).$
  We identify this group with $GL_2 \times SL_2$
  via the map $(g_1, g_2) \mapsto \diag( g_1, g_2,\,_tg_1^{-1}).$
   For elements of this group $\varpi_3$ is given by
 $$
 \diag( g_1, g_2,\, _tg_1^{-1} )
 \mapsto \diag(\det g_1 \cdot g_2,\ g_1,\ \rho(g_1,g_2),
 \,_tg_1^{-1} ,\ \det g_1^{-1} \cdot _tg_2^{-1}),
 \qquad (g_1\in GL_2, g_2 \in SL_2).
 $$
 Here,
  $\rho(g_1,g_2)$ is a certain $6 \times 6$ matrix.

Plug in
$$
 \wt\theta_{\phi|_5}^\psi(ug)
 =
\sum_{\xi \in F^5}
\left[\omega_\psi\left(
ug\right)\phi\right](0,0, \xi_3, \dots, \xi_7),
$$
and
split $\xi$ up into $\xi' = (\xi_3, \xi_4) \in F^2$ and $\xi'' = (\xi_5,\xi_6, \xi_7) \in F^3.$
For $g$ in the $GL_2$ factor of the Levi of $L_{\nu_0},$
we have
$
\varpi_3(g) = \diag(\det g,\det g ,g, s^2(g),\, _ts^2(g)^{-1},\, _tg^{-1} ,\det g^{-1} ,\det g^{-1}),$  where
$$
s^2\bpm a&b\\ c&d\epm=
\frac1{ad-bc}
\left(
\begin{array}{ccc}
 a^2 & a b & b^2 \\
 2 a c & b c+a d & 2 b d \\
 c^2 & c d & d^2
\end{array}
\right)
$$
Thus $GL_2(F)$ acts on $\{ \xi' = (\xi_3, \xi_4): \xi_3, \xi_4\in F\}$ with two orbits.
The contribution from the trivial orbit
is
\begin{equation}\label{F4A1P3-TrivialOrbit}
I_{w_0, \nu_0, 0}=
\il_{U_{1242}U_{1342}(\A)\bs U_1(\A)}
    \il_{(P_3\cap C)(F) \bs C(\A)}
    \wt\varphi_\pi(g)
    f_\tau( w_0 ug, s)
     \wt\theta_{\phi|_3}^\psi(ug)
    \, dg
    \, du.
\end{equation}
To see that this is equal to zero, one checks that
$$
\left\{\bpm I_2&*&* \\ &I_2&* \\ &&I_2\epm \right\} \subset Sp_6
\text{ maps into }
\left\{\bpm I_2&*&*&*&*\\
&I_2&*&*&*\\
&&I_6&*&*\\
&&&I_2&*\\
&&&&I_2\epm\right\}\subset Sp_{14},
$$
under the embedding $\varpi_3:Sp_6 \to Sp_{14},$  defined
using the ordering given in \eqref{F4A1P3-U->H}.
It follows from lemma \ref{Lem: Reduction to theta function on smaller symplectic group-- partial theta} that the function $\wt\theta_{\phi|_3}^\psi(ug)$ is
invariant by this group, and,
since $\tau$ is a character, $f_\tau( w_0 ug,s)$ is invariant as well.  It follows that
the mapping $\wt \vph_\pi \mapsto I_{w_0, \nu_0, 0},$ where $I_{w_0, \nu_0, 0}$ is defined by equation
\eqref{F4A1P3-TrivialOrbit},
factors through the constant term of $\wt\varphi_\pi,$
and hence vanishes.

To describe the contribution from the open orbit,
let $S_{0,1}$ denote the stabilizer of
$\xi'=(0,1)$ in $GL_2$
(still identified with a subgroup of $L_{\nu_0}$ as above).
It consists of all matrices of the form
$\bspm \alpha&*\\&1\espm.$  Also,
   $S_{0,1}\bs GL_2$ may be identified with $P_2^0\bs L_{\nu_0},$ where
$$P_2^0 := \left\{\bpm
\alpha&*&*&*&*\\&1&*&*&*\\&&h&*&*\\&&&1&*\\&&&&\alpha^{-1}\epm
\in Sp_6: \alpha \in GL_1, \; h \in SL_2
\right\}.$$
  It follows that
$$
\sum_{\xi' \in F^2\smallsetminus\{(0,0)\}}\; \sum_{\xi'' \in F^3}
\omega_\psi(\wt g)\phi_1 (0,0,\xi', \xi'')
=
\sum_{\gamma \in P_2^0(F)\bs L_{\nu_0}(F)}
\sum_{\xi'' \in F^3}
\omega_\psi( \gamma \wt g)\phi_1 ( 0,0,0,1,\xi''),
$$
for all $\wt g \in \wt{Sp}_{14}(\A)$ and $\phi_1 \in S(\A^7).$
Plugging in,
$$
I_{w_0, \nu_0}=
\int_{P_2^0(F) \bs Sp_6(\A)}
\int_{U'(\A)} \wt\varphi_\pi(g)
\sum_{\xi'' \in F^3}
\omega_\psi ( l(u')\varpi_3(g))
\phi( 0,0,0,1,\xi'') f_\tau( w_0 u'g,s)  du'\, dg.
$$
Here $U'=U_{1242}U_{1342}\bs U_\O.$

Next, it follows from \eqref{E:WeilRepActionOfSiegelParabolic} and \eqref{Example With a Theta; P2 : varpi3 of SiegelParabolic}
that
\begin{equation}\label{pqr}
\omega_\psi \left(\varpi_3 \bpm
1&0&0&0&q&r\\
0&1&0&0&p&q\\
0&0&1&0&0&0\\
0&0&0&1&0&0\\
0&0&0&0&1&0\\
0&0&0&0&0&1
\epm
\right)
\phi_1\left( 0,0,0,1,\xi''\right)
=
\psi(-p)\phi_1\left( 0,0,0,1,\xi''\right),
\end{equation}
for all $p,q,r \in \A,$  all $\xi'' \in F^3,$
and all $\phi_1 \in S(\A^7).$

 We shall write $Z$ for the three dimensional unipotent group considered in \eqref{pqr}, because it is the center of the unipotent radical of $L_{\nu_0}.$
 Then  $h \mapsto f_\tau( w_0h, s)$ is left-invariant by $Z(\A).$ It follows that   $I_{w_0, \nu_0}$ is equal to
\begin{equation}\label{F4A1P3-Iw0nu0-One}
\begin{aligned}
\il_{P_2^0(F) Z(\A)\bs Sp_6(\A)}
& \left[\iq{Z}
\wt\varphi_\pi( z g) \psi_Z(z) \; dz\right]\times \\&
\il_{U'(\A)}
\left[\sum_{\xi'' \in F^3} \omega_\psi( l(u') \varpi_3(g))
\phi(0,0,0,1,\xi'')\right]
f_\tau( w_0u'g,s)
\,du'\,dg,
\end{aligned}
\end{equation}
where $\psi_Z(z)=\psi(-p)$ in the coordinates of  \eqref{pqr}.

Now, the function
$\iq Z
\wt\varphi_\pi( z g) \psi_Z(z) \; dz$
is invariant by
$Y(F),$ where
\begin{equation}\label{example with theta: P3: def of Y}
Y = \left\{y=
\bpm
1&0&a&b&q&r\\
0&1&0&0&p&q\\
0&0&1&0&0&b\\
0&0&0&1&0&-a\\
0&0&0&0&1&0\\
0&0&0&0&0&1
\epm \in Sp_6
\right\}
\end{equation}
Hence we can apply
Fourier expansion on the group $Y(F)Z(\A)\backslash Y(\A).$
The subgroup of $P_2^0$
 consisting of matrices of the form $\diag( \alpha ,1 , h, 1,\alpha^{-1}) : \alpha \in GL_1, h \in SL_2$
acts on the space of characters of $Y(F)Z(\A)\backslash Y(\A)$ with 2 orbits.
We claim that the trivial orbit contributes zero.
The term corresponding to the trivial
character is
\begin{equation}\label{example with theta: P3: expansion on Y: contribution from trivial}\begin{aligned}
\int_{P_2^0(F) Z(\A)\bs Sp_6(\A)}
\int_{U'(\A)} \iq{Y}&
\wt\varphi_\pi(y g) \psi_Y^0(y) \; dy\\
\sum_{\xi'' \in F^3} &\omega_\psi( l(u') \varpi_3(g))
\phi(0,0,0,1,\xi'')
f_\tau( w_0u'g,s)
\,du'\,dg,
\end{aligned}
\end{equation}
where $\psi_Y^0$ is the trivial extension of
$\psi_Z$ to a character of $Y.$  In the coordinates
of \eqref{example with theta: P3: def of Y}, it is given by
$\psi_Y^0(y)= \psi( -p).$
\begin{lem}
The integral \eqref{example with theta: P3: expansion on Y: contribution from trivial} vanishes identically.
\end{lem}
\begin{proof}  In the proof we shall repeatedly use the
formulae for $\omega_\psi$ given in \eqref{E:SchrodingerRep-Formulae}
and
\eqref{E:WeilRepActionOfSiegelParabolic}.
Set $\wt\varphi_\pi^{(Y, \psi_Y^0)}(g)= \iq Y
\wt\varphi_\pi( y g) \psi_Y^0(y) \; dy.$
Factor the integration over $Y(F)Z(\A)\bs Y(\A),$
which may be identified
with $\quo{U_{0011}U_{0111}}.$  Note that $f_\tau$
and $\wt\varphi_\pi^{(Y, \psi_Y^0)}$ are both
invariant by this group on the left.
One computes
$$
\omega_\psi(\varpi_3( x_{0111}(r_3)x_{0011}(r_2)))\phi(0,0,0,1,\xi_5, \xi_6,\xi_7)
= \psi(-2r_3\xi_5) \phi(0,0,0,1,\xi_5, \xi_6, \xi_7+2r_2),
$$z
deducing that \begin{equation}\label{example with theta p3, psiY0 lemma}
\begin{aligned}
\iiFA2 \sum_{\xi'' \in F^3}
\il_{U'(\A)}&
\omega_\psi( l(u')\varpi_3(x_{0111}(r_3)x_{0011}(r_2) g))\phi
(0,0,0,1,\xi_5, \xi_6, \xi_7)\\&
f_\tau( w_0 u' x_{0111}(r_3)x_{0011}(r_2) g, s)
\wt\varphi_\pi^{(Y, \psi_Y^0)}(x_{0111}(r_3)x_{0011}(r_2) g)
\, du'
\, dr \\
=\iiFA2 \sum_{\xi'' \in F^3}
\il_{U'(\A)}&
\omega_\psi( \varpi_3(x_{0111}(r_3)x_{0011}(r_2)) l(u')\varpi_3(g))\phi
(0,0,0,1,\xi_5, \xi_6, \xi_7)\\&
f_\tau( w_0 x_{0111}(r_3)x_{0011}(r_2) u'  g, s)
\wt\varphi_\pi^{(Y, \psi_Y^0)}(x_{0111}(r_3)x_{0011}(r_2) g)
\, du'
\, dr \\
=\iiFA2 \sum_{\xi'' \in F^3}
\il_{U'(\A)}&
 \psi(-2r_3\xi_5)
\omega_\psi(  l(u')\varpi_3(g))\phi
(0,0,0,1,\xi_5, \xi_6, \xi_7+2r_2)\\&
f_\tau( w_0 x_{0111}(r_3)x_{0011}(r_2) u'  g, s)
\wt\varphi_\pi^{(Y, \psi_Y^0)}(x_{0111}(r_3)x_{0011}(r_2) g)
\, du'
\, dr.\end{aligned}\end{equation}
As indicated in the table above, in this case $\tau$ is a character.
Since $w_0 \cdot 0111$ and $w_0 \cdot 0011$ are both positive, it
follows that $f_\tau( w_0 x_{0111}(r_3)x_{0011}(r_2) u'  g, s)=
f_\tau( w_0u'  g, s).$
Also it follows from the definition of $\wt\varphi_\pi^{(Y, \psi_Y^0)}$
that $\wt\varphi_\pi^{(Y, \psi_Y^0)}(x_{0111}(r_3)x_{0011}(r_2) g)=
\wt\varphi_\pi^{(Y, \psi_Y^0)}(g).$
Collapsing the summation on $\xi_7$ with the integration
on $r_2,$ and
applying Fourier inversion in $r_3$ and $\xi_5,$ we find that \eqref{example with theta p3, psiY0 lemma} equals
$$
\begin{aligned}
\il_{U'(\A)}\il_{\A}\sum_{\xi_6 \in F}&\omega_\psi( l(u')\varpi_3( g))\phi
(0,0,0,1,0, \xi_6, r_2)\,dr_2\, f_\tau( w_0 u'  g, s)
\wt\varphi_\pi^{(Y, \psi_Y^0)}(g)
\, du'
\end{aligned}
$$
Next, $\varpi_3(x_{0001}(r)) = I+re_{34}'+re_{56}'+2re_{67}'+r^2e_{57}',$ hence
$
\omega_\psi(\varpi_3(x_{0001}(r)))\phi(0,0,0,1,0,\xi_6,r_2)
=\phi(0,0,0,1,0,\xi_6,r_2-2r\xi_6),
$ and so
$$
\il_{U'(\A)}\il_{\A}\sum_{\xi_6 \in F}\omega_\psi( l(u')\varpi_3( g))\phi
(0,0,0,1,0, \xi_6, r_2)\,dr_2\, f_\tau( w_0 u'  g, s)\,du'
$$
is invariant by $U_{0001}.$
Factoring the integration, one obtains
an expression for \eqref{example with theta: P3: expansion on Y: contribution from trivial} as an iterated integral,
such that the inner integration is
the constant
term of $\wt\varphi_\pi$ along the unipotent radical of the parabolic
subgroup of $Sp_6$ whose Levi part is $GL_1\times Sp_4.$
 It then follows that the
integral \eqref{example with theta: P3: expansion on Y: contribution from trivial} is zero.
\end{proof}

We continue to study the expansion of \eqref{F4A1P3-Iw0nu0-One}
along $Y(F)\backslash
Y({\A)}.$
Next
we consider the contribution from the nontrivial orbit. Choose the character  $\psi_Y (y)= \psi( a-p)$
(in terms of the coordinates given in \eqref{example with theta: P3: def of Y}) as a representative for this orbit.
Then the stabilizer inside
$GL_1\times SL_2$ as embedded in $Sp_6$ is the semidirect product of
$T_1=\text{diag}\ (\alpha,1,\alpha,\alpha^{-1},1,\alpha^{-1})$, and
the unipotent group $I_6+re_{3,4}$. Thus the  integral \eqref{F4A1P3-Iw0nu0-One} is equal
to
\begin{equation}\label{F4A1P3E1}
\begin{aligned}
\int_{Z(\A)T_1(F)U_{\max}^{Sp_6}(F)\bs Sp_6(\A)}
\int_{U'(\A)}&
\int_{Y(F)\bs Y(\A)}
\wt\varphi_\pi(yg)\psi_Y(y)
\; dy \\
\sum_{\xi'' \in F^3} &\omega_\psi( l(u') \varpi_3(g))
\phi( 0,0,0,1,\xi'')
f_\tau( w_0u'g,s)
 \,du'\,dg.
\end{aligned}
\end{equation}
Here, $U_{\max}^{Sp_6}$ is the standard maximal
unipotent subgroup of $Sp_6.$
Next we factor the integration over
$U_{0111}(F) \bs U_{0111}(\A).$
As a subgroup of
$Sp_6$ this group is
$$\left\{
\bpm
1&0&0&m&0&0\\
0&1&0&0&0&0\\
0&0&1&0&0&m\\
0&0&0&1&0&0\\
0&0&0&0&1&0\\
0&0&0&0&0&1
\epm
\right\}.$$
Note that
$$
\wt\varphi_\pi^{(Y, \psi_Y)}(g):=
\int_{Y(F)\bs Y(\A)}
\wt\varphi_\pi(yg)\psi_Y(y)
\; dy $$
is invariant by this group, as is $f_\tau.$
Using \eqref{Example With a Theta; P2 : varpi3 of SiegelParabolic}, one can check that
$$
\omega_\psi( \varpi_3(x_{0111}(m)))
\phi( 0,0,0,1,\xi_5,\xi_6,\xi_7)
= \psi(-2m\xi_5)
\phi( 0,0,0,1,\xi_5,\xi_6,\xi_7).
$$
Integration over $m$ then picks of the term $\xi_5=0.$
Hence \eqref{F4A1P3E1} equals
$$
\il_{Z_1(\A)T_1(F)U_{\max}^{Sp_6}(F)\bs Sp_6(\A)}
\il_{U'(\A)}
\wt\varphi_\pi\FC Y(g)\sum_{\xi'' \in F^2} \omega_\psi( l(u') \varpi_3(g))
\phi( 0,0,0,1,0,\xi'')
f_\tau( w_0u'g,s) \,du'\,dg,
$$
where $Z_1 = U_{0111}Z \subset Y.$
Next, we consider the group $L:=U_{0010}U_{0011}.$
We compute
$$\varpi_3(x_{0010}(\xi_6)x_{0011}(\xi_7/2))=\bpm h&\\&_th^{-1}\epm,\quad \text{ where }
h=\bpm
1&0&0&0&0&0&0\\
&1&\xi_6&\xi_7/2&\xi_6^2&\xi_6\xi_7/2&\xi_7^2/4\\
&&1&0&2\xi_6&\xi_7/2&0\\
&&&1&0&\xi_6&\xi_7\\
&&&&1&0&0\\
&&&&&1&0\\
&&&&&&1\\
\epm,
$$
Hence, by \eqref{E:WeilRepActionOfSiegelParabolic},
$$
\sum_{\xi_6, \xi_7 \in F}
\omega_\psi( \varpi_3(g)) \phi
(0,0,0,1,0,\xi_6, \xi_7)
=
\sum_{\gamma \in U_{0010}U_{0011}(F)}
\omega_\psi(\varpi_3(\gamma g)) \phi
(\xi_0),
$$
where $\xi_0=(0,0,0,1,0,0,0).$
Collapsing summation with integration, the
above integral is equal to
\begin{equation}\label{example with theta: P3: main integral, expression 1}
\int_{Z_1(\A)T_1(F)V_1(F)\bs Sp_6(\A)}
\int_{U'(\A)}
\wt\varphi_\pi^{(Y,\psi_Y)}(g) \omega_\psi( l(u') \varpi_3(g))
\phi( \xi_0)
f_\tau( w_0u'g,s)
 \,du'\,dg,
\end{equation}
where
$$
V_1 =
\left\{
\bpm
1&a&0&b&c&d\\
&1&0&e&f&*\\
&&1&g&*&*\\
&&&\ddots
\epm\in Sp_6
\right\}$$
(a subgroup of $U_{\max}^{Sp_6}$ complementary to $U_{0010}U_{0011}$).  Here, and in what follows, we exploit the fact that an element
$u$ of $U_{\max}^{Sp_6}$ is determined by the entries $u_{ij}$ with $1 \le i < j \le 7-i.$
Now, it is easily checked that the function
$h \mapsto f_\tau(w_0h)\; (h \in F_4(\A))$ is  left $V_1(\A)$
invariant.  (Here, we identify $V_1$ with a subgroup
of $F_4$ using the identification of $M_4\subset F_4$ with $GSp_6$
given in section \ref{s:MaximalParabolicsOfF4}.)
Using \eqref{E:WeilRepActionOfSiegelParabolic}
and \eqref{Example With a Theta; P2 : varpi3 of SiegelParabolic} one can check that
that the  function $h \mapsto \omega_\psi(h) \phi(\xi_0) \; (h \in \Sp_6(\A))$ is invariant by the groups $U_{0001}(\A),U_{0100}(\A),$ and $U_{0110}(\A).$  It then follows that
$$g \mapsto \int_{U'(\A)} f_\tau( w_0 u' g,s) \omega_\psi(l(u')\varpi_3(g))\phi(\xi_0)\, du'$$ is also invariant by $U_{0001}(\A),U_{0100}(\A),$ and $U_{0110}(\A).$
 The product
of these three groups may be identified with
$Z_1(\A)\bs V_1(\A).$
  Therefore, when we  factor the integration over $V_1(F) Z_1(\A)\bs V_1(\A),$
  the integral \eqref{example with theta: P3: main integral, expression 1} is equal to
$$
\int_{T_1(F) V_1(\A) \bs Sp_6(\A) }
\int_{Y_2(F) \bs Y_2(\A)}
\wt\varphi_\pi( y_2 g) \psi_{Y_2}(y_2) \, dy_2
\int_{U'(\A)}
\omega_\psi(u' \varpi_3(g))\phi(\xi_0)
f_\tau(w_0 u' g) \, du' \, dg,
$$
where
$$Y_2:= \left\{
\bpm
1&a&h&b&c&d\\
&1&0&e&f&*\\
&&1&m&*&*\\
&&&\ddots
\epm\in Sp_6
\right\},\qquad
\psi_{Y_2} \bpm
1&a&h&b&c&d\\
&1&0&e&f&*\\
&&1&m&*&*\\
&&&\ddots
\epm = \psi(h-f).
$$
Using the left
invariance property of $\wt\varphi_\pi$ under elements in $Sp_6(F)$, we
deduce that
$$\begin{aligned}
\wt\varphi_\pi^{(Y_2, \psi_{Y_2})}(g)
&:=
\int_{Y_2(F) \bs Y_2(\A)}
\wt\varphi_\pi( y_2 g) \psi_{Y_2}(y_2) \, dy_2\\
&=
\int_{Y_2(F) \bs Y_2(\A) }
\wt\varphi_\pi(y_2 w[3] g)
\psi_{Y_2}' \, (y_2) dy_2,
\end{aligned}
$$
where $w[3]$ is our standard
representative
in $N_G(T)$ for the simple reflection
$s_{\alpha_3} \in W(G,T),$ and
$\psi_{Y_2}'(y_2) = \psi_{Y_2}( w[3]^{-1} y_2 w[3]).$  One checks that
$$
\psi_{Y_2}' \bpm
1&a&h&b&c&d\\
&1&0&e&f&*\\
&&1&g&*&*\\
&&&\ddots
\epm = \psi(a-g).
$$
Finally, the function
$\wt\varphi_\pi^{(Y_2, \psi_{Y_2}')}$
is invariant by $U_{0010}(F).$
We plug in the Fourier expansion along this group.  The term corresponding to the trivial character vanishes by cuspidality.  The remaining characters are permuted transitively by the action of a one dimensional torus which is conjugate under $w[3]$ to the torus $T_1.$
Thus, in the end we obtain
$$
\int_{V_1(\A) \bs Sp_6(\A)}
\int_{U'(\A)}
 W_{\widetilde\varphi_\pi}(w[3]g)
\omega_\psi( u' \varpi_3(g)) \phi(\xi_0)
f_\tau(w_0u'g) du' dg.
$$
We conclude that this case is preWhittaker.

 \subsection{$\bf P = P_4$}
 The basic data in this case is
$$\begin{aligned}
 w_0&=w[4,3,2,1,3,2,4,3,2,1]\\
 M_{w_0}&\cong GL_1\times GSp_4 \qquad (\Delta\cap \Phi(M_{w_0},T)=\{ \alpha_2, \alpha_3\}), \\
 \nu_0 &= \text{ identity,}\\
  L_{\nu_0}&= C \cap Q_{w_0}\\
 &=\text{the maximal parabolic subgroup of $Sp_6$ with Levi}\cong GL_1\times Sp_4 \\
 &= Sp_4^{\{\alpha_2, \alpha_3\}} \cdot \alpha_4^\vee(GL_1)  \\
   m( w_0 L_{\nu_0} w_0^{-1} )&=Sp_4^{\{\alpha_2, \alpha_3\}}\cdot 2321^\vee(GL_1)
    \subset M\cong \GSpin_7\\
 U^{w_0} &= U_{1122}U_{1222}U_{1232}U_{1242}U_{1342}\\ V&=U_{1000}U_{1100}U_{1110}U_{1120}U_{1220}\\
 &= \text{the unipotent radical of the maximal parabolic subgroup of}  \\
 &\phantom{=} \text{$M \cong \GSpin_7,$ with Levi $\cong GL_1\times \GSpin_5.$}
  \end{aligned}$$
  Here $2321^\vee$ denotes the cocharacter of $T$
  given by
  $$
  2321^\vee(t) = \alpha_1^\vee(t^2)\alpha_2^\vee(t^3) \alpha_3^\vee(t^2) \alpha_4^\vee(t),\; t\in \G_m.
  $$
  \begin{lem}
  \label{conjecture for theta case P4}
  Conjecture \ref{conj1} holds in this case.
  \end{lem}
\begin{proof}
The set $P\bs H/ CU_\O$ has only three elements, represented by the
identity, $w_0,$ and $w:=[4,3,2,1].$  The identity maps $2342$ into $U$
which proves that the term corresponding to this double coset is zero.
 Let $Q_{w}^0= Sp_6 \cap w^{-1} P w.$  It is the standard maximal parabolic subgroup of $Sp_6$ with Levi subgroup isomorphic to $GL_2\times SL_2.$

 Fix the homomorphisms $l:U_\O \to \cH_{15}$ and $\varpi_3:Sp_6 \to Sp_{14}$ by ordering the roots of $U_\O/(U_\O,U_\O)$ as follows
 $$
 1000, 1100;\; \;1110, 1111; \; 1120, 1220, 1121| 1221, 1122, 1222; \; 1231, 1231; \; 1242, 1342.
 $$
 The last four roots are sent to $\Phi(U,T)$ by $w.$  Write $V_w$ for
 the corresponding unipotent subgroup.  Then applying lemma \ref{Lem: Reduction to theta function on smaller symplectic group-- partial theta}
 we have
 $$
 I_w =
 \il_{Q_w^0(F) \bs Sp_6(\A)}
 \wt\vph_\pi(g)
 \il_{U_\O^w(F)V_w(\A)\bs U_\O(\A)}
 f_\tau(wug,s)\wt\theta_{\phi|_{3}}^\psi(l(u)\varpi_3(g))
 \, du\,dg.
 $$
 Further, both $g \mapsto f_\tau(wg,s)$ and $g\mapsto \wt\theta_{\phi|_{3}}^\psi(\varpi_3(g))$ are invariant by the unipotent radical of $Q_w^0,$
 so $I_w$ vanishes by cuspidality.\end{proof}

To analyze $I_{w_0, \nu_0},$ it is more convenient to
 define the isomorphism $l: U\to \cH_{15},$
and the embedding $\varpi_3: Sp_6 \to Sp_{14}$ by
 ordering of the roots of $T$ in $U/(U,U),$ thus:
$$
1000, 1100, 1110, 1120, 1220;  \quad 1111, 1121| 1221, 1231; \quad1122, 1222, 1232, 1242, 1342.
$$(Recall that changing the isomorphism $l$ changes the bijection
$\phi \to \wt\theta_\phi^\psi$ between Schwartz functions and theta functions, but \ul{not} the space of functions obtained.)
It follows  that
\begin{equation}l^{-1}( 0|y|0)
=x_{1221}\left(-\frac{y_1}2\right)x_{1231}\left(\frac{y_2}2\right)x_{1122}(-y_3)x_{1222}(y_4)x_{1232}\left(-\frac{y_5}2\right)x_{1242}(y_6)x_{1342}(-y_7),
\label{eq: U -> cH15, example with theta, P4}
\end{equation}
for $y=(y_1, y_2, \dots , y_7).$
Having identified $C$ with $Sp_6,$ the group $C\cap M_2$ is identified
with $\{ \diag( \alpha, g, \alpha^{-1}) \subset Sp_6: \alpha \in GL_1, g \in Sp_4\}.$
If $g \in Sp_4,$ we may write
$\varpi_3(g) = \diag( \wedge_0^2 (g) , g', \,_t(\wedge^2_0(g))),$ where
$\wedge_0^2 (g)$ denotes matrix for $g$ acting on the five-dimensional
fundamental representation of $Sp_4$ with respect to a suitable basis,
and $g' = tgt^{-1},$ for $t=\diag(-2,-2,1,1).$

We may then write
$$
I_{w_0, \nu_0}=
\il_{U_{w_0}(\A)\bs U(\A)}
\il_{(C\cap P_2)(F)\bs C(\A)}
\iiFA5\wt\varphi_{\pi}(g)
\wt \theta_\phi^\psi (
(0|0,y|0)u_0
 \varpi_3(g) ) f_\tau\left(v(y) w_0u_0 g,s\right)\, dy
\, dg\, du_0,
$$
where
\begin{equation}\label{F4A1P4 def of v(y)}
v( -y_1, -y_2, -y_3, -y_4, -y_5)=x_{1000}(y_1) x_{1100}(y_2) x_{1110}\left(\frac{y_3}2\right) x_{1120}(-y_4) x_{1220}(y_5)
\end{equation}
satisfies $l(w_0^{-1} v(y) w_0)=(0|0,0,y|0)\in \cH_{15},$ for $y=(y_1, \dots, y_5).$
The image of $v(y)$
under the
isomorphism $M\to \GSpin_7$ fixed in section \ref{s:MaximalParabolicsOfF4} and the
natural homomorphism from $\GSpin_7$
onto $SO_7$ is $$
\bpm
1&y_1&y_2&y_3/2&y_4&y_5&*\\
&1&&&&&-y_5\\
&&1&&&&-y_4\\
&&&1&&&-y_3/2\\
&&&&1&&-y_2\\
&&&&&1&-y_1\\
&&&&&&1
\epm \in SO_7.
$$
We now expand $f_\tau$ along the group $V.$
The group $Sp_4^{\{\alpha_2, \alpha_3\}}(F)\cdot 2321^\vee(GL_1(F))$ acts on the characters of $\quo V$ with infinitely many orbits.
Indeed, the space of characters of $\quo V$ may be identified
with the $F$ points of the rational representation of $Sp_4^{\{\alpha_2, \alpha_3\}}(F)\cdot 2321^\vee(GL_1(F))$
which is dual to the quotient of $V$ by its commutator subgroup
$(V,V).$  We denote this space by $V/(V,V)^*.$  Then
$2321^\vee(GL_1)$ acts linearly on $V/(V,V)^*,$
and
the action of $Sp_4^{\{\alpha_2, \alpha_3\}}$ on $V/(V,V)^*$ can be identified with the action of
$SO_5$ on its standard representation (regarding $Sp_4$ as $\Spin_5$).
Identify $V/(V,V)^*$ with column vectors and let $Q$ denote the nondegenerate
$\Spin_5$-invariant quadratic form on this space.
Then the orbits for $Sp_4 \cdot GL_1$ are as follows:  one orbit consists of the zero vector, and
a second corresponds to all  vectors $ x \ne 0$ such that $Q(x)=0.$
The vectors $x$ with $Q(x) \ne 0$ split into infinitely many orbits
parametrized by square classes in $F^\times.$  That is,
  two vectors $x_1, x_2$ with $Q(x_1)Q(x_2) \ne 0$
lie in the same orbit if and only if $Q(x_1)/Q(x_2)$ is a square.
The zero vector of course corresponds to the trivial character.
As representatives for the remaining orbits we take
$$
\psi_V^0(v(y)) = \psi( y_1), \qquad\psi_V^1(v(y)) = \psi( y_3), \qquad \text{ and }\qquad
\{ \psi_V^{a}(v(y))= \psi(y_2-ay_4)\},
$$
where $a$ ranges over a set of representatives for the nonzero, nonsquare  square
classes in $F.$  (If $a$ is a square, then $\psi_V^a,$ defined as
above, is equivalent to $\psi_V^1.$)
The stabilizer of $\psi_V^0$
 in $Sp_4^{\{\alpha_2, \alpha_3\}}(F)\cdot 2321^\vee(F^\times)$  is $SL_2^{\alpha_3}(F)\cdot2421^\vee(F^\times)U_{0100}U_{0110}U_{0120}(F),$
with the cocharacter $2421^\vee$ being defined in the
same way that $2321^\vee$ was defined above.
The space $SL_2^{\alpha_3}(F)\cdot2421^\vee(F^\times)U_{0100}U_{0110}U_{0120}(F)\bs Sp_4^{\{\alpha_2, \alpha_3\}}(F)\cdot 2321^\vee(F^\times)$
can be identified with $P_3^0(F) \bs (C\cap P_2)(F),$ where
$$
P_3^0 \mapsto
\left\{
\bpm a&*&*&*&*\\
&a&*&*&*\\
&&h&*&*\\
&&&a^{-1}&*\\
&&&&a^{-1}
\epm a \in GL_1, h \in SL_2
\right\}\subset Sp_6.
$$
The stabilizer of $\psi_V^1$ in $Sp_4^{\{\alpha_2, \alpha_3\}}(F)\cdot 2321^\vee(F^\times)$
is isomorphic to the semidirect product of $SL_2\times SL_2$ and a group of order $2$ which acts by reversing the factors.  Specifically, the stabilizer
contains $SL_2^{\alpha_2}, SL_2^{0120},$
and suitable (nonstandard) representatives for the simple reflection $w[3]$ in $W(G,T).$

If $a$ is not square, then the stabilizer of $\psi_V^a$ in $Sp_4^{\{\alpha_2, \alpha_3\}}(F)\cdot 2321^\vee(F^\times)$
also has two components, but
the identity component is isomorphic not
to $SL_2^2,$ but to $\Res_{F(\sqrt{a})/F}SL_2.$
Here $F( \sqrt{a})$ denotes the
unique quadratic extension of $F$ in which $a$ is a square,
and $\Res_{F(\sqrt{a})/F}$ denotes
restriction of scalars.
If we identify $M_1$ with $GSp_6$
as in section \ref{s:MaximalParabolicsOfF4},
then $Sp_4^{\{\alpha_2, \alpha_3\}}(F)\cdot 2321^\vee(F^\times)$
corresponds to all matrices of the
form
$$
\alpha \cdot \bpm 1&&\\&g&\\&&1\epm, \qquad \alpha \in GL_1, \;g \in Sp_4,
$$
while the identity component of the  stabilizer of $\psi_V^a$ consists of all matrices of the form
$$
\bpm 1&&\\&g&\\&&1\epm, \qquad
g = \bpm
g_1& g_2&g_3& g_4 \\
g_2 a & g_1&g_4 a&g_3 \\
g_5& g_6&g_7& g_8 \\
g_6 a & g_5&g_8 a&g_7
\epm, \qquad \bm g_1g_7+ag_2g_8 - g_3g_5-ag_4g_6=1,\\ \\
g_1g_8+g_2g_7-g_3g_6-g_4g_5=0.
\em
$$
A representative for the second connected component
of the stabilizer in this case
is $$\diag(-1,-1,1,1,-1,-1),$$
which acts on the identity component in a manner which
corresponds to the mapping $h \mapsto \, _t\bar h^{-1},$
with $\bar{\phantom{h}}$ being the action of the
nontrivial element of $\Gal(F(\sqrt{a})/F).$

From this discussion, it follows that
$$
I_{w_0, \nu_0}= I_{w_0, \nu_0}^{00}+I_{w_0, \nu_0}^{0}+
I_{w_0, \nu_0}^{1}+
\sum_{{a \in F^\times/F^{\times,2}}\atop{a\ne 1}}
I_{w_0, \nu_0}^a,
$$
where
 $$I_{w_0, \nu_0}^{00}
=
\il_{U_{w_0}(\A)\bs U(\A)}
\il_{(C\cap P_2)(F)\bs C(\A)}
\iiFA5\wt\varphi_{\pi}(g)
\wt \theta_\phi^\psi  (
(0|0,y|0)u_0
 \varpi_3(g) )
 f_\tau^V\left(v(y) w_0u_0 g,s\right)
\, dy
\, dg\, du_0,
$$
$$
f_\tau^V(h,s)= \iiFA5f(v(y')h,s)\, dy',
$$
$$I_{w_0, \nu_0}^0
=
\il_{U_{w_0}(\A)\bs U(\A)}
\il_{P_3^0(F)\bs C(\A)}
\iiFA5\wt\varphi_{\pi}(g)
\wt \theta_\phi^\psi (
(0|0,y|0)u_0
 \varpi_3(g) )
 f_\tau^{(V,\psi_V^0)}\left(v(y) w_0u_0 g,s\right)
\, dy
\, dg\, du_0,
$$
$$\begin{aligned}I_{w_0, \nu_0}^1
=\frac12
\il_{U_{w_0}(\A)\bs U(\A)}
\il_{SL_2^{\alpha_2}(F)SL_2^{0120}(F)(C \cap U_2)(F)\bs C(\A)}&
\wt\varphi_{\pi}(g) \\
\iiFA5&
\wt \theta_\phi^\psi ( (0|0,y|0)u_0\varpi_3(g) )
 f_\tau^{(V,\psi_V^1)}\left(v(y) w_0u_0 g,s\right)
\, dy
\, dg\, du_0,\end{aligned}
$$
$$\begin{aligned}I_{w_0, \nu_0}^a
=
\il_{U_{w_0}(\A)\bs U(\A)}
\il_{S_a(F)(C \cap U_2)(F)\bs C(\A)}&
\wt\varphi_{\pi}(g) \\
\iiFA5&
\wt \theta_\phi^\psi ( (0|0,y|0)u_0\varpi_3(g) )
 f_\tau^{(V,\psi_V^a)}\left(v(y) w_0u_0 g,s\right)
\, dy
\, dg\, du_0,\end{aligned}
$$
Where $S_a$ is the stabilizer of $\psi_V^a,$ described above, and
$$
f_\tau^{(V, \psi_V^i)}(h,s)= \iiFA5f_\tau(v(y')h,s)\psi_V^i(v(y))\, dy',
\qquad (i=0,1,a ).
$$

(The integral $I_{w_0, \nu_0}^1$ can also be
expressed as an integral over
$
U_{w_0}(\A)\bs U(\A)\times S_1(F)(C \cap U_2)(F)\bs C(\A),
$
where $S_1$ is
the full stabilizer of $\psi_V^1$ in $(C \cap P_2).$  In the expression above, the
factor $\frac12$ is present because the group
$SL_2^{\alpha_2}SL_2^{0120}(C \cap U_2)$ is of index
two in this stabilizer.)

Now, it follows from the table in the end of section 4 that $\mathcal{O}(\tau)=(31^4).$
For the next step, we shall use this fact to prove that the integrals $I_{w_0,\nu_0}^{00}$ and $I_{w_0,\nu_0}^0$ both vanish.

\begin{prop}\label{I0and I00 vanish}
Since $\tau$ is attached to the orbit $(31^4),$
$I_{w_0,\nu_0}^{00} = I_{w_0,\nu_0}^0=0.$
\end{prop}
\begin{proof}
We have  $$I_{w_0,\nu_0}^{00}=
\il_{U_{w_0}(\A)\bs U(\A)}
\il_{(C\cap P_2)(F)\bs C(\A)}
\wt\varphi_{\pi}(g)
\tilde \theta^\psi_{\phi|_2} (
u_0
 \varpi_3(g) ) f_\tau\left(v(y) w_0u_0 g,s\right)\, dy
\, dg\, du_0.
$$
The parabolic subgroup $C \cap P_2$ of $C \cong Sp_6$
maps into the standard parabolic subgroup of $Sp_{14}$ whose
Levi subgroup is isomorphic to $GL_5 \times Sp_4,$ with the unipotent radical of $C\cap P_2$ mapping into the unipotent radical of this parabolic subgroup in $Sp_{14}.$  The function $  \tilde \theta^\psi_{\phi|_2}$ is invariant by the
unipotent radical of the parabolic subgroup of $Sp_{14}(\A)$
(which lifts into $\Sp_{14}(\A)$) and so is $f_\tau\left(w_0u_0 g,s\right),$ so this term vanishes by cuspidality
of $\wt \vph_\pi.$

Next consider the integral $I_{w_0,\nu_0}^{0}$.
Changing variables, and using the formulas for the action of the
Weil representation, we have$$
\begin{aligned}
\iiFA 5
\wt\theta_\phi^\psi((0|0y|0)u_0\varpi_3(g) )
 f_\tau^{(V, \psi_V^0)}&\left(v(y) w_0u_0 g,s\right)\, dy\\
 &=f_\tau^{(V, \psi_V^0)}\left( w_0u_0 g,s\right)\iiFA 5
\wt\theta_\phi^\psi((0|0y|0)u_0\varpi_3(g) )\psi(y_1)
 \, dy\\
&=f_\tau^{(V, \psi_V^0)}\left( w_0u_0 g,s\right)
\sum_{\xi \in F^2} [\omega_\psi
( u_0\varpi_3(g) )\phi](0,0,0,0,-1,\xi).
\end{aligned}
$$
Hence, the integral
$I_{w_0,\nu_0}^{0}$ is equal to$$
\il_{U_{w_0}(\A)\bs U(\A)}
\il_{(C\cap P_2)(F)\bs C(\A)}
\wt\varphi_{\pi}(g)
\sum_{\xi \in F^2} [\omega_\psi
( u_0\varpi_3(g) )\phi](0,0,0,0,-1,\xi) f_\tau^{(V, \psi_{V}^0)}\left( w_0u_0 g,s\right)
\, dg\, du_0.
$$
Next, let $N_1$ denote the unipotent
radical of $(C\cap P_2)$ and
factor the integration over this group.
Since $w_0 N_1 w_0^{-1} \subset U_4,$ the function
$f_\tau^{(V, \psi_{V}^0)}$ is invariant by $N_1(\A).$

Now consider the partial theta function $
\wt\theta_1(h):=
\sum_{\xi \in F^2} [\omega_\psi
(h )\phi](0,0,0,0,-1,\xi),$
with $h \in \Sp_{14}(\A).$
We claim that the function
 $n_1 \mapsto \wt\theta_1( \varpi_3(n_1)h)$ with $n_1 \in N_1(\A), h \in \Sp_{14}(\A),$
 depends only on the inner $6\times 6$ block
  of $\varpi_3(n_1).$
  To see this,
  let $Q_4$ denote the standard parabolic subgroup  $Sp_{14}$ with Levi factor isomorphic to $GL_4 \times Sp_{6}.$
  Note that $\varpi_3(n_1)$
  is
upper triangular, and hence contained in $Q_4.$
Now, the function $\wt\theta_1$ is invariant on the left by the unipotent
radical of $Q_4,$ and  satisfies $\wt\theta_1( \diag( g, I_6,\, _tg^{-1}) h)
= \gm_{\det g} \wt\theta_1(h)$ for all $g \in GL_4(\A)$ and
$h \in \Sp_{14}(\A).$
This follows from the same argument used to prove lemma \ref{Lem: Reduction to theta function on smaller symplectic group-- partial theta}.
The upper left $4\times 4$ block of $\varpi_3(n_1)$ will be unipotent,
and hence have determinant one.  It follows that
the function
 $n_1 \mapsto \wt\theta_1( \varpi_3(n_1)h)$ with $n_1 \in N_1(\A), h \in \Sp_{14}(\A),$
 depends only on the inner $6\times 6$ block
  of $\varpi_3(n_1),$ as claimed.

Each element of $N_1(\A)$ can be written uniquely
as
 $$
n_1(r)=x_{0001}(r_{0001})x_{0011}(r_{0011})x_{0111}(r_{0111})x_{0121}(r_{0121})x_{0122}(r_{0122}),$$
and the projection of $\varpi_3(n_1(r))$ as above onto
$Sp_6$ is
$$
\bpm 1&0&0&-2r_{0001}&-2r_{0011}&0\\0&1&0&0&0&-2r_{0011}\\0&0&1&0&0&-2r_{0001}\\0&0&0&1&0&0\\0&0&0&0&1&0\\0&0&0&0&0&1\epm.$$
It follows that
$$[\omega_\psi(u_0 \varpi_3(n_1(r)g))\phi](0,0,0,0,-1,\xi_6,\xi_7)
=\psi(2 \xi_7 r_{0001} + 2 \xi_6 r_{0011})
[\omega_\psi(u_0 \varpi_3(g))\phi](0,0,0,0,-1,\xi_6,\xi_7).
$$
First, factor the integration over
$N_2:=U_{0111}U_{0121}U_{0122}.$  Thus $I_{w_0, \nu_0}^0$
is equal to
$$\il_{U_{w_0}(\A)\bs U(\A)}
\il_{(C\cap P_2)(F)N_2(\A)\bs C(\A)}
\wt\varphi_{\pi}^{N_2}(g)
\sum_{\xi \in F^2} [\omega_\psi
( u_0\varpi_3(g) )\phi](0,0,0,0,-1,\xi) f_\tau^{(V, \psi_{V}^0)}\left( w_0u_0 g,s\right)
\, dg\, du_0,$$
where
$$\wt\varphi_{\pi}^{N_2}(g)=\iq{N_2}
\wt\varphi_{\pi}(n_2g)\, dn_2.$$
Now expand $\wt\varphi_{\pi}^{N_2}(g)$ along
$U_{0001}U_{0011}.$
The term corresponding to the trivial character
produces a constant term of $\wt\varphi_{\pi},$
and therefore vanishes by cuspidality.
The group $SL_2^{\alpha_3}$
permutes the nontrivial characters of this group transitively.
Choose $x_{0001}(r_{0001})x_{0011}(r_{0011})
\mapsto \psi(r_{0001})$ as a representative.
Then the stabilizer is $U_{0010}.$
Now factoring the integration over $\quo{U_{0001}U_{0011}}$
picks off a single term in the sum over $\xi,$
 and we have
 $$\begin{aligned}
 \il_{U_{w_0}(\A)\bs U(\A)}
\il_{(C\cap P_2)(F)N_1(\A)\bs C(\A)}
&\wt\varphi_{\pi}^{(N_1,\psi_{N_1})}(g) f_\tau^{(V, \psi_{V}^0)}\left( w_0u_0 g,s\right)
\\
& [\omega_\psi
( u_0\varpi_3(g) )\phi](0,0,0,0,-1,0,-1/2)\, dg\, du_0,
\end{aligned}$$
where $N_1,$ as above,
is the unipotent radical of $C \cap P_2$ and can also be
described as
$U_{0001}U_{0011}N_2,$
and
$$\wt\varphi_{\pi}^{(N_1,\psi_{N_1})}(g)
=\iq{N_1}\wt\varphi_{\pi}(ng) \, \psi_{N_1}(n)\, dn
\qquad \psi_{N_1}(n) = \psi( n_{0001}), \; n \in N_1(\A).
$$
Next, the projection of
$\varpi_3( x_{0010}(r_1) x_{0110}(r_2) x_{0120}(r_3))$
onto $Sp_6$ is given by
$$\bpm 1&0&0&0&0&0\\0&1&r_1&2r_2&2(r_1r_2-r_3)&0\\0&0&1&0&2r_2&0\\0&0&0&1&-r_1&0\\0&0&0&0&1&0\\0&0&0&0&0&1\epm.$$
It follows that
$$
[\omega_\psi
( u_0\varpi_3(g) )\phi](0,0,0,0,-1,0,-1)
$$
is left-invariant by $U_{0010}U_{0110}U_{0120}.$
To show that $I_{w_0, \nu_0}^0=0,$ we will prove the following.

\begin{prop}
\label{invariance of ftau}
Since $\tau$ is attached to $(31^4),$ the function
 $$f_\tau^{(V, \psi_V^0)}(g,s):=\iq V f_\tau(vg,s) \psi_V^0(v) \, dv$$
 is left-invariant by $U_{0100}U_{0110}U_{0120}(\A).$
 \end{prop}

 Assuming proposition \ref{invariance of ftau},
 the proof of proposition \ref{I0and I00 vanish} is as follows:
 we obtain as inner integration the constant term
$\widetilde{\varphi}_\pi^L$ where $L$ is the unipotent radical of
the maximal standard parabolic subgroup of $Sp_6$ whose Levi part is
$GL_2\times SL_2$. By cuspidality this is zero.

 \begin{proof}[Proof of proposition  \ref{invariance of ftau}]
 Put $N_3=U_{0100}U_{0110}U_{0120}.$
 Expand $f_\tau^{(V, \psi_V^0)}(g,s)$ along $\quo{N_3}.$
 Under the action of $SL_3^{\alpha_3},$ there are infinitely many orbits of characters, represented by
 $$
 n_3(r_1, r_2, r_3) \mapsto 1, \quad
 n_3(r_1, r_2, r_3) \mapsto \psi(r_3),\text{ and } \quad
 n_3(r_1, r_2, r_3) \mapsto \psi(r_1+ar_3), a \in F^\times
 $$
 respectively, where $n_3(r_1, r_2, r_3)=x_{0100}(r_1) x_{0110}(r_2) x_{0120}(r_3).$
 Now,  for any $a \in F^\times,$
 \begin{equation}\label{F4A1P4 Proposition FC for 511}
 \iiFA 3 f_\tau^{(V, \psi_V^0)} \left(n_3(r_1, r_2, r_3) g,\, s\right)\psi(r_1+ar_3)\,dr
 \end{equation}
 is a Fourier coefficient attached to the orbit
 $(51^2)$ which is greater than $(31^4).$
 Hence integral \eqref{F4A1P4 Proposition FC for 511}
is zero.

In order to prove that
 \begin{equation}\label{F4A1P4 Proposition -- Integral we need to relate to the FC 331}
 \iiFA 3 f_\tau^{(V, \psi_V^0)} \left(n_3(r_1, r_2, r_3) g, \, s\right)\psi(r_3)\,dr
 \end{equation}
 vanishes, we shall relate it to  Fourier coefficients
 attached to the unipotent orbit $(3^21),$ which is also
 greater than $(31^4).$
 Let $N_4$ denote the unipotent subgroup of $M \cong \GSpin_7$ which projects to
 $$
 \left\{ \pr(n_4)=\bpm
1&0&r_1&r_2&r_3&*&*\\
&1&r_4&r_5&r_6&*&*\\
&&1&0&0&-r_6&-r_3\\
&&&1&0&-r_5&-r_2\\
&&&&1&-r_4&-r_1\\
&&&&&1&0\\
&&&&&&1\\
 \epm \right\}
 \subset SO_7.$$
Since this projection is an isomorphism on $N_4,$ we can use the matrix entries above to define characters of $N_4.$
An integral
$$
\iq{N_4} f_\tau(n_4 g,\,s) \psi\left( \sum_{i=1}^6 {a_i} r_i \right) \, dn_4
$$
is a Fourier coefficient attached to $(3^21)$ if and only if the character $n_4 \mapsto \psi\left( \sum_{i=1}^6 {a_i} r_i \right)$ is in general position, which is the case if and only if the matrix
$\bspm
a_1&a_2&a_3\\
a_4&a_5&a_6\espm$ has rank $2.$
Now let $N_4^0$ denote the subgroup of $N_4$ defined by setting $r_4=r_5=0.$  Use $r_1, r_2, r_3, r_6$ as
coordinates on $N_4^0$ in the obvious way.  We claim that
\begin{equation}\label{F4A1P3 UniPer on N40}
\iq{N_4^0} f_\tau\left(n_4^0 g,\,s\right) \psi\left( a_1r_1+a_6r_6 \right) \, dn_4^0
\end{equation}
vanishes identically whenever $a_1a_6\ne 0.$  Indeed, this integral is left-invariant by the $F$-points of two dimensional subgroup $U_{0010}U_{0110}$ corresponding to the omitted variables $r_4$ and $r_5.$  One may perform a Fourier expansion along this group, and every term will produce a Fourier coefficient of $f_\tau$ attached to $(3^21).$
Thus integral \eqref{F4A1P3 UniPer on N40} is zero for all choice of data, whenever $a_1a_6\ne 0.$

In \eqref{F4A1P3 UniPer on N40}, use the left invariance property of
$f_\tau$ to obtain $f_\tau(n_4^0g,s)=f_\tau(w[2]n_4^0g,s)$.
Conjugating $w[2]$ to the right, we obtain an integral
over $w[2]\quo{N_4^0}w[2]^{-1}.$
This integral is clearly zero
for all choice of data, and, for suitable
nonzero $a_1,a_6,$ is also an inner integration to integral
\eqref{F4A1P4 Proposition -- Integral we need to relate to the FC 331}. Thus we conclude that integral \eqref{F4A1P4 Proposition -- Integral we need to relate to the FC 331} is zero for all choice of
data.
 \end{proof}
\end{proof}
 We now treat $I_{w_0,\nu_0}^1.$
Let
$$\psi_{U^{w_0}}^1((0|0,0,y|0))= \psi(-y_3), \qquad y = (y_1, y_2,y_3,y_4,y_5) \in \A^5$$
be the unique character of $U^{w_0}(\A)$ such that
$\psi_{U^{w_0}}^1(u) \psi_V^1(v(w_0uw_0^{-1}))$is trivial.

  Using the action of the Weil
representation, we obtain that
$$\int\limits_{(F\backslash {\bf
A})^5}\widetilde{\theta}_\phi^\psi((0|0,y|,0) u \wt h)\psi_{U^{w_0}}^1(y)dy=
\sum_{\xi \in F^2} \omega_\psi(u\wt h) \phi(0,0,-1,0,0,\xi),
$$ for all  $\wt h \in \Sp_{14}(\A), u \in \cH_{14}(\A).$
 Denote the left hand side by
$\wt\theta_\phi^{(U^{w_0}, \psi_{U^{w_0}}^1)}(u\wt h).$

Changing variables
in $f_\tau ^{V,\psi_V^1}$  then proves that
$$\begin{aligned}I_{w_0, \nu_0}^1
=\frac12
\il_{U_{w_0}(\A)\bs U(\A)}\,&
\il_{SL_2^{\alpha_2}(F)SL_2^{0120}(F)(C \cap U_2)(F)\bs C(\A)}
\wt\varphi_{\pi}(g) \\
&\qquad \qquad \qquad \qquad
\wt\theta_\phi^{(U^{w_0}, \psi_{U^{w_0}}^1)} (u_0\varpi_3(g) )
 f_\tau^{(V,\psi_V^1)}\left(w_0u_0 g,s\right)
\, dg\, du_0.\end{aligned}
$$

Next we factor the integration over $N_1:=U_{0001}U_{0011}U_{0111}U_{0121}U_{0122}.$
This group is the unipotent radical of a standard parabolic subgroup of $C \cong Sp_6,$ having Levi subgroup isomorphic to $GL_1 \times Sp_4.$  Further, it is isomorphic to the Heisenberg group $\cH_5.$

Set
 $$
n_1'(r)=x_{0111}(r_{0111})x_{0121}(r_{0121})x_{0122}(r_{0122}),\;  n_1''(r) = x_{0001}(r_{0001})x_{0011}(r_{0011}).$$
Then one may compute $\varpi_3( n_1'(r))$
and  $\varpi_3( n_1''(r)).$
Each will be of the form $\bspm u&b&c \\ & h & b'\\ && _tu^{-1} \espm,$
with $u \in SL_2, h \in Sp_{10}, b \in \Mat_{2 \times 10},$ etc.,
and it follows from  lemma \ref{Lem: Reduction to theta function on smaller symplectic group-- partial theta}, that  we only need the middle $10\times 10$
block, $h.$
The projection of $\varpi_3(n_1'(r))$ onto $Sp_{10}$ is
$$
\bpm I_5&X'(r)\\&I_5 \epm, \qquad
X'(r) = \bpm2 r_{0111}& 2 r_{0121}& 0& 0& -2 r_{0122}\\0& -2 r_{0111}& 0& 0&
  0\\0& 0& 0& 0& 0\\0& 0& 0& -2 r_{0111}& 2 r_{0121}\\0& 0& 0&
  0& 2 r_{0111}\epm,
$$
whence, by the action of the Weil
representation,
$$
[\omega_\psi( \varpi_3( n_1'(r)))\phi](0,0,-1,0,0,\xi_1,\xi_2)
=\psi(-r_{0122}-2\xi_1r_{0121}-2\xi_2r_{0111})
\phi(0,0,-1,0,0,\xi_1,\xi_2).
$$
The projection of $\varpi_3(n_1''(r))$ onto $Sp_{10}$ is
$$
\bpm I_5&X_1''(r)\\&I_5 \epm
\bpm h_1''(r)&\\&_th_1''(r)^{-1} \epm, \qquad
h_1''(r):=\bpm 1& 0& 0& r_{0001}& r_{0011}\\0& 1& 0& 0& r_{0001}\\0& 0& 1&
  0& 0\\0& 0& 0& 1& 0\\0& 0& 0& 0& 1\epm,$$
$$X_1''(r):=\bpm 0& 0& -2 r_{0001} r_{0011}& 0& 0\\0& 0& -r_{0001}^2& 0&
  0\\-2 r_{0001}& -2 r_{0011}&
  0& -r_{0001}^2&- 2 r_{0001} r_{0011}\\0& 0& -2 r_{0011}& 0& 0\\0&
   0& -2 r_{0001}& 0& 0\epm,
$$
whence
$$
[\omega_\psi( \varpi_3( n_1''(r)))\phi](0,0,-1,0,0,\xi_1,\xi_2)
=
\phi(0,0,-1,0,0,\xi_1-r_{0001},\xi_2-r_{0011}).
$$
Define $\phi_0 \in S(\A^2)$
by $\phi_0(x_1,x_2) = \phi(0,0,-1,0,0,x_1,x_2).$
Then the above shows that
\begin{equation}\label{F4A1P4 theta to theta on Sp4}
[\omega_\psi\circ \varpi_3(n_1)]\phi(0,0,-1,0,0,x_1,x_2) = [\omega_{\psi_{-2}}^{(4)}\circ \iota(n_1)]\phi_0(x_1,x_2).
\end{equation}
for a suitable isomorphism $\iota:N_1 \to \cH_{5}.$
Here $\omega^{(4)}_{\psi_{-2}}$ denotes the Weil representation of $\cH_5(\A) \rtimes \Sp_{4}(\A)$
defined using the additive character $\psi_{-2}(x) = \psi(-2x)$
in place of $\psi.$
By uniqueness of the extension of $\omega_{\psi_{-2}}^{(4)}$
from $ \cH_{5}$ to $\cH_5(\A) \rtimes \Sp_{4}(\A),$
it follows that \eqref{F4A1P4 theta to theta on Sp4}
 holds for $g \in \Sp_4(\A)$ (i.e., the preimage
 in $\wt C(\A)$ of $Sp_4^{\{\alpha_2, \alpha_3\}}(\A)$) as well.

Now, the identity component of the stabilizer of $\psi_{U^{w_0}}^1$ in
$Sp_4^{\{\alpha_2, \alpha_3\}}$ is $SL_2^{\alpha_2} SL_2^{0120}.$  In order to complete our analysis
of this case, we show that $f_\tau^{(U^{w_0}, \psi^1_{U^{w_0}})}$ is left-invariant by the $\A$ points of this group.
To prove that, and also the
analogous statement related to
the integrals $I_{w_0,\nu_0}^a,$ we  prove the
following.

 \begin{prop}
 Since $\tau$ is attached to $(31^4),$
 the function
 $$f_\tau^{(V, \psi_V^a)}(g, s)
 =\iq{V} f_\tau(vg,s) \psi_V(v)\, dv
 $$
 satisfies
 $$
 f_\tau^{(V, \psi_V^a)}( hg,s) = f_\tau^{(V, \psi_V^a)}(g,s) \qquad( \forall g \in C(\A), \text{ and } h \in S_a(\A)),
 $$
 for all $a \ne 0,$
 where $S_a$ is the stabilizer of $\psi_V^a$ in $Sp_4^{\{\alpha_2, \alpha_3\}}.$
 \end{prop}
 \begin{proof}
The group $S_a(\A)$ is generated by $S_a(F)$ and the two dimensional maximal
unipotent subgroup $L(\A):=S_a(\A)\cap U_{0100}U_{0110}U_{0120}(\A),$
so it suffices to prove that $f_\tau^{(V, \psi_V^a)}$ is invariant
by this unipotent group.
Clearly, we may expand $f_\tau^{(V, \psi_V^a)}$
along $\quo{L}.$  We claim that every term other than the constant
term vanishes.

Recall that, by hypothesis, $\tau$ is
attached to the unipotent orbit $(31^4),$ and thus does not
support any Fourier coefficient which is attached to the
orbit $(3^21).$  The set of such  Fourier coefficients was described before
equation \eqref{F4A1P3 UniPer on N40}.

Let $V_1$ denote the product  of all groups $U_{\alpha} \subset V$
except for $U_{1000}.$  Let $N_4$ denote the unipotent radical of the
maximal parabolic subgroup of $\GSpin_7$ with Levi isomorphic
to $GL_2\times \GSpin_3,$ and
$R_1$ denote the product of $V_1$
and $L.$  Then $R_1$ is a codimension $1$ subgroup of $N_4.$
To be precise:
there is a column vector $c_a=(1,0,-a)$ or $(0,1,0),$ such that
$$\psi_{V}^a \bpm 1&x&*\\&I_5&-_t x \\&&1\epm
= \psi\left(\bpm x_2&x_3&x_4 \epm \cdot c_a \right),$$
and
$$
R_1 = \left\{ \bpm I_2&X&*\\&I_3&-_tX\\ &&I_2 \epm :
X= \bpm x_1&x_2&x_3\\x_4&x_5&x_6 \epm,
\; \bpm x_4&x_5&x_6 \epm \cdot c_a=0\right\}.
$$
Here we identify unipotent elements of $M_4$ with their images in
$SO_7.$
This follows from the fact that $I_7 + re'_{12} \in V$ and
$L$ preserves $\psi_V^a.$  Now, let $\psi_{R_1}$ be
any character such that $\left.\psi_{R_1} \right|_{V_1(\A)}=\left.\psi_{V} \right|_{V_1(\A)},$ and $\left.\psi_{R_1}\right|_{L(\A)}$ is nontrivial.
We claim that
$$
f_\tau\FC{R_1}:=\iq{R_1} f_\tau(r_1 g,s)\psi_{R_1}(r_1) \, dr_1.
$$vanishes identically.
Indeed, one may expand $
f_\tau\FC{R_1}$ on $R_1(\A)N_4(F) \bs N_4(\A)$ and obtain an
expression as a sum of terms, each of which is an integral
on $\quo{N_4}$ against a character whose restriction to $R_1$ is $\psi_{R_1}.$  But it is clear from the remarks above that every character of $N_4$
which restricts to $\psi_{R_1}$ is in general position.  As explained
before equation \eqref{F4A1P3 UniPer on N40}, such an integral
is a Fourier coefficient attached to $(3^21),$ and therefore vanishes
identically.
  \end{proof}
  It follows that with the notation as above
  $$\begin{aligned}
  I_{w_0, \nu_0}^1
  = \frac12\il_{U^{w_0}(\A)\bs U(\A)}
  \il_{N_1(F)SL_2^{\alpha_2}(F) SL_2^{0120}(F)\bs C(\A)}
 & \wt\vph_\pi(g) f_\tau\FC V(u_0g,s)\\
&  \sum_{\xi \in F^2}
 [\omega_\psi(u_0 \varpi_3(g))\phi](0,0,-1,0,0,\xi_6, \xi_7)
  \, dg
  \, du_0.\end{aligned}
  $$
 The inner period in this case is
 $$
 \iq{\cH_5} \iq{SL_2}\iq{SL_2}
 \wt\vph_\pi \left(u \bpm 1&&\\&\alpha(g_1, g_2)&\\&&1 \epm \right)
 \theta_\phi^{\psi_2} (u\alpha(g_1, g_2))
 \,dg_1\,dg_2
 \,du,
 $$
 where $ \wt\vph_\pi $ is a genuine cusp form
 $\Sp_6(\A) \to \C,$ $\theta_\phi^{\psi_2}$ is a
theta function on $\cH_5(\A) \rtimes \Sp_4(\A)$
defined using the character $\psi_2(x) = \psi(2x)$
and we have identified
$\cH_5$ with
$$
\left\{ \bpm1&X&y\\&I_4&X' \\&&1 \epm \right\} \subset Sp_6.
$$
Also
$\alpha: SL_2\times SL_2 \to Sp_4$ is given by
$$
\alpha\left(\bpm a& b \\ c& d \epm , g_2\right)
= \bpm a& &b\\&g_2&\\c&&d\epm.
$$

The analysis of $I_{w_0, \nu_0}^a$ nonsquare is similar.
Let
$$\psi_{U^{w_0}}^a((0|0,0,y|0))= \psi(-y_2+ay_4)$$
be the unique character of $U^{w_0}(\A)$ such that
$\psi_{U^{w_0}}^1(u) \psi_V^1(v(w_0uw_0^{-1}))$ is trivial.
Then the integral over $\quo{V}$ picks off a partial theta series
$$
\theta_\phi^{(U^{w_0}, \psi_{U^{w_0}}^1)}(h) :=
\sum_{\xi \in F^2} \omega_\psi(h) \phi(0,-1,0,a,0,\xi).
$$

With notation as before, similar computations show that
$$
[\omega_\psi( \varpi_3( n_1'(r)))\phi](0,-1,0,a,0,\xi_1,\xi_2)
=\psi(-ar_{0122}+2a\xi_1r_{0111}-2\xi_2r_{01s1})
\phi(0,-1,0,a,0,\xi_1,\xi_2),
$$
whence
$$
[\omega_\psi( \varpi_3( n_1''(r)))\phi](0,-1,0,a,0,\xi_1,\xi_2)
=
\phi(0,-1,0,a,0,\xi_1+r_{0011},\xi_2-ar_{0011}).
$$
This is still essentially the
action of the Weil representation of $\cH(\A)\rtimes \Sp_4(\A).$
Determined by the character $\psi_{-2a}(x) = \psi( -2ax).$

We conclude that this case does not appear to be of open orbit
type.

\section{Appendix}
The purpose of this appendix is to elaborate slightly on the
earlier remark that the Fourier coefficient $\varphi\FC{U_\O}$
of an automorphic form $\varphi$ will not, in general, be
an automorphic form, because it will not, in general, be $\frak z$-finite.
This should be something that one can prove using purely local
methods, but here it is more convenient to point out that the
existence of an integral representation which involves some Fourier
coefficient tends to suggest that the Fourier coefficient is far from
$\mathfrak{z}$-finite.  Indeed, it is, in a sense, spread out all over
the spectrum of the operators in $\mathfrak{z}.$

To illustrate this, consider the integral representation for the exterior
square $L$ function of a cuspidal automorphic representation of $GL_{2n}(\A)$, which was
given in \cite{Jacquet-Shalika}.  This construction
consists of integrating $\varphi\FC{U_\O}$ against an
Eisenstein series $E(g,s)$ defined on the group $GL_n,$ with $\varphi$
itself being a cuspform in the space of an irreducible cuspidal
automorphic representation of the group $GL_{2n}.$  The function
$\varphi\FC{U_\O}$ inherits the rapid decay of the cusp form $\vph.$
Hence the integral
$$
\il_{Z(\A)\quo{GL_n}} \varphi\FC{U_\O}(g) E(g,s)\, dg
$$
is absolutely convergent for all $s$ where $E(g,s)$ has no poles,
uniformly for $s$ in a compact set.
Reversing order and making a change of variable,
$$\begin{aligned}
\il_{Z(\A)\quo{GL_n}} \lim_{t\to 0}&\varphi\FC{U_\O}(g\exp (tX)) E(g,s)\, dg
\\& =\il_{Z(\A)\quo{GL_n}} \varphi\FC{U_\O}(g) \lim_{t\to 0}E(g\exp (tX),s)\, dg
\end{aligned}
$$
for all $X \in \frak{gl}(2, \R).$  Consequently, even though $E(g,s)$
is not $L^2,$ one may still think if this integral as a pairing $\la, \ra$ with
respect to which the elements of the center of the universal
enveloping algebra are ``self-adjoint.''
Now for concreteness, assume $n=2$ and $F=\Q,$  and take $\Delta$
the Laplace-Beltrami operator.
Then $\Delta E(g,s) = s(1-s)$ for all $s \in \C.$
Suppose that $\vph\FC{U_\O}$ is
$\mathfrak z$-finite.  Then it is killed by a polynomial $P(\Delta)$ in $\Delta.$
But then
$$0 = \la P(\Delta) \vph, E( \cdot, s) \ra
= \la \vph,  P(\Delta) E( \cdot, s) \ra
=P( s(1-s) ) \la \vph, E( \cdot, s)\ra
$$
for all values of $s.$ Since $P(s(1-s))$ has only finitely many zeros, and
$ \la \vph, E( \cdot, s)\ra$ is nonzero for $\Re(s)$ large, we have a contradiction.

\end{document}